\documentclass{amsart}
\usepackage{graphicx}
\usepackage{hyperref}
\usepackage{latexsym}
\usepackage{amsmath}
\usepackage{amsthm}
\usepackage{amssymb}
\usepackage{graphicx}
\usepackage{pdfpages}
\usepackage{tikz}
\usetikzlibrary{calc,trees,positioning,arrows,chains,shapes.geometric,%
    decorations.pathreplacing,decorations.pathmorphing,shapes,%
    matrix,shapes.symbols}
\usepackage[stable]{footmisc}
\usepackage[labelformat=empty]{caption}

\newtheorem{theorem}{Theorem}[section]
\newtheorem{lemma}[theorem]{Lemma}
\newtheorem{cor}[theorem]{Corollary}
\newtheorem{prop}[theorem]{Proposition}

\theoremstyle{definition}

\theoremstyle{remark}
\newtheorem{remark}[theorem]{Remark}

\newcommand{\vip}{\vskip.2cm}
\newcommand{\R}{\mathbb{R}}
\newcommand{\E}{\mathbb{E}}
\newcommand{\Et}{\E_\theta}

\newcommand{\Vart}{\mathbb{V}{\rm ar}_\theta\,}

\newcommand{\Covt}{\mathbb{C}{\rm ov}_\theta\,}
\newcommand{\indiq}{{{\mathbf 1}}}
\newcommand{\intot}{{\int_0^t}}

\newcommand{\bl}{\boldsymbol{\ell}}

\newcommand{\bX}{{\mathbf X}}

\newcommand{\bx}{\boldsymbol{x}}

\newcommand{\omg}{\indiq_{\Omega_{N,K}}}
\newcommand{\e}{{\varepsilon}}

\newcommand{\cA}{{\mathcal A}}

\newcommand{\cV}{{\mathcal V}}
\newcommand{\cW}{{\mathcal W}}

\newcommand{\cX}{{\mathcal X}}

\newcommand{\cY}{{\mathcal Y}}

 \makeatletter
\@namedef{subjclassname@2020}{ \textup{2020} Mathematics Subject Classification}
\makeatother

\hypersetup{pdfmenubar=true}

\begin{document}

\title[Central limit theorem for Hawkes processes]
{Central limit theorem  for a partially observed interacting system of Hawkes processes I:  subcritical case}

\author{ Chenguang Liu, Liping Xu, An Zhang}

\address{School of mathematical sciences, Beihang University, PR China}
\email{xuliping.p6@gmail.com}
\address{School of mathematical sciences, Beihang University, PR China}
\email{anzhang@buaa.edu.cn}

\address{Delft Institute of Applied Mathematics, EEMCS, TU Delft, 2628 Delft, The Netherlands}
\email{LIUCG92@gmail.com}

\begin{abstract}
We consider a system of  $N$ Hawkes processes and observe the actions of a subpopulation of size 
 $K \le N$ up to time $t$, where $K$ is large. The influence relationships between each pair of individuals are modeled by i.i.d.Bernoulli($p$) random variables, where $p \in [0,1]$ is an unknown parameter. Each individual acts at a {\it baseline} rate $\mu > 0$ and, additionally, at an {\it excitation} rate of the form $N^{-1} \sum_{j=1}^{N} \theta_{ij} \int_{0}^{t} \phi(t-s)\,dZ_s^{j,N}$, which depends on the past actions of all individuals that influence it, scaled by $N^{-1}$ (i.e. the mean-field type), with the influence of older actions discounted through a memory kernel $\phi \colon \mathbb{R}{+} \to \mathbb{R}{+}$. Here, $\mu$ and $\phi$ are treated as nuisance parameters. The aim of this paper is to establish a central limit theorem for the estimator of $p$ proposed in \cite{D}, under the subcritical condition $\Lambda p < 1$.

\end{abstract}


\maketitle

\tableofcontents
 
\label{key}\section{Introduction} 
Hawkes processes, originally introduced by Hawkes \cite{C} in 1971, have been widely applied across various fields such as neuroscience, finance, social network interactions, and criminology, among others (see, e.g., \cite{bdhm, Q, F, T, LXF, S,  PRBB, ASMI,WZC, zhzh} for a non-exhaustive list). From a mathematical perspective, a substantial body of theoretical literature has been devoted to Hawkes processes and their generalizations (see, e.g., \cite{bdhm2, R,O,  MR3718099, CMR2407, M, MR4512148,  KC2409, LXZ2024, MR4092763, XZ2025} for a non-exhaustive list).

\vip
Regarding statistical inference for Hawkes processes, most studies have focused on the fixed finite-dimensional case (i.e., fixed $N$) with the asymptotics $t\to\infty$. For parametric models, Ogata \cite{o1} investigated the maximum likelihood estimator for stationary point processes. In Bacry-Muzzy  \cite{bmu2}, Delattre et al. \cite{MR3449317}, Hansen et al. \cite{hrr},  Reynaud-Bouret et al. \cite{rrgt,PRBB, rs}, the non-parametric estimation are considered for the following system: for fixed $N\ge 1,$ and $i,j=1,...,N$, the  counting process $(Z_{s}^{i,N})_{i=1...N,0\le s\le t}$ is governed by its intensity process $(\lambda_{s}^{i,N})_{i=1...N,0\le s\le t},$ defined by
\begin{align}\label{BMuy}
\lambda_{t}^{i,N}:=\mu_i+\sum_{j=1}^{N}\int_{0}^{t-}\phi_{ij}(t-s)dZ_{s}^{j,N},
\end{align}
for $\mu_i>0$ and $\phi_{ij}:[0,\infty)\to [0,\infty)$ is measurable and locally integrable. They provided estimators for the $\mu_i$ and the functions $\phi_{ij}.$ In \cite{MR3085883}, Rasmussen considered the Bayesian inference of one dimensional system: the  counting process $(Z_{s})_{0\le s\le t}$ is determined by its intensity process $(\lambda_{s})_{0\le s\le t}$ of the form 
\begin{align}\label{rms}
\lambda_{t}:=\mu_t+\int_{0}^{t-}\phi(t-s)dZ_{s},
\end{align}
where the rate  $\mu_t$ depends on time $t.$

 \vip
 
 In real-world applications, however, it is often necessary to investigate interactions among a large number of measured components within a system. For example, in neuroscience, the number of neurons involved is typically enormous. Therefore, it is natural to consider a double asymptotic case where both  $t\to \infty$ and $N\to \infty$. Research in this setting remains scarce. In \cite{A}, Delattre and Fournier examined a graphical model comprising $N$ Hawkes point processes with pairwise interactions occurring with probability $p$. They proposed an estimator for $p$ based on observing the entire system $(Z_{s}^{i,N})_{i=1...N,0\le s\le t}$ and gave the explicit rate $N^{-1/2}+N^{1/2}m_t^{-1}$ (up to some arbitrarily small loss), where $m_t$ denotes the mean number of events per point process.  Subsequently, Liu \cite{D} studied the same problem of estimating $p$  in the same setting as \cite{A} but using only partial information, specifically, the information obtained  from  $K$ Hawkes processes where $K\le N$.  The author established that under  \eqref{H(q)}  for some $q>3$, the estimator $\hat p_{N,K,t}$ for $p$ with a rate of convergence $K^{-1/2}+N/(K^{1/2}m_t)+N/(Km_t^{1/2})$. More recently, Chevallier, L\"{o}cherbach and Ost \cite{eva} investigated a system of $N$ interacting $\{0,1\}$-valued chains (rather than Hawkes point processes) with binary interactions occurring with unknown probability $p$ on an underlying Erd\"{o}s-R\'{e}nyi random graph. By analyzing coalescing random walks that define a backward regeneration representation of the system, they demonstrated that the unknown connection probability $p$ can be estimated by an computationally efficient estimator with a convergence rate  $N^{-1/2}+N^{1/2}t+(\log(t)/t)^{1/2}$.  Meanwhile, Chevallier and Ost \cite{Chevallier2024} considered the problem of estimating the sets $\mathcal{P}+$ and $\mathcal{P}-$ without prior knowledge of the remaining model parameters, in the same setting as \cite{eva}.
 
 \subsection{Setting}
 We consider some  unknown parameters $p$ $\in[0,1], \mu >0$ and a measurable, locally integrable function $\phi:[0,\infty)\to [0,\infty)$. For  $N\ge 1$,  let $(\Pi^{i}(dt,dz))_{i=1,...,N}$  be an i.i.d. family of
Poisson random measures on $[0,\infty)\times [0,\infty)$ with intensity $dtdz$. Independent of this family, let $(\theta_{ij})_{i,j=1,...,N}$ be an i.i.d. family of  Bernoulli($p$) random variables. We study the following system: for each $i\in \{1,...,N\}$ and all $t\geq 0,$
\begin{align}\label{sssy}
Z_{t}^{i,N}=\int^{t}_{0}\int^{\infty}_{0}\boldsymbol{1}_{\{z\le\lambda_{s}^{i,N}\}}\Pi^{i}(ds,dz), \hbox{ where }
\lambda_{t}^{i,N}=\mu+\frac{1}{N}\sum_{j=1}^{N}\theta_{ij}\int_{0}^{t-}\phi(t-s)dZ_{s}^{j,N}.
\end{align}
The solution $((Z_{t}^{i,N})_{t\ge 0})_{i=1,...,N}$ is a family of counting processes. 
By \cite[Proposition 1]{A}, system \eqref{sssy} admits a unique c\`adl\`ag 
solution that is $(\mathcal{F}_{t})_{t\ge 0}$-measurable, provided that 
$\phi$ is locally integrable. Here,
$$\mathcal{F}_{t}=\sigma(\Pi^{i}(A):A\in\mathcal{B}([0,t]\times [0,\infty)),i=1,...,N)\vee 
\sigma(\theta_{ij},i,j=1,...,N),$$ 
where $\mathcal{B}([0,t]\times [0,\infty))$ denotes the Borel $\sigma$-algebra on the corresponding product space.

\vip

Intuitively, the process $Z_t^{i,N}$ counts the actions of individual $i$ in $[0,t]$.  We say that individual $j$ influences individual $i$ if and only if $\theta_{ij} = 1$ (allowing for the possibility that $i = j$). At any time $t$, the $i$-th individual acts according to the  intensity $\lambda_t^{i,N}$. This intensity consists of two components: a constant {\it autonomous} rate $\mu > 0$, and an {\it interaction-driven} component of the form $$N^{-1} \sum_{j=1}^{N} \theta_{ij} \int_{0}^{t} \phi(t-s)\,dZ_s^{j,N},$$
which models imitation. The interaction term depends on the past actions of all individuals that influence $i$, weighted by $N^{-1}$, and discounts the influence of older actions through the memory kernel $\phi \colon \mathbb{R}_{+} \to \mathbb{R}_{+}$.

\subsection{Assumptions}

Define  $\Lambda :=\int_{0}^{\infty}\phi(t)dt\in (0,\infty].$ For some $q\geq 1$,
\renewcommand\theequation{{$H(q)$}}
\begin{align}\label{H(q)}
\mu \in (0,\infty),\quad \Lambda\in(0,\infty), \quad \Lambda p \in [0,1),
\quad \int_0^\infty s^q\phi(s)ds <\infty, 
\quad \int_0^\infty (\phi(s))^2ds <\infty.
\end{align}
\renewcommand\theequation{\arabic{equation}}\addtocounter{equation}{-1}

  \subsection{Model}
Consider a system of $N$ individuals. For each individual $j \in \{1,\dots,N\}$, denote by 
$S_j=\{i \in \{1,\dots,N\} : \theta_{ij}=1\}$, the set of individuals connected to $j$. The only action available to individual $i$ is to send a message to every member of $S_i$. Here $Z^{i,N}_t$ stands for the total number of messages sent by individual $i$ during $[0,t]$. The  counting process $(Z_{s}^{i,N})_{i=1...N,0\le s\le t}$ is governed by its intensity processes $(\lambda_{s}^{i,N})_{i=1...N,0\le s\le t}.$ Informally, the intensity is defined by 
\begin{align*}
    P\Big(Z_{t}^{i,N} \textit{has a jump in } [t, t + dt]\Big|\mathcal{F}_t\Big)= \lambda_{t}^{i,N}dt,\  i = 1,..., N,
\end{align*}
where   $\mathcal{F}_{t}$ denotes the sigma-field generated by $(Z_{s}^{i,N})_{i=1...N,0\le s\le t}$ and $(\theta_{ij})_{i,j=1,...,N}$.   The rate $\lambda^{i,N}_t$ at which $i$ sends messages can be decomposed into the sum of two components:

$\bullet$ {\bf New messages}: new messages  generated at rate $\mu$;

$\bullet$ {\bf Forwarded messages}: messages that $i$ has received and forwards after some delay (possibly infinite) depending on the 
age of the message, which contributes a sending rate of the form 
$$\frac{1}{N}\sum_{j=1}^{N}\theta_{ij}\int_{0}^{t-}\phi(t-s)dZ_{s}^{j,N}.$$

\smallskip
If for example $\phi=\boldsymbol{1}_{[0,K]}$, then  $N^{-1} \sum_{j=1}^N \theta_{ij}\int_0^{t-}  \phi(t-s)dZ_{s}^{j,N}$ is
precisely the number of messages that the $i$-th individual received between the time $t-K$ and $t$,
divided by $N$.
\begin{remark}
In \cite{Nature}, we  see that the  Erd\"{o}s-R\'{e}nyi  graph can be applied to model a social network. And  \cite{F} tells us that Hawkes process can model the number of the messages.
\end{remark}

\subsection{Main Goal}
 In the  present work, we  follow the same setting as \cite{D}. Specially, we consider a system of $N$ i.i.d. Hawkes point processes $(Z_{s}^{i,N})_{i=1...N,0\le s\le t}$ and a family of  i.i.d. Bernoulli($p$) random variables $(\theta_{ij})_{i,j=1,...,N}$, where $p\in[0,1]$ is an unknown parameter. The interactions among the Hawkes processes are binary and  encoded by a directed Erd\"{o}s-R\'{e}nyi random graph with $p$. The objective of this paper is to  establish the estimation of $p$  through the limit distribution of the corresponding estimator based on partial observations from $N$ Hawkes processes, that is, knowing only the first $1\ll K \le N$ processes of $(Z^{i,N})_{i=1,...,N}$ with $t$ large. 
\begin{remark}
Since the family of  $(Z^{i,N})_{i=1,...,N}$ is exchangeable,  the observation  given by  the first $K$ processes is not a restriction.
\end{remark}

\subsection{Notations}\label{notation}
Throughout this paper,  the conditional expectation given $(\theta_{ij})_{i,j=1,\dots N}$ is denoted by $\Et$. The corresponding conditional  variance and covariance are denoted by $\Vart$ and $\Covt$, respectively. For  $f,g:[0,\infty)\to\R$, we define their convolution by  $$f*g(t)=\int_0^tf(t-s)g(s)ds,\quad t>0,$$ and  $\varphi^{*n}$ denotes the $n$-fold convolution of $\varphi$. We adopt the conventions $\phi^{*0}(s)ds=\delta_0(ds)$ and $\phi^{*0}(t-s)ds=\delta_t(ds)$, so that in particular,  $\int_{0}^{t}s\phi^{*0}(t-s)ds=t$. $\stackrel{d}{\longrightarrow}$ and $\stackrel{\mathbb{P}}{\longrightarrow}$ refer to the convergence in distribution and  convergence in probability, respectively.
\smallskip

We use $C$ to denote a positive constant whose value might change from line to line.

\renewcommand\theequation{\arabic{equation}}\addtocounter{equation}{0}

\section{Main result}\label{MR}

\subsection{Main result}\label{TRISC}
We assume \eqref{H(q)} for some $q\geq 1$. The supercritical case ($\Lambda p>1$) is not addressed in this work. Its treatment would involve different techniques and significantly more technical arguments,  and is therefore deferred to a separate, completed paper for independent investigation. We first remind the estimator built in \cite{D}.
For $N\geq 1$ and for $((Z^{i,N}_t)_{t\geq 0})_{i=1,\dots,N}$ the solution of system \eqref{sssy}, we set
$\bar{Z}^{N}_{t}:=N^{-1}\sum_{i=1}^N Z^{i,N}_t$, and $\bar{Z}^{N,K}_{t}:=K^{-1}\sum_{i=1}^K Z^{i,N}_t.$ Next, we introduce
$$\varepsilon _{t}^{N,K}:=\frac 1t(\bar{Z}_{2t}^{N,K}-\bar{Z}_{t}^{N,K}),\qquad \mathcal{V}_{t}^{N,K}
:=\frac{N}{K}\sum_{i=1}^{K}\Big[\frac{Z_{2t}^{i,N}-Z_{t}^{i,N}}t-\varepsilon_{t}^{N,K}\Big]^{2}-\frac{N}{t}\varepsilon_{t}^{N,K}.$$ 
And for $\Delta>0$ such that $t/(2\Delta)\in \mathbb{N}^{*}.$
$$
\mathcal{X}_{\Delta,t}^{N,K}:=\mathcal{W}_{\Delta,t}^{N,K}-\frac{N-K}{K}\varepsilon_{t}^{N,K},
$$
where
$$
\mathcal{W}_{\Delta,t}^{N,K}=2\mathcal{Z}_{2\Delta,t}^{N,K}-\mathcal{Z}_{\Delta,t}^{N,K},\qquad
\mathcal{Z}^{N,K}_{\Delta,t}=\frac{N}{t}\sum_{a=\frac{t}{\Delta}+1}^{\frac{2t}{\Delta}}(\bar{Z}_{a\Delta}^{N,K}-\bar{Z}_{(a-1)\Delta}^{N,K}-\Delta\varepsilon_{t}^{N,K})^{2}.
$$
We then introduce the function $\Psi^{(3)}$ defined by
$$
\Psi^{(3)}(u,v,w)=\frac{u^2(1-\sqrt{\frac{u}{w}})^2}{v+u^2(1-\sqrt{\frac{u}{w}})^{2}}\quad \text{if $u>0$, $v>0$, $w>0$}
\quad \text{and \ $\Psi^{(3)}(u,v,w)=0$ otherwise.}
$$
We set
$$
\hat p_{N,K,t}:= \Psi^{(3)}(\varepsilon_{t}^{N,K},\mathcal{V}_{t}^{N,K},\mathcal{X}_{\Delta_{t},t}^{N,K}),
$$
with 
\begin{equation} \label{Deltat}
\Delta_t=(2 \lfloor t^{1-4/(q+1)}\rfloor)^{-1}t.
\end{equation}

The main result of this paper,  which is proved in Section \ref{finsecsub}, is stated below.
\begin{theorem}\label{mainsubcr}
We assume that $p>0$ and that  \eqref{H(q)}  holds for some $q> 3$. Define $\Delta_t$ by \eqref{Deltat}.
We set $c_{p,\Lambda}:=(1-\Lambda p)^2/(2\Lambda^2)$.
 We always work in the asymptotic  $(N,K,t)\to (\infty,\infty,\infty)$
and  $\frac 1{\sqrt K} + \frac NK \sqrt{\frac{\Delta_t}t}+ \frac{N}{t\sqrt K}+ Ne^{-c_{p,\lambda} K} \to 0$.

(i) The dominant term  is $\frac 1{\sqrt K}$, i.e. when
$[\frac 1{\sqrt K}]/[ \frac NK \sqrt{\frac{\Delta_t}t}+ \frac{N}{t\sqrt K}]\to \infty$,
it holds that
$$
\sqrt K \Big(\hat p_{N,K,t}-p \Big) \stackrel{d}{\longrightarrow} 
\mathcal{N}\Big(0,  p^2(1-p)^{2}\Big).
$$

(ii) The dominant term is $\frac{N}{t\sqrt K}$, i.e. when
$[\frac{N}{t\sqrt K}]/[\frac 1{\sqrt K}+\frac NK \sqrt{\frac{\Delta_t}t}]\to \infty$,
we have
$$
\frac{t\sqrt{K}}{N}\Big(\hat p_{N,K,t}-p \Big)\stackrel{d}{\longrightarrow} 
\mathcal{N}\Big(0,\frac{2(1-\Lambda p)^2}{\mu^2 \Lambda^4}\Big).
$$

(iii) The dominant term is  $\frac{N}{K}\sqrt{\frac{\Delta_t}t}$, i.e. when
$[\frac{N}{K}\sqrt{\frac{\Delta_t}t}]/[\frac 1{\sqrt K}+ \frac{N}{t\sqrt K} ]\to \infty$, 
imposing moreover that $\lim_{N,K\to\infty} \frac{K}{N}=\gamma\in [0,1]$,
$$
\frac{K}{N}\sqrt{\frac{t}{\Delta_{t}}}\Big(\hat p_{N,K,t}
-p\Big)\stackrel{d}{\longrightarrow} \mathcal{N}\Bigg(0,
\frac{6(1-p)^2}{\Lambda^2}\Big((1-\gamma)(1-\Lambda p)^{3}+\gamma(1-\Lambda p)\Big)^2\Bigg).
$$
\end{theorem}

We will not examine the cases involving two or three dominant terms, as we believe this is not very restrictive in practice. Furthermore, such a study would be much more tedious due to the difficulty of analyzing the correlations between the different terms. An alternative formulation of Theorem \ref{mainsubcr} can also be provided.

\begin{cor}
Under the assumption of Theorem \ref{mainsubcr}. We also assume $\lim_{N,K\to\infty} \frac{K}{N}=\gamma\in [0,1]$ and  $$\lim max\{\frac 1{\sqrt K},\ \frac NK \sqrt{\frac{\Delta_t}t},\  \frac{N}{t\sqrt K}\}\Big/\Big(\frac 1{\sqrt K} + \frac NK \sqrt{\frac{\Delta_t}t}+ \frac{N}{t\sqrt K}-\max\{\frac 1{\sqrt K},\ \frac NK \sqrt{\frac{\Delta_t}t},\  \frac{N}{t\sqrt K}\}\Big)=\infty.$$   Then we have 
\begin{align*}
\Big[&\max\Big\{\frac{p(1-p)}{\sqrt K},\  \frac{\sqrt{2}(1-\Lambda p)}{\mu \Lambda^2}\frac{N}{t\sqrt K},\\
&\hskip  2cm\frac{(1-p)}{\Lambda}\Big[(1-\gamma)(1-\Lambda p)^{3}+\gamma(1-\Lambda p)\Big]\frac NK\sqrt{\frac {6\Delta_t}{t} } \Big\}\Big]^{-1} \Big(\hat p_{N,K,t}-p \Big) 
\stackrel{d}{\longrightarrow} 
\mathcal{N}(0,  1).
\end{align*}
\end{cor}

\begin{remark}
This result allows us to construct an asymptotic confidence interval for $p$ in the subcritical case.
We define 
\begin{align*}
    \hat \mu_{N,K,t}:=\Psi^{(1)}(\varepsilon_{t}^{N,K},\mathcal{V}_{t}^{N,K},\mathcal{X}_{\Delta_{t},t}^{N,K}), \quad  \hat \Lambda_{N,K,t}:=\Psi^{(2)}(\varepsilon_{t}^{N,K},\mathcal{V}_{t}^{N,K},\mathcal{X}_{\Delta_{t},t}^{N,K})
\end{align*}
where
\begin{gather*}
\Psi^{(1)}(u,v,w):=u\sqrt{\frac{u}{w}},\quad  \Psi^{(2)}(u,v,w):=\frac{v+[u-\Psi^{(1)}(u,v,w)]^{2}}{u[u-\Psi^{(1)}(u,v,w)]},
\end{gather*}
if $u>0$, $v>0$, $w>u$
and \ $\Psi^{(1)}(u,v,w)=\Psi^{(2)}(u,v,w)=0$ otherwise.
By \cite[Theorem 2.3]{D}, we have, 
when $\frac 1{\sqrt K} + \frac NK \sqrt{\frac{\Delta_t}t}+ \frac{N}{t\sqrt K}+Ne^{-c_{p,\Lambda} K} \to 0$,
$$
\Big(\hat \mu_{N,K,t},  \hat \Lambda_{N,K,t},\hat p_{N,K,t}\Big)\stackrel{\mathbb{P}}{\longrightarrow} (\mu,\Lambda,p). 
$$
Hence by Theorem \ref{mainsubcr},
in the cases (i), (ii) or (iii), for $0< \alpha<1,$
$$
 \lim \mathbb{P}\Big(|\hat p_{N,K,t}-p|\le I_{N,K,t,\alpha}\Big)= 1-\alpha,
$$
where
\begin{align*}
&I_{N,K,t,\alpha}= \Phi^{-1}(1-\frac{\alpha}{2})
\Big(\frac{1}{\sqrt{K}}\hat p_{N,K,t}(1-\hat p_{N,K,t})+\frac{N}{t\sqrt{K}}\frac{\sqrt{2}(1-\hat \Lambda_{N,K,t}\, \hat p_{N,K,t})}{\hat \mu_{N,K,t}(\hat \Lambda_{N,K,t})^2}\\
&\hskip2cm +\frac{N}{K}\sqrt{\frac{6\Delta_t}{t}}\frac{(1-\hat p_{N,K,t})}{\hat \Lambda_{N,K,t}}\Big[(1-\gamma)(1-\hat \Lambda_{N,K,t} \, \hat p_{N,K,t})^{3}+\gamma(1-\hat \Lambda_{N,K,t}\, \hat p_{N,K,t})\Big]\Big),
\end{align*}
 and $\Phi(x)=\frac{1}{\sqrt{2\pi}}\int^x_{-\infty}e^{-\frac{s^2}{2}}ds.$
\end{remark}

Concerning the case $p=0$, the following result shows that $\hat p_{N,K,t}$ is not always consistent.

\begin{prop}\label{pzero}
We assume $p=0$ and that  \eqref{H(q)}  holds for some $q> 3$. We set $c_{p,\Lambda}:=(1-\Lambda p)^2/(2\Lambda^2)$.
We always work in the asymptotic  $(N,K,t)\to (\infty,\infty,\infty)$
and in the regime $ \frac NK \sqrt{\frac{\Delta_t}t}+ \frac{N}{t\sqrt K}+Ne^{-c_{p,\Lambda}K} \to 0$.

(i) If
$[\frac{N}{t\sqrt K}]/[\frac NK \sqrt{\frac{\Delta_t}t}]^2\to \infty$,
we have
$$
\hat p_{N,K,t} \stackrel{\mathbb{P}}{\longrightarrow} 0.
$$

(ii) If
$[\frac{N}{K}\sqrt{\frac{\Delta_t}t}]^2/[ \frac{N}{t\sqrt K} ]\to \infty$, we have
$$
\hat p_{N,K,t} \stackrel{d}{\longrightarrow} X
$$
where $\mathbb{P}(X=1)=\mathbb{P}(X=0)=\frac{1}{2}$.
\end{prop}


\subsection{Heuristics for the three estimators}
The three estimators were initially proposed in \cite{A} and later extended to partially observed settings in \cite{D}. They can essentially be viewed as  analogues of the sample mean, the sample variance, and a time–shifted (temporal) empirical variance for stochastic processes. Since  our objective  is to establish central limit theorem, we focus on the leading term of each estimator. For intuition on their construction, we refer to \cite[Section 2.1]{A}. The  adaptation to partial observation follows analogously, see \cite[Section 3.1]{D}. For the reader's convenience, we also provide a brief explanation below, and for more details, see \cite[Section 2.1]{A} or \cite[Section 3.1]{D}.
\vip
Consider the matrix $A_N$ defined by $A_N(i,j)=\theta_{ij}/N$ and set $Q_N=(I-\Lambda A_N)^{-1}$.  Under the subcritical  condition  $\Lambda p<1$, $Q_N$ exists with high probability and admits the series expansion $\sum_{n\ge 0} \Lambda^n A^n_N$. Let $\ell_{N}(i)=\sum_{j=1}^{N}Q_{N}(i,j)$, $c_{N}(i)=\sum_{j=1}^{N}Q_{N}(j,i)$ and $\bar{\ell}^K_N=\frac{1}{K}\sum_{i=1}^K\ell_N(i)$.
\vip
By \cite[Section 2.1]{A}, it can be informally shown that  $\bar{\ell}^K_N\simeq \frac{1}{1-\Lambda p}$ for sufficiently large $N$. Conditional on $(\theta_{ij})_{i,j=1,...,N},$ and for $t,N,K$ large enough, the law of large numbers suggests that $\sum_{i=1}^KZ^{i,N}_{t}\simeq \sum_{i=1}^K\Et [Z^{i,N}_{t}] $ (i.e. $(\sum_{i=1}^K Z^{i,N}_{t})/(\sum_{i=1}^K\Et [Z^{i,N}_{t}]) \to 1$).
Assume  the limit $\gamma_N(i): = \lim_{t\to\infty}\Et [Z^{i,N}_{t}]/t$ exists. Then Definition \eqref{sssy} implies  $\boldsymbol{\gamma}_N=\mu \boldsymbol{1}_N+\Lambda A_N  \boldsymbol{\gamma}_N$,  so that  $\boldsymbol{\gamma}_N=\mu  Q_N \boldsymbol{1}_N$. Consequently,  $t^{-1}\bar Z_t^{N,K}\simeq (tK)^{-1}\sum_{i=1}^K\Et [Z^{i,N}_{t}]\simeq \mu K^{-1} \sum_{i=1}^K\ell_N(i)\simeq \frac{\mu}{1-\Lambda p}$.
\vip
Next, we explain why we use $\bar{Z}_{2t}^{N,K}-\bar{Z}^{N,K}_{t}$ rather than $\bar{Z}^{N,K}_{t}$ itself.   Under Assumption \eqref{H(q)} (see the proof of Lemma 16 in \cite{A}), we have $\Et [Z_t^{i,N}] =\mu \ell_N(i)t +\chi^N_i+\pm t^{1-q}, $  where $\chi^N_i$ is some finite random variable.  Consequently, $t^{-1}\Et [Z_{2t}^{i,N}-Z_t^{i,N}]$ converges to $\ell_N(i)$ at rate $t^{-q}$, which is faster than the rate $t^{-1}$ obtained from $t^{-1}\Et [Z_{t}^{i,N}]$ alone. 
\vip
Based on \cite[Section 2.1]{A} and proceeding from a similar argument,  we have $\mathcal{V}_{t}^{N,K}\simeq \frac{\mu^2\Lambda^2p(1-p)}{(1-\Lambda p)^2}$ and $\mathcal{X}_{\Delta_{t},t}^{N,K}\simeq \frac{\mu}{(1-\Lambda p)^3}$. Then we can construct the partial observed estimator for parameters $(\mu,\Lambda, p)$ by  arranging $(\varepsilon_{t}^{N,K},\mathcal{V}_{t}^{N,K},\mathcal{X}_{\Delta_{t},t}^{N,K})$. 

\subsection{Plan of the paper}
After some preliminaries stated in Section \ref{sesub}, we study some random matrix in Section \ref{sec3},
 some limit theorems for the first and second estimator are established  in Section \ref{first theorem}, 
and the limit theorem for the third
one is established in Section \ref{sec5}. Finally, we conclude the proof of  the main results in Section \ref{finsecsub}. Moreover, the proofs for some technical Lemmas are presented in Appendix.

\section{Preliminaries}\label{sesub}
\subsection{Some notations}\label{subimn}
For $r\in [1,\infty)$ and $\boldsymbol{x}\in \R^N$, we set $\|\boldsymbol{x}\|_{r}=(\sum_{i=1}^{N}|x_{i}|^{r})^{\frac{1}{r}}$, and$\ \|\boldsymbol{x}\|_{\infty}=\max_{i=1...N}|x_{i}|$. For  a $N\times N$ matrix $M$, we denote by $|||M|||_{r}$  the operator norm associated to $ \|\cdot \|_{r}$, that is $|||M|||_{r}=\sup_{\boldsymbol{x}\in \R^{N}}\|M\boldsymbol{x}\|_{r}/\|\boldsymbol{x}\|_{r}$. We  have the special cases
$$
|||M|||_{1}=\sup_{j=1,...,N}\sum_{i=1}^{N}|M_{ij}|,\quad |||M|||_{\infty}=\sup_{i=1,...,N}\sum_{j=1}^{N}|M_{ij}|.
$$
We also have the inequality
$$
|||M|||_{r}\le|||M|||_{1}^{\frac{1}{r}}|||M|||_{\infty}^{1-\frac{1}{r}}\quad \hbox{for any}\quad r\in [1,\infty).
$$

We define the $N\times N$  random matrix $A_{N}$ with  $A_{N}(i,j):=N^{-1}\theta_{ij}$ for $i,j=1,\dots,N$, and the matrix 
$Q_{N}:=(I-\Lambda A_{N})^{-1}$ on the event on which $I-\Lambda A_N$ is invertible.

\vip

For $1\leq K \leq N$, we introduce the $N$-dimensional vector  $\boldsymbol{1}_K$ with $i$-th coordinate $\boldsymbol{1}_K(i)=\indiq_{\{1\leq i\leq K\}}$ for $i=1,\dots,N$, and the $N\times N$ matrix $I_K$  
defined by $I_K(i,j)={\bf 1}_{\{i=j\leq K\}}$.

\vip

Next, we define
$\boldsymbol{\ell}_N:=Q_N\boldsymbol{1}_N$, i.e. $\ell_{N}(i):=\sum_{j=1}^{N}Q_{N}(i,j)$, and $\boldsymbol{\ell}^K_{N}:=I_K\boldsymbol{\ell}_N$, i.e. $\ell^K_{N}(i):=\ell_{N}(i)\indiq_{\{i\leq K\}}$. We also set  $\bar{\ell}_N:=\frac{1}{N}\sum_{i=1}^N\ell_N(i),\ \bar{\ell}^K_N:=\frac{1}{K}\sum_{i=1}^K\ell_N(i)$, and define the difference vector $\boldsymbol{x}^{K}_{N}:=
(x_N^K(i))_{i=1,\dots,N}=\boldsymbol{\ell}_{N}^K- \bar{\ell}^{K}_{N}\boldsymbol{1}_K$ with $x_{N}^{K}(i):=(\ell_{N}(i)-\bar{\ell}^{K}_{N})\indiq_{\{i\leq K\}}$ and $\boldsymbol{x}_{N}:=(x_N(i))_{i=1,\dots,N}$, with $x_{N}(i):=\ell_{N}(i)-\bar{\ell}_{N}$.

\vip

Recall that  $\boldsymbol{1}_N$ denotes the $N$-dimensional vector  with  all coordinates equal $1$. Let $\boldsymbol{L}_N:=A_N\boldsymbol{1}_N$, so that  $L_{N}(i):=\sum_{j=1}^{N}A_{N}(i,j).$ We also define $\bar{L}_N:=
\frac{1}{N}\sum_{i=1}^NL_N(i),\ \bar{L}^K_N:=\frac{1}{K}\sum_{i=1}^KL_N(i)$ and denote the difference vector by $\boldsymbol{X}^{K}_{N}=
(X_N^K(i))_{i=1,\dots,N}$, 
 where  $X_{N}^{K}(i)=(L_{N}(i)-\bar{L}^{K}_{N})\indiq_{\{i\leq K\}}$, For convenience, we set  
 $\boldsymbol{X}_{N}:=\boldsymbol{X}_{N}^{N}$.  Let  $\boldsymbol{C}_N:=A_N^T\boldsymbol{1}_N$ ($A_N^T$ is the transpose of $A_N$),  so that $C_{N}(j):=\sum_{i=1}^{N}A_{N}(i,j)$. We also define
$\bar{C}_N:=\frac{1}{N}\sum_{j=1}^NC_N(j),\ \bar{C}^K_N:=\frac{1}{K}\sum_{j=1}^KC_N(i)$
and consider the event
\begin{align}\label{mA}
\mathcal{A}_{N}:=\{\|\boldsymbol{L}_{N}-p\boldsymbol{1}_{N}\|_{2}+\|\boldsymbol{C}_{N}-p\boldsymbol{1}_{N}\|_{2}\le N^{\frac{1}{4}}\}.
 \end{align}

\vip

We assume here that $\Lambda p\in\ (0,1)$ and  set 
$a=\frac{1+\Lambda p}{2}\in\ (\frac 1{2},1).$ We introduce  the events
\begin{eqnarray*}
&\Omega_{N}^{1}:=\Big\{\Lambda |||A_{N}|||_{r}\le a ,\ \hbox{for all } \  r\in[1,\infty]\Big\},\quad\\ &\mathcal{F}_{N}^{K,1}:=\Big\{\Lambda |||I_{K}A_{N}|||_{r}\le \Big(\frac{K}{N}\Big)^{\frac{1}{r}}a, \hbox{for}\  \hbox{all}\  r\in [1,\infty)\Big\},\\
&\mathcal{F}_{N}^{K,2}:=\Big\{\Lambda |||A_{N}I_{K}|||_{r}\le \Big(\frac{K}{N}\Big)^{\frac{1}{r}}a,\ \hbox{for}\   \hbox{all}\  r\in [1,\infty)\Big\},\\
&\Omega^{1}_{N,K}:= \Omega^1_N \cap\mathcal{F}_{N}^{K,1}, \quad \Omega^{2}_{N,K}:= \Omega^1_N \cap\mathcal{F}_{N}^{K,2},
\quad  \Omega_{N,K}=\Omega^{1}_{N,K}\cap\Omega^{2}_{N,K}.
\end{eqnarray*}

We now review the following lemma established in \cite{D} with $c_{p,\Lambda}=(1-\Lambda p)^2/(2\Lambda^2)$.

\begin{lemma}[Lemma 5.7, \cite{D}]\label{ONK}
Assume that $\Lambda p<1$. It holds that $$\mathbb{P}(\Omega_{N,K})\ge 1-CNe^{-c_{p,\Lambda}K}$$ for some constants $C>0$.
\end{lemma}

\smallskip

Next, we also remind some important result in \cite{A}.
\begin{lemma}\label{lo}
We assume that $\Lambda p<1$ and recall \eqref{mA}.  Then
$\Omega_{N,K}\subset\Omega_{N}^{1}\subset\{|||Q_{N}|||_{r}\le C, $ for all $r \in[1,\infty]\}\subset\{\sup_{i=1,\cdots,N}\ell_{N}(i)\le C\}$, where $C=(1-a)^{-1}$. For any $\alpha>0$, there exists a constant $C_{\alpha}$ such that 
$$\mathbb{P}(\mathcal{A}_{N})\ge 1-C_{\alpha}N^{-\alpha}.$$

\end{lemma}
\begin{proof}
See \cite[Notation 12 and Proposition 14, Step 1]{A}.
\end{proof}

\subsection{Some auxilliary processes}\label{aux}
Based on \eqref{sssy}, we first introduce a family of martingales: for $i=1,\dots,N$,
$$
M_{t}^{i,N}=\int_{0}^{t}\int_{0}^{\infty}  \boldsymbol{1}_{\{z\le\lambda_{s}^{i,N}\}}\widetilde{\pi}^{i}(ds,dz),
$$
where $\widetilde{\pi}^{i}(ds,dz)=\pi^i(ds,dz)-dsdz$.
We further define a family of centered   processes $U_{t}^{i,N}=Z_{t}^{i,N}-\mathbb{E}_{\theta}[Z_{t}^{i,N}]$.

\vip

Let $\boldsymbol{Z}_{t}^{N}$ (resp. $\boldsymbol{U}_{t}^{N}$, $\boldsymbol{M}_{t}^{N}$) denote the $N$-dimensional vector  with  coordinates $Z_{t}^{i,N}$ (resp. $U_{t}^{i,N}$, $M_{t}^{i,N}$). Define the vectors
$$
\boldsymbol{Z}_{t}^{N,K}=I_{K}\boldsymbol{Z}_{t}^{N},\quad \boldsymbol{U}_{t}^{N,K}=I_{K}\boldsymbol{U}_{t}^{N},
$$
and the corresponding averages $$\bar{Z}^{N,K}_{t}=K^{-1}\sum_{i=1}^{K}Z_{t}^{i,N},\,  \bar{U}^{N,K}_{t}=K^{-1}\sum_{i=1}^{K}U_{t}^{i,N},\,  \bar{M}^{N,K}_{t}=K^{-1}\sum_{i=1}^{K}M_{t}^{i,N}.$$
From \cite[Remark 10 and Lemma 11]{A}, we recall the following identities:
\begin{align}
\label{ee1}&\mathbb{E}_{\theta}[\boldsymbol{Z}_{t}^{N,K}]=\mu\sum_{n\ge0}\Big[\int_{0}^{t}s\phi^{*n}(t-s)ds\Big]I_{K}A_{N}^{n}\boldsymbol{1}_{N},\\
\label{ee2}&\boldsymbol{U}_{t}^{N,K}=\sum_{n\ge0}\int_{0}^{t}\phi^{*n}(t-s)I_{K}A_{N}^{n}\boldsymbol{M}_{s}^{N}ds,\\
\label{ee3}&[M^{i,N},M^{j,N}]_{t}=\boldsymbol{1}_{\{i=j\}}Z_{t}^{i,N}.
\end{align}
In particular, for $i=1,\cdots, N$,
\begin{align}\label{UPro}
U^{i,N}_t =  \sum_{n\geq 0} \intot \phi^{*n}(t-s) \sum_{j=1}^N A_N^n(i,j)M^{j,N}_sds.
\end{align}
Adopting the convention that $\phi^{*0}(s)ds=\delta_0(ds)$ and 
$\int_{0}^{t}s\phi^{*0}(t-s)ds=t$,  we establish some prior estimates  for the intensity process $\lambda_t^{i,N}$ defined by $(\ref{sssy})$ and  the processes introduced in Section \ref{aux}. We first review the following results established in  \cite[Lemma 6.1]{D}.

\begin{lemma}[Lemma 6.1,\cite{D}]\label{Zt}
Assume  \eqref{H(q)}  for some $q \ge 1$. 

\vip

(i) For all $r$ in $[1,\infty]$, all  $t \ge0$, a.s.,
$$
\boldsymbol{1}_{\Omega_{N,K}}\|\mathbb{E}_{\theta}[\boldsymbol{Z}_{t}^{N,K}]\|_{r}\le CtK^{\frac{1}{r}}.
$$

(ii) For any $r\in[1,\infty]$, for all $t\geq s \geq 0$,
$$
\boldsymbol{1}_{\Omega_{N,K}}\|\mathbb{E}_{\theta}[\boldsymbol{Z}^{N,K}_{t}-\boldsymbol{Z}_{s}^{N,K}-\mu(t-s)\boldsymbol{\ell}_{N}^{K}]\|_{r}\le C(\min\{1,s^{1-q}\})K^{\frac{1}{r}}.
$$

\end{lemma}

\begin{lemma}\label{lambar}
Assume  \eqref{H(q)}  for some $q\geq 1$. Then the following inequalities holds a.s. on  $\Omega_{N,K},$ 
\vip
$ (i)\ 
\sup_{t\in \R_+}\max_{i=1,...,N}\mathbb{E}_\theta[\lambda_t^{i,N}]\le \frac{\mu}{1-a}
$ \quad  and \quad  $\sup_{t\in \R_+}\max_{i=1,...,N}\mathbb{E}_\theta[(\lambda_t^{i,N})^2]^{\frac{1}{2}}\le C$.
\vip
$ (ii)$ For all $t\geq 1$,  
$$\frac{1}{K}\sum_{i=1}^{K}\mathbb{E}_{\theta}\Big[\Big(\lambda_{t}^{i,N}-\mu\ell_{N}(i)\Big)^{2}\Big]^{\frac{1}{2}}\le \frac{C}{t^{q}}+\frac{C}{\sqrt{N}}.$$

\vip
$(iii)$ For all $t\geq s+1 \geq 1$,
$$\max_{i=1,\dots,N}\mathbb{E}_{\theta}[(U^{i,N}_{t}-U_{s}^{i,N})^4]\le C(t-s)^{2}\quad  \hbox{and}\quad
  \max_{i=1,\dots,N}\mathbb{E}_{\theta}[(Z^{i,N}_{t}-Z_{s}^{i,N})^4]\le C (t-s)^{4}. 
$$

\vip
(iv) For all $t\geq s+1 \geq 1$, 
$$
 \mathbb{E}_{\theta}[(\bar{U}^{N,K}_{t}-\bar{U}_{s}^{N,K})^4]\le \frac{C (t-s)^{2}}{K^{2}}\quad \hbox{and}\quad
 \mathbb{E}_{\theta}[(\bar{Z}^{N,K}_{t}-\bar{Z}_{s}^{N,K})^4]\le C (t-s)^{4}. 
$$

\end{lemma}

The proof of Lemma \ref{lambar} is tedious which is deferred to Appendix \ref{prof: lambar}.

\section{Some limit theorems for the random matrix}\label{sec3}
In this section, we prove the asymptotic behavior of the quantities associated with the random matrix $Q_N$, which determines the  asymptotic behavior of  the estimators $\e_t^{N,K}, \mathcal{V}_{t}^{N,K}$ and $\mathcal{X}_{\Delta,t}^{N,K}$, defined in Section \ref{TRISC}.

\subsection{First estimator}\label{Fe}
Recall from Section \ref{subimn} that the event $\Omega_{N,K}$ and the quantities $\ell_{N}(i)= \sum_{j=1}^{N} Q_{N}(i,j)$ and $\bar{\ell}^K_N= \frac{1}{K} \sum_{i=1}^K \ell_N(i)$. 
As established in Lemma \ref{barell}, the estimator $\e_t^{N,K}=(\bar{Z}_{2t}^{N,K}-\bar{Z}_{t}^{N,K})/t$ is closely related to $\bar{\ell}_{N}^{K}$. To establish the limit of $\e_t^{N,K}$, we therefore require the following inequality for $\bar{\ell}_{N}^{K}$ which proved in  \cite[Lemma 5.9]{D}.

\begin{lemma}[Lemma 5.9, \cite{D}]\label{ellp}
If $\Lambda p<1$, there is $C>0$ such that for all $1\leq K \leq N$,
$$
\mathbb{E}\Big[\boldsymbol{1}_{\Omega_{N,K}}\Big|\bar{\ell}_{N}^{K}-\frac{1}{1-\Lambda p}\Big|^{2}\Big]\le\frac{C}{NK}.
$$
\end{lemma}

\subsection{Second estimator}\label{Sece}
Recall the estimator $$\mathcal{V}_{t}^{N,K}
=\frac{N}{K}\sum_{i=1}^{K}\Big[\frac{Z_{2t}^{i,N}-Z_{t}^{i,N}}t-\varepsilon_{t}^{N,K}\Big]^{2}-\frac{N}{t}\varepsilon_{t}^{N,K},$$ 
with $\varepsilon _{t}^{N,K}=(\bar{Z}_{2t}^{N,K}-\bar{Z}_{t}^{N,K})/t$, and the definitions from Section \ref{subimn} of the matrices $A_N$, $Q_N$, the event $\Omega_{N,K}$,  the quantities $\ell_{N}(i)=\sum_{j=1}^{N}Q_{N}(i,j)$ and $\bar{\ell}^K_N=\frac{1}{K}\sum_{i=1}^K\ell_N(i)$. Furthermore, there is  a close connection between the second estimator $\cV_t^{N,K}$ and $\cV_\infty^{N,K}=\frac{N \mu^2}{K}\|\bx^{K}_{N}\|_{2}^{2}$ (see Theorem \ref{VVNK}),  where $\boldsymbol{x}^{K}_{N}=
(x_N^K(i))_{i=1,\dots,N}=\boldsymbol{\ell}_{N}^K- \bar{\ell}^{K}_{N}\boldsymbol{1}_K$ with $x_{N}^{K}(i)=(\ell_{N}(i)-\bar{\ell}^{K}_{N})\indiq_{\{i\leq K\}}$ defined  in Section \ref{subimn}. Hence, determining the limit of $\cV_t^{N,K}$ is equivalent to finding  the limit of $\cV_\infty^{N,K}$.

\begin{theorem}\label{21}
Assume $\Lambda p<1$. Then,  as $(N,K)\to (\infty,\infty),$ and  $Ne^{-c_{p,\Lambda}K}\to 0$ with $c_{p,\Lambda}=(1-\Lambda p)^2/(2\Lambda^2)$,
$$
\indiq_{\Omega_{N,K}} \sqrt K\Big( \cV_\infty^{N,K}-\frac{\mu^2\Lambda^{2}p(1-p)}{(1-\Lambda p)^{2}}\Big)
\stackrel{d}\longrightarrow \mathcal{N}\Big(0,\Big(\mu^2\Lambda^{2}\frac{p(1-p)}{(1-\Lambda p)^2}\Big)^{2}\Big).$$
\end{theorem}
 We first write  the following decomposition
\begin{align*}
&\sqrt{K}\Big(\cV_\infty^{N,K}\!-\!\frac{\mu^2\Lambda^{2}p(1-p)}{(1-\Lambda p)^{2}}\Big)\\
&=\frac{N\mu^2}{\sqrt{K}}\Big(\|\bx^{K}_{N}\|_{2}^{2}-(\Lambda\bar{\ell}_{N})^{2}\|\bX_{N}^{K}\|_{2}^{2}\Big)
+\frac{N(\mu\Lambda\bar{\ell}_{N})^{2}}{\sqrt{K}}\|\bX_{N}^{K}\|_{2}^{2}
-\frac{\mu^2\Lambda^{2}p(1-p)\sqrt{K}}{(1-\Lambda p)^{2}},
\end{align*}
where  $\boldsymbol{X}^{K}_{N}=
(X_N^K(i))_{i=1,\dots,N}=\boldsymbol{L}_{N}^K- \bar{L}^{K}_{N}\boldsymbol{1}_K$ with $X_{N}^{K}(i)=(L_{N}(i)-\bar{L}^{K}_{N})\indiq_{\{i\leq K\}}$, $L_{N}(i)=\sum_{j=1}^{N}A_{N}(i,j)$, and $\bar{L}^K_N=\frac{1}{K}\sum_{i=1}^KL_N(i)$  defined in Section \ref{subimn}.  The  proof of Theorem \ref{Sece} then proceeds by analyzing these terms separately. The term $\frac{N(\mu\Lambda\bar{\ell}_{N})^{2}}{\sqrt{K}}\|\bX_{N}^{K}\|_{2}^{2}-\frac{\mu^2\Lambda^{2}p(1-p)\sqrt{K}}{(1-\Lambda p)^{2}}$ constitutes the principal term (see Lemma \ref{simple}-(iv)), whereas the term $\frac{N\mu^2}{\sqrt{K}}(\|\bx^{K}_{N}\|_{2}^{2}-(\Lambda\bar{\ell}_{N})^{2}\|\bX_{N}^{K}\|_{2}^{2})$ is shown to be negligible (see Lemma \ref{simple}-(iii)).

\vip

We now turn to Lemma \ref{simple}, whose proof is deferred to Appendix \ref{app: prof: 4.4}.  

\begin{lemma}\label{simple}
Assume $\Lambda p<1$  and recall  $\mathcal{A}_N$ in (\ref{mA}), there is $C>0$ such that for all $1\leq K \leq N$, 

 $(i)\ 
\mathbb{E}[||(I_{K}A_{N})^{T}\bX_{N}^{K})||_{2}^{2}]\le \frac{CK^2}{N^{3}}.
$
\vip
 $(ii)\ \mathbb{E}\Big[\Big|\Big(I_{K}A_{N}\bX_{N},\bX_{N}^{K}\Big)\Big|\Big]\le \frac{CK}{N^2},$ here $(\cdot,\cdot)$ is the inner product between two vectors.

\vip
 
$(iii)\ 
\frac{N}{K}\mathbb{E}\Big[\boldsymbol{1}_{\Omega_{N,K}\cap \cA_{N}}\Big|\Big(\|\bx^{K}_{N}\|_{2}^{2}-(\Lambda\bar{\ell}_{N})^{2}\|\bX_{N}^{K}\|^{2}_{2}\Big)-\|\bx_{N}^{K}-\bar{\ell}_{N}\Lambda \bX_{N}^{K}\|^{2}_{2}\Big|\Big]\le \frac{C}{N}.
$
\vip
$(iv)$  as $(N,K)\to (\infty,\infty)$ and $Ne^{-c_{p,\Lambda}K}\to 0$, where $c_{p,\Lambda}=(1-\Lambda p)^2/(2\Lambda^2)$,
$$
\omg\sqrt{K}\Big[\frac{N}{K}(\bar{\ell}_{N}\|\bX_{N}^{K}\|_{2})^{2}-\frac{p(1-p)}{(1-\Lambda p)^{2}}\Big]\stackrel{d}{\longrightarrow} \mathcal{N}\Big(0,\Big(\frac{p(1-p)}{(1-\Lambda p)^2}\Big)^2\Big).
$$
\end{lemma}

Now,  we  give the proof of Theorem \ref{21}.

\begin{proof}[Proof of Theorem \ref{21}]
Recalling that $\cV_\infty^{N,K}=\frac{N\mu^2}{K}\|\bx^{K}_{N}\|_{2}^{2}$, we write 
\begin{align*}
&\sqrt{K}\Big(\cV_\infty^{N,K}\!-\!\frac{\mu^2\Lambda^{2}p(1-p)}{(1-\Lambda p)^{2}}\Big)\\
&=\frac{N\mu^2}{\sqrt{K}}\Big(\|\bx^{K}_{N}\|_{2}^{2}-(\Lambda\bar{\ell}_{N})^{2}\|\bX_{N}^{K}\|_{2}^{2}\Big)
+\frac{N\mu^2(\Lambda\bar{\ell}_{N})^{2}}{\sqrt{K}}\|\bX_{N}^{K}\|_{2}^{2}
-\frac{\mu^2\Lambda^{2}p(1-p)\sqrt{K}}{(1-\Lambda p)^{2}}.
\end{align*}
By Lemma \ref{simple}-(iv), it suffices to 
check that
$$\zeta_{N,K}:=\indiq_{\Omega_{N,K}}\frac{N}{\sqrt{K}}\Big(\|\bx^{K}_{N}\|_{2}^{2}-(\Lambda\bar{\ell}_{N})^{2}\|\bX_{N}^{K}\|_{2}^{2}\Big)$$ 
converges to $0$ in probability. Since $\indiq_{\cA_N} \to 1$ a.s. by Lemma \ref{lo},  it is enough to verify that $\indiq_{\cA_N}\zeta_{N,K} \to 0$ in probability. To this end, we write
\begin{align*}
\E[\indiq_{\cA_N}\zeta_{N,K}]\leq& \frac{N}{\sqrt{K}}\E[\indiq_{\Omega_{N,K}\cap \cA_N} \|\bx^{K}_{N}- \Lambda\bar{\ell}_{N}\bX_{N}^{K}\|_{2}^{2}]\\
&+ \frac{N}{\sqrt{K}}\mathbb{E}\Big[\boldsymbol{1}_{\Omega_{N,K}\cap \cA_{N}}\Big|\Big(\|\bx^{K}_{N}\|_{2}^{2}-(\Lambda\bar{\ell}_{N})^{2}\|\bX_{N}^{K}\|^{2}_{2}\Big)-\|\bx_{N}^{K}-\bar{\ell}_{N}\Lambda \bX_{N}^{K}\|^{2}_{2}\Big|\Big].
\end{align*}
By \cite[Lemma 5.11]{D}, the first term is bounded by $C/\sqrt K$. By Lemma \ref{simple}-(iii), the second term is bounded by $C \sqrt K / N\le C / \sqrt N$, which completes the proof.
\end{proof}

\subsection{Third estimator}\label{Te}
For $\Delta>1$  satisfying  $t/(2\Delta)\in \mathbb{N}^{*}$, we recall the definition $\mathcal{X}_{\Delta,t}^{N,K}=\mathcal{W}_{\Delta,t}^{N,K}-\frac{N-K}{K}\varepsilon_{t}^{N,K},$ where
$\mathcal{W}_{\Delta,t}^{N,K}=2\mathcal{Z}_{2\Delta,t}^{N,K}-\mathcal{Z}_{\Delta,t}^{N,K}, 
\mathcal{Z}^{N,K}_{\Delta,t}=\frac{N}{t}\sum_{a=\frac{t}{\Delta}+1}^{\frac{2t}{\Delta}}(\bar{Z}_{a\Delta}^{N,K}-\bar{Z}_{(a-1)\Delta}^{N,K}-\Delta\varepsilon_{t}^{N,K})^{2}$ and $\varepsilon _{t}^{N,K}=(\bar{Z}_{2t}^{N,K}-\bar{Z}_{t}^{N,K})/t$. Further recall the matrices $A_N$ and $Q_N$, the event $\Omega_{N,K}$ defined in Section \ref{subimn} , as well as $\ell_{N}(i)=\sum_{j=1}^{N}Q_{N}(i,j)$ and $\bar{\ell}^K_N=\frac{1}{K}\sum_{i=1}^K\ell_N(i)$. Now, taking $\Delta$ specifically as $\Delta_t=(2 \lfloor t^{1-4/(q+1)}\rfloor)^{-1}t$ defined in \eqref{Deltat}, we will see (Theorem \ref{corX})  that the third estimator $\cX^{N,K}_{\Delta_t,t}$ is closely related to 
\begin{align*}
   \cX^{N,K}_{\infty,\infty}=\cW^{N,K}_{\infty,\infty}-\frac{(N-K)\mu}{K}\bar{\ell}_N^K,
\end{align*}
where $\cW^{N,K}_{\infty,\infty}=\mu\frac{N}{K^{2}}A^{N,K}_{\infty,\infty}$, $A^{N,K}_{\infty,\infty}=\sum_{j=1}^{N}\Big(\sum_{i=1}^{K}Q_{N}(i,j)\Big)^{2}\ell_{N}(j)$. Therefore, establishing the convergence of $\cX^{N,K}_{\Delta_t,t}$ reduces to establishing the convergence of  $\cX^{N,K}_{\infty,\infty}$. The latter relies on the following two key estimates.

\begin{lemma}[Lemma 5.19, \cite{D}]\label{W}
If $\Lambda p<1$, there is $C>0$ such that for all $1\leq K \leq N$, 
$$
\mathbb{E}\Big[\boldsymbol{1}_{\Omega_{N,K}}\Big|\cX^{N,K}_{\infty,\infty}-\frac{\mu}{(1-\Lambda p)^{3}}\Big|\Big]\le \frac{C}{K}.
$$

\end{lemma}


The objective of the  following lemma is to establish that
$\frac{(A^{N,K}_{\infty,\infty})^{2}}{2K^{2}}$ is close to $\frac{1}{2}(\frac{N-K}{N(1-\Lambda p)}+\frac{K}{N(1-\Lambda p)^3})^2$.

\begin{lemma}\label{Agamma}
When $(N,K)$ tends to $(\infty,\infty)$, with $K\le N$ and in the regime where  
$\lim_{N,K\to\infty} \frac{K}{N}=\gamma\in [0,1]$, we have 
$$
\boldsymbol{1}_{\Omega_{N,K}} \frac{A^{N,K}_{\infty,\infty}}{K} \longrightarrow \frac{1-\gamma}{(1-\Lambda p)}+\frac{\gamma}{(1-\Lambda p)^3},
$$
in probability.
\end{lemma}

\begin{proof}
Recalling that $\cX^{N,K}_{\infty,\infty}=\cW^{N,K}_{\infty,\infty}-\frac{(N-K)\mu}{K}\bar{\ell}_N^K$ and $\cW^{N,K}_{\infty,\infty}=(\mu N / K^2) A^{N,K}_{\infty,\infty}$, we obtain
$$
\frac{A^{N,K}_{\infty,\infty}}K=  \frac{K}{\mu N} \cX^{N,K}_{\infty,\infty} + \frac{N-K}{N}\bar{\ell}_N^K.
$$
The result then follows  immediately  by combining the convergence $\bar{\ell}_N^K\to \frac{1}{1-\Lambda p}$ in  probability from Lemma \ref{ellp} with  $\cX^{N,K}_{\infty,\infty}\to\frac{\mu}{(1-\Lambda p)^{3}}$ in probability from Lemma \ref{W}.
\end{proof}

\section{The limit theorems for the first and second estimators}\label{first theorem}

This section is devoted to establishing the asymptotic behavior of the estimators $\varepsilon_{t}^{N,K}$ and $\cV_{t}^{N,K}$, defined in Section \ref{TRISC}.
\begin{itemize}
\item For $\varepsilon_{t}^{N,K}$, its limit follows directly from \cite[Lemma 7.3]{D}, as stated in Lemma \ref{barell}.  
\vip
\item For $\cV_{t}^{N,K}$, however, a more delicate analysis is required. We begin by decomposing  $\frac{t\sqrt{K}}{N}(\cV_{t}^{N,K}-\cV_\infty^{N,K})$ into several terms, namely, $J^{N,K,1}_t$, $J^{N,K,211}_t$, $J^{N,K,212}_t$, $J^{N,K,213}_t$, $J^{N,K,22}_t$, 
$J^{N,K,23}_t, J^{N,K,3}_t$. We observe that on $\Omega_{N,K}$, the dominant contribution  is  coming from the term $J^{N,K,211}_t$ (see Lemma \ref{Delta1}). Further decomposition of this term reveals that the leading-order asymptotic behavior is determined by
$$\frac{2}{t\sqrt{K}}\sum_{i=1}^K\int_{t}^{2t}(M_{s-}^{i,N}-M_{t}^{i,N})dM_{s}^{i,N},$$ 
as shown in Step 3 of the proof of Theorem \ref{VVNK}. This expression converges in distribution to a Gaussian random variable with variance $2\mu^2/(1-\Lambda p)^2$, as established in  Lemma \ref{Gauss}.
\end{itemize}

Let us remind that $\mathcal{V}_{t}^{N,K}
=\frac{N}{K}\sum_{i=1}^{K}\Big[\frac{Z_{2t}^{i,N}-Z_{t}^{i,N}}t-\varepsilon_{t}^{N,K}\Big]^{2}-\frac{N}{t}\varepsilon_{t}^{N,K}$ defined in Section \ref{TRISC} and that $\cV_\infty^{N,K}=\frac{N}{K}\mu^2\|\bx^{K}_{N}\|_{2}^{2}$,  where $\varepsilon _{t}^{N,K}=(\bar{Z}_{2t}^{N,K}-\bar{Z}_{t}^{N,K})/t$, $x_{N}^{K}(i)=(\ell_{N}(i)-\bar{\ell}^{K}_{N})\indiq_{\{i\leq K\}}$ and $\boldsymbol{x}^{K}_{N}=
(x_N^K(i))_{i=1,\dots,N}$. The definition of $(\ell_{N}(i))_{i=1,\dots,N}$ and $\bar{\ell}^{K}_{N}$ are introduced in Section \ref{subimn}.
\begin{lemma}[Lemma 7.3, \cite{D}]\label{barell}
Assume  \eqref{H(q)}  for some $q\geq 1$, in the regime $\frac{K}{t^{2q}} \to 0,$ we have
$$
\lim_{(N,K,t)\to (\infty,\infty,\infty)} \boldsymbol{1}_{\Omega_{N,K}}\sqrt{K}\mathbb{E}_{\theta}\Big[\Big|\varepsilon_{t}^{N,K}-\mu\bar{\ell}_{N}^{K}\Big|\Big]=0,
$$
almost surely.
\end{lemma}

The main result of this section is the following limit theorem.

\begin{theorem}\label{VVNK}
Assume  \eqref{H(q)}  for some $q> 1$. When $(N,K,t)\to (\infty,\infty,\infty)$ and 
$\frac{t\sqrt{K}}{N}(\frac{N}{t^q}+\sqrt{\frac{N}{Kt}})+Ne^{-c_{p,\Lambda}K}\to 0$ with $c_{p,\Lambda}=(1-\Lambda p)^2/(2\Lambda^2)$, we have
$$
\omg\frac{t\sqrt{K}}{N}(\cV_{t}^{N,K}-\cV_\infty^{N,K})\stackrel{d}{\longrightarrow}\mathcal{N}
\Big(0,\frac{2\mu^2}{(1-\Lambda p)^2}\Big).
$$
\end{theorem}
Prior to the proof, we decompose the difference $\mathcal{V}_{t}^{N,K}-\mathcal{V}_{\infty}^{N,K}:= J_{t}^{N,K,1}+J_{t}^{N,K,2}+J_{t}^{N,K,3},$
where 
\begin{align*}
J_{t}^{N,K,1}&=\frac{N}{K}\Big\{\sum_{i=1}^{K}\Big[\frac{Z_{2t}^{i,N}-Z_{t}^{i,N}}t-\varepsilon_{t}^{N,K}\Big]^{2}
-\sum_{i=1}^{K}\Big[\frac{Z_{2t}^{i,N}-Z_{t}^{i,N}}t-\mu\bar{\ell}_{N}^{K}\Big]^{2}\Big\}, \\
J_{t}^{N,K,2}&=\frac{N}{K}\Big\{\sum_{i=1}^{K}\Big[\frac{Z_{2t}^{i,N}-Z_{t}^{i,N}}t-\mu\ell_{N}(i)\Big]^{2}-\frac Kt
\varepsilon_{t}^{N,K}\Big\},\\
J_{t}^{N,K,3}&=2\frac{N}{K}\sum_{i=1}^{K}\Big[\frac{Z_{2t}^{i,N}-Z_{t}^{i,N}}t-\mu\ell_{N}(i)\Big]\Big[\mu\ell_{N}(i)-\mu\bar{\ell}_{N}^{K}\Big].
\end{align*}
We further decompose 
$J_{t}^{N,K,2}= J_{t}^{N,K,21}+J_{t}^{N,K,22}+J_{t}^{N,K,23},$
where
\begin{align*}
J_{t}^{N,K,21}&=\frac{N}{K}\Big\{\sum_{i=1}^{K}\Big[\frac{Z_{2t}^{i,N}-Z_{t}^{i,N}}t-\frac{\mathbb{E}_{\theta}[Z_{2t}^{i,N}-Z_{t}^{i,N}]}t\Big]^{2}-\frac Kt\varepsilon_{t}^{N,K}\Big\},\\
J_{t}^{N,K,22}&=\frac{N}{K}\sum_{i=1}^{K}\Big\{\frac{\mathbb{E_{\theta}}[Z_{2t}^{i,N}-Z_{t}^{i,N}]}t-
\mu\ell_{N}(i)\Big\}^{2},\\
J_{t}^{N,K,23}&=2\frac{N}{K}\sum_{i=1}^{K}\Big[\frac{Z_{2t}^{i,N}-Z_{t}^{i,N}}t-
\frac{\mathbb{E}_{\theta}(Z_{2t}^{i,N}-Z_{t}^{i,N})}t\Big]\Big[\frac{\mathbb{E}_{\theta}(Z_{2t}^{i,N}-Z_{t}^{i,N})}t-
\mu\ell_{N}(i)\Big].
\end{align*}
Recalling that  $U_{t}^{i,N}=Z_{t}^{i,N}-\mathbb{E}_{\theta}[Z_{t}^{i,N}]$, we further write  $J_{t}^{N,K,21}= J_{t}^{N,K,211}+J_{t}^{N,K,212}+J_{t}^{N,K,213}$, where
\begin{align*}
J_{t}^{N,K,211}&= \frac{N}{K}\sum_{i=1}^{K}\Big\{\frac{(U_{2t}^{i,N}-U_{t}^{i,N})^{2}}{t^{2}}-
\frac{\mathbb{E_{\theta}}[(U_{2t}^{i,N}-U_{t}^{i,N})^{2}]}{t^{2}}\Big\},\\
J_{t}^{N,K,212}&= \frac{N}{K}\Big\{\sum_{i=1}^{K}\frac{\mathbb{E_{\theta}}[(U_{2t}^{i,N}-U_{t}^{i,N})^{2}]}{t^{2}}
-\frac Kt\mathbb{E_{\theta}}[\varepsilon_{t}^{N,K}]\Big\},\\
J_{t}^{N,K,213}&= \frac{N}{K}\Big\{\frac Kt\mathbb{E_{\theta}}[\varepsilon_{t}^{N,K}]-\frac Kt \varepsilon_{t}^{N,K}\Big\}.
\end{align*}
Finally, we also write $J_{t}^{N,K,3}:=J_{t}^{N,K,31}+J_{t}^{N,K,32}$, where
\begin{align*}
J_{t}^{N,K,31}&=2\frac{N}{K}\sum_{i=1}^{K}\Big[\frac{Z_{2t}^{i,N}-Z_{t}^{i,N}}t-
\frac{\mathbb{E_{\theta}}[Z_{2t}^{i,N}-Z_{t}^{i,N}]}t\Big]\Big[\mu\ell_{N}(i)-\mu\bar{\ell}_{N}^{K}\Big],\\
J_{t}^{N,K,32}&=2\frac{N}{K}\sum_{i=1}^{K}\Big[\frac{\mathbb{E_{\theta}}[Z_{2t}^{i,N}-Z_{t}^{i,N}]}t-\mu\ell_{N}(i)\Big]\Big[\mu\ell_{N}(i)-\mu\bar{\ell}_{N}^{K}\Big].
\end{align*}

Although the decomposition above is somewhat involved, the principal term is $J_{t}^{N,K,211}$, which converges to a Gaussian distribution after normalization. The remaining terms, namely,  $J_{t}^{N,K,1},$ $J_{t}^{N,K,22},$ $J_{t}^{N,K,23},$ $J_{t}^{N,K,213},$ $J_{t}^{N,K,32},$ $J_{t}^{N,K,212},$ $J_{t}^{N,K,31}$ are all suitably bounded as a consequence of Lemma \ref{Delta1}.  

\begin{lemma}\label{Delta1}
Assume  \eqref{H(q)}  for some $q> 1$.
When $(N,K,t)\to (\infty,\infty,\infty)$ and 
$\frac{t\sqrt{K}}{N}(\frac{N}{t^q}+\sqrt{\frac{N}{Kt}})\to 0$, 
$$
\frac{t\sqrt{K}}{N}\mathbb{E}\Big[\omg \Big|J_{t}^{N,K,1}+J_{t}^{N,K,212}+J_{t}^{N,K,213}+J_{t}^{N,K,22}+J_{t}^{N,K,23}+J_{t}^{N,K,3}\Big| \Big]\longrightarrow 0.
$$
\end{lemma}
\begin{proof}
Bounds for  $J_{t}^{N,K,1}, J_{t}^{N,K,22}, J_{t}^{N,K,23},J_{t}^{N,K,213},J_{t}^{N,K,32}$ are provided in \cite[Lemma 8.2]{D}. While  $J_{t}^{N,K,212}$ is bounded by  \cite[Lemma 8.3]{D}. It remains to handle $J_{t}^{N,K,3} = J_{t}^{N,K,31}+J_{t}^{N,K,32}$. For $J_{t}^{N,K,31}$,  \cite[Lemma 8.5]{D} implies that
$$
\boldsymbol{1}_{\Omega_{N,K}\cap \cA_{N}}\mathbb{E}_\theta[|J_{t}^{N,K,31}|] \leq C 
\frac{N}{K\sqrt t} \Big[\sum_{i=1}^K ( \ell_N(i)-\bar\ell^K_N)^2\Big]^{1/2}=
 C \frac{N}{K\sqrt t} ||\bx_N^K ||_2.
$$
 Taking expectation and applying the estimate 
  $\E[\boldsymbol{1}_{\Omega_{N,K}\cap \cA_{N}} ||\bx_N^K ||_2] \leq C K^{1/2}N^{-1/2}$ from \cite[Lemmas 5.14]{D}, we obtain $\mathbb{E}\big[\omg |J_{t}^{N,K,31}|\big]\leq C\sqrt{\frac{N}{Kt}}$. The desired result follows by aggregating the individual bounds for all terms.
\end{proof}
The following tedious lemma will allow us to treat the contribution term $J^{N,K,211}_t$.
\begin{lemma}\label{Gauss}
Assume \eqref{H(q)}  for some $q> 1$. For $u\in [0,1]$, define the process  
$$N_{u}^{t,i,N}:=\int_{t}^{t+\sqrt{u}t}(M_{s-}^{i,N}-M_{t}^{i,N})dM_{s}^{i,N},$$ where $M^{i,N}$ defined in Section \ref{aux}.  When $(N,K,t)\to (\infty,\infty,\infty)$ ,
\begin{equation}\label{ababa}
\Big(\frac{1}{t\sqrt{K}} \sum_{i=1}^{K} N_{u}^{t,i,N} \Big)_{u\in [0,1]}
\stackrel{d}\longrightarrow \Big(\frac{\mu}{\sqrt{2}(1-\Lambda p)} B_u \Big)_{u\in[0,1]},
\end{equation}
where $(B_u)_{u\in [0,1]}$ is a Brownian motion. 
\end{lemma}

\begin{proof}
Note that for fixed $t\geq 0$, the process $(N_{u}^{t,i,N})_{u \in [0,1]}$ is a martingale 
w.r.t  the filtration $\mathcal{F}^{N}_{t+\sqrt{u}t}$.  To prove \eqref{ababa}, we apply Jacod-Shiryaev \cite[Theorem VIII-3-8]{B}, which requires verifying that  as $(t,N,K)\to (\infty,\infty,\infty)$,
\vip

(a) $[\frac{1}{t\sqrt{K}} \sum_{i=1}^{K} N_{.}^{t,i,N},\frac{1}{t\sqrt{K}} \sum_{i=1}^{K} N_{.}^{t,i,N}]_u \to \frac{\mu^2}{2(1-\Lambda p)^2}u$ in probability, for all $u\in [0,1]$ fixed.

\vip

(b) $\sup_{u \in [0,1]} \frac{1}{t\sqrt{K}} \sum_{i=1}^{K} |N_{u}^{t,i,N}-N_{u-}^{t,i,N}| \to 0$ in probability.

\vip

The verification of point (b) is relatively straightforward. Using the independence of the Poisson measures in  \eqref{sssy} and the fact that the jumps of $M^{i,N}$ are always of size $1$, we obtain
\begin{align*}
\frac{1}{t\sqrt{K}}\mathbb{E}\Big[\omg\sup_{u\in[0,1]} \sum_{i=1}^{K} \Big|N_{u}^{t,i,N}-N_{u-}^{t,i,N}\Big| \Big]
\le& \frac{C}{t\sqrt{K}}\mathbb{E}\Big[\omg\sup_{u\in[0,1]}\max_{i=1,...,K}\Big|M_{t+t\sqrt{u}}^{i,N}-M_{t}^{i,N}\Big|\Big]\\
\le& \frac{C}{t\sqrt{K}}\mathbb{E}\Big[\omg\sup_{u\in[0,1]}\Big|\sum_{i=1}^K(M_{t+t\sqrt{u}}^{i,N}-M_{t}^{i,N})^{2}\Big|^\frac{1}{2}\Big].
\end{align*}
Applying the Cauchy-Schwarz  inequality and  using \eqref{ee3} yields
\begin{align*}
\frac{1}{t\sqrt{K}}\mathbb{E}\Big[\omg\sup_{u\in[0,1]} \sum_{i=1}^{K} \Big|N_{u}^{t,i,N}-N_{u-}^{t,i,N}\Big| \Big]
\le& \frac{C}{t\sqrt{K}}\mathbb{E}\Big[\omg\sup_{u\in[0,1]}\sum_{i=1}^K(M_{t+t\sqrt{u}}^{i,N}-M_{t}^{i,N})^{2}\Big]^\frac{1}{2}\\
\le& \frac{C}{t\sqrt{K}}\mathbb{E}\Big[\omg\Big|\sum_{i=1}^K(Z_{2t}^{i,N}-Z_{t}^{i,N})\Big|\Big]^\frac{1}{2}
\le \frac{C}{\sqrt{t}}.
\end{align*}
The last inequality follows from Lemma \ref{Zt}-(ii) with $K=N$ and $r=\infty$, which gives us that
$\max_{i=1,\dots,N}\Et[Z^{i,N}_{t}-Z^{i,N}_{s}]\le C(t-s)$ on $\Omega_{N,K}\subset \Omega_{N,N}$.  
\vip
Regarding point (a),  recall that $Z^{i,N}_t=M^{i,N}_t+\int_0^t \lambda^{i,N}_s ds$. For fixed $u$, we  write
\begin{align*}
\Big[\frac{1}{t\sqrt{K}}\sum_{i=1}^{K}N_{.}^{t,i,N}, \frac{1}{t\sqrt{K}} \sum_{i=1}^{K} N_{.}^{t,i,N}\Big]_{u}=& \frac{1}{t^2 K}
\sum_{i=1}^{K}\int_{t}^{t+\sqrt{u}t}(M_{s-}^{i,N}-M_{t}^{i,N})^{2}dZ_{s}^{i,N} \\
&:= \varUpsilon^1_{t,N,K,u}+\varUpsilon^2_{t,N,K,u}+\varUpsilon^3_{t,N,K,u},
\end{align*}
where,
\begin{align*}
\varUpsilon^1_{t,N,K,u}:=&\frac{1}{t^2 K}\sum_{i=1}^{K}\int_{t}^{t+\sqrt{u}t}(M_{s-}^{i,N}-M_{t}^{i,N})^{2}dM_{s}^{i,N},\\
\varUpsilon^2_{t,N,K,u}:=&\frac{1}{t^2 K}\sum_{i=1}^{K}\int_{t}^{t+\sqrt{u}t}(M_{s}^{i,N}-M_{t}^{i,N})^{2}(\lambda_{s}^{i,N}-\mu \ell_{N}(i))ds,\\
\varUpsilon^3_{t,N,K,u}:=&\frac{1}{t^2 K} \sum_{i=1}^{K}\mu\ell_{N}(i)\int_{t}^{t+\sqrt{u}t}(M_{s}^{i,N}-M_{t}^{i,N})^{2}ds.
\end{align*}
Each term will be handled in a separate step.
\vip
{\bf Step 1.}
In this step,  we verify that $\E[\omg \varUpsilon^1_{t,N,K,u}] \to 0$ as $(N,K,t)\to (\infty,\infty,\infty)$.  Using  \eqref{ee3}, we obtain
\begin{align*}
\Et[&(\varUpsilon^1_{t,N,K,u})^2]=\frac{1}{K^{2}t^{4}}\sum_{i=1}^{K}\mathbb{E}_{\theta}\Big[\int_{t}^{t+\sqrt{u}t}(M_{s-}^{i,N}-M_{t}^{i,N})^{4}dZ_{s}^{i,N}\Big]\\
=&\frac{1}{K^{2}t^{4}}\sum_{i=1}^{K}\mathbb{E}_{\theta}\Big[\int_{t}^{t+\sqrt{u}t}(M_{s}^{i,N}-M_{t}^{i,N})^{4}\lambda_{s}^{i,N}ds\Big]\\
\le& \frac{1}{K^{2}t^{4}}\sum_{i=1}^{K}\int_{t}^{t+\sqrt{u}t}\Big\{\mathbb{E}_{\theta}[(M_{s}^{i,N}-M_{t}^{i,N})^{4}|\lambda_{s}^{i,N}-\mu \ell_{N}(i)|]+\mu\mathbb{E}_{\theta}[(M_{s}^{i,N}-M_{t}^{i,N})^{4}]|\ell_{N}(i)|\Big\}ds.
\end{align*}
Applying the Cauchy–Schwarz and Burkholder inequalities, we further obtain
\begin{align*}
&\Et[(\varUpsilon^1_{t,N,K,u})^2]\\
\le& \frac{1}{K^{2}t^{4}}\sum_{i=1}^{K}\int_{t}^{t+\sqrt{u}t}\!\!\!\Big\{\mathbb{E}_{\theta}[(M_{s}^{i,N}-M_{t}^{i,N})^{8}]^\frac{1}{2}\mathbb{E}_{\theta}[|\lambda_{s}^{i,N}-\mu\ell_{N}(i)|^2]^\frac{1}{2}+C\mu\mathbb{E}_{\theta}[(Z_{s}^{i,N}-Z_{t}^{i,N})^{2}]|\ell_{N}(i)|\Big\}ds\\
\le& \frac{C}{K^{2}t^{4}}\sum_{i=1}^{K}\int_{t}^{t+\sqrt{u}t}\!\!\!\Big\{\mathbb{E}_{\theta}[(Z_{s}^{i,N}-Z_{t}^{i,N})^{4}]^\frac{1}{2}\mathbb{E}_{\theta}[|\lambda_{s}^{i,N}-\mu\ell_{N}(i)|^2]^\frac{1}{2}+\mu\mathbb{E}_{\theta}[(Z_{s}^{i,N}-Z_{t}^{i,N})^{4}]^\frac{1}{2}|\ell_{N}(i)|\Big\}ds.
\end{align*}
By Lemma \ref{lambar}-(iii), on
$\Omega_{N,K}$, we have
$\max_{i=1,...,N}\mathbb{E}_\theta[(Z_{s}^{i,N}-Z_{t}^{i,N})^{4}]\le C(t-s)^4$ for all $s\geq t$. Moreover,
 $\ell_N$ is bounded on $\Omega_{N,K}$. Therefore,
\begin{align*}
\Et[(\varUpsilon^1_{t,N,K,u})^2] \le& \frac{C}{K^{2}t^{2}}\sum_{i=1}^{K}\int_{t}^{t+\sqrt{u}t}
\Big( 1+ \mathbb{E}_{\theta}[|\lambda_{s}^{i,N}-\mu \ell_{N}(i)|^2]^\frac{1}{2}\Big)  ds
\leq \frac C{Kt}\Big(1+\frac 1{t^q}+ \frac 1 {\sqrt N} \Big),
\end{align*}
where the last inequality follows from Lemma \ref{lambar}-(ii). This completes the step.

\vip

{\bf Step 2.} Similarly, it holds that, on $\Omega_{N,K}$, 
\begin{align*}
\Et[|\varUpsilon^2_{t,N,K,u}|]
\le &\frac{1}{Kt^{2}}\sum_{i=1}^{K}\int_{t}^{t+\sqrt{u}t}\mathbb{E}_{\theta}[(M_{s}^{i,N}-M_{t}^{i,N})^{4}]^{\frac{1}{2}}\mathbb{E}_{\theta}[|\lambda_{s}^{i,N}-\mu\ell_{N}(i)|^{2}]^{\frac{1}{2}}\\
\le & \frac{C }{Kt} \sum_{i=1}^{K} \int_t^{2t} \mathbb{E}_{\theta}[|\lambda_{s}^{i,N}-\mu\ell_{N}(i)|^{2}]^{\frac{1}{2}}
ds \le \frac{C}{t^{q}}+\frac{C}{\sqrt{N}}.
\end{align*}

{\bf Step 3.} Finally, we  prove  that $\varUpsilon^3_{t,N,K,u} \to \mu^2u/[2(1-\Lambda p)^2]$ in probability
as $(N,K,t)\to(\infty,\infty,\infty)$. Applying  It\^o's  formula and \eqref{ee3}, we write
\begin{align*}
&(M_{s}^{i,N}-M_{t}^{i,N})^{2}\\
=&2\int_{t}^{s}(M_{r-}^{i,N}-M_{t}^{i,N})dM_{r}^{i,N}+Z_{s}^{i,N}-Z_{t}^{i,N}\\
=&2\int_{t}^{s}(M_{r-}^{i,N}-M_{t}^{i,N})dM_{r}^{i,N}+U_{s}^{i,N}-U_{t}^{i,N}+\Et[Z_{s}^{i,N}-Z_{t}^{i,N}-\mu(s-t)\ell_{N}(i)]+\mu(s-t)\ell_{N}(i).
\end{align*}
Consequently,  we decompose $\varUpsilon^3_{t,N,K,u}:=\varUpsilon^{3,1}_{t,N,K,u}+\varUpsilon^{3,2}_{t,N,K,u}+\varUpsilon^{3,3}_{t,N,K,u}+\varUpsilon^{3,4}_{t,N,K,u}$, where 
\begin{align*}
\varUpsilon^{3,1}_{t,N,K,u}:=&\frac{2}{t^2 K} \sum_{i=1}^{K} \mu\ell_{N}(i)\int_{t}^{t+\sqrt{u}t}
\int_{t}^{s}(M_{r-}^{i,N}-M_{t}^{i,N})dM_r^{i,N}ds,\\
\varUpsilon^{3,2}_{t,N,K,u}:=&\frac{1}{t^2 K} \sum_{i=1}^{K} \mu\ell_{N}(i)\int_{t}^{t+\sqrt{u}t}(U_{s}^{i,N}-U_{t}^{i,N}) ds,\\
\varUpsilon^{3,3}_{t,N,K,u}:=&\frac{1}{t^2 K} \sum_{i=1}^{K} \mu\ell_{N}(i)\int_{t}^{t+\sqrt{u}t}\Et[Z_{s}^{i,N}-Z_{t}^{i,N}-\mu(s-t)\ell_{N}(i)] ds,\\
\varUpsilon^{3,4}_{t,N,K,u}:=&\frac{1}{t^2 K} \sum_{i=1}^{K} \mu^2 (\ell_{N}(i))^2 \times \frac{u t^2}2= \frac{\mu^2 u}{2K}
\sum_{i=1}^{K} (\ell_{N}(i))^2.
\end{align*}
First,  noting  that $\cV_\infty^{N,K}=\frac{N}{K}\mu^2\|\bx^{K}_{N}\|_{2}^{2}$, we have
$$
2\varUpsilon^{3,4}_{t,N,K,u}= \mu^2 u (\bar \ell_N^K)^2 + \frac{\mu^2u}K \sum_{i=1}^{K} (\ell_{N}(i)-\bar\ell_N^K)^2
=  \mu^2 u(\bar \ell_N^K)^2 + \frac{\mu^2u}K ||\bx_N^K||_2^2=\mu^2u (\bar \ell_N^K)^2 + \frac{u}N \cV^{N,K}_\infty.
$$
Then,  Lemma \ref{ellp} and Theorem \ref{21}  implies immediately that $\varUpsilon^{3,4}_{t,N,K,u}$ converges to 
$\mu^2u/[2(1-\Lambda p)^2]$ in probability. 

\vip

For the second term, we recall  \eqref{UPro} and write for $s\ge t,$
$$
U^{i,N}_s-U^{i,N}_t =  \sum_{n\geq 0} \int_0^s (\phi^{* n}(s-u)-\phi^{* n}(t-u)) \sum_{j=1}^N A_N^n(i,j)M^{j,N}_udu,
$$ 
so that, by Minkowski's inequality and separating as usual the terms $n=0$ and $n\geq 1$,
\begin{align*}
\Et[|\varUpsilon^{3,2}_{t,N,K,u}|^2]^\frac{1}{2}
\le&\frac{C}{t^2 K}  \int_{t}^{t+\sqrt{u}t}\Et\Big[\Big(\sum_{i=1}^{K}\ell_{N}(i)(U_{s}^{i,N}-U_{t}^{i,N})\Big)^2\Big]^\frac{1}{2} ds\\
 \leq& \frac{C}{t^2 K}  \int_{t}^{t+\sqrt{u}t}\Big\{\Et\Big[\Big(\sum_{i=1}^{K}\ell_{N}(i)(M_{s}^{i,N}-M_{t}^{i,N})\Big)^2\Big]^\frac{1}{2} \\
&+ \sum_{n\geq 1} \int_0^s (\phi^{* n}(s-r)-\phi^{* n}(t-r))
\Et\Big[\Big(\sum_{i=1}^K\sum_{j=1}^N \ell_{N}(i) A_N^n(i,j)M^{j,N}_r\Big)^2\Big]^\frac{1}{2} dr\Big\}ds.
\end{align*}
By (\ref{ee3}), we see that on $\Omega_{N,K}$, for all $t\le s\le 2t,$
\begin{align*}
 \Et\Big[\Big(\sum_{i=1}^{K}\ell_{N}(i)(M_{s}^{i,N}-M_{t}^{i,N})\Big)^2\Big]^\frac{1}{2}
=& \Et\Big[\sum_{i=1}^{K}(\ell_{N}(i))^2(Z_{s}^{i,N}-Z_{t}^{i,N})\Big]^\frac{1}{2}\\
=& \Big\{\sum_{i=1}^{K}(\ell_{N}(i))^2\Et\Big[Z_{s}^{i,N}-Z_{t}^{i,N}\Big]\Big\}^\frac{1}{2}\\
\le& C\sqrt{Kt},
\end{align*}
by Lemma \ref{Zt}-(i) with $r=\infty$, together with the boundedness of $\ell_N(i)$ on $\Omega_{N,K}$.
Next, for $n \geq 1$,
\begin{align*}
    \Et\Big[\Big(\sum_{i=1}^K\sum_{j=1}^N \ell_{N}(i) A_N^n(i,j)M^{j,N}_r\Big)^2\Big]=&\sum_{j=1}^N\Big(\sum_{i=1}^K\ell_{N}(i) A_N^n(i,j)\Big)^2\Et[Z^{j,N}_r]\\
    \le& C\sum_{j=1}^N\Big(\sum_{i=1}^K A_N^n(i,j)\Big)^2\Et[Z^{j,N}_r]\\
     \le& C\sum_{j=1}^N|||I_KA^n_N|||^2_1\Et[Z^{j,N}_r]\\
\le& \frac{CK^2}{N}|||A_N|||^{2n-2}_1r.
\end{align*}
The last line follows from $|||I_KA_N|||_1 \leq C K/N$ on $\Omega_{N,K}$ and by another application of Lemma \ref{Zt}-(i). 
Therefore, for any $u\in [0,1]$ (recalling that $\int_0^\infty \phi^{*n}(u)du=\Lambda^n$),
\begin{align*}
\Et[|\varUpsilon^{3,2}_{t,N,K,u}|^2]^\frac{1}{2}
 \leq& \frac{C}{t^2 K}  \int_{t}^{t+\sqrt{u}t}\Big\{\sqrt{Kt}+\sum_{n\geq 1} \int_0^s (\phi^{* n}(s-r)-\phi^{* n}(t-r))\frac{K}{\sqrt{N}}|||A_N|||^{n-1}_1\sqrt{r} dr\Big\}ds\\
 \le& \frac{C}{t^2 K} \int_{t}^{t+\sqrt{u}t}\Big\{\sqrt{Kt}+\sum_{n\geq 1} \sqrt{s} \frac{K}{\sqrt{N}}\Lambda^n|||A_N|||^{n-1}_1\Big\}ds\le \frac{C}{\sqrt{Kt}}.
\end{align*}
The last inequality uses the fact that on $\Omega_{N,K}$, we have $\Lambda |||A_N|||_1 \leq a<1$.
\vip

For the third term, since $\ell_N(i)$ is  uniformly
bounded on $\Omega_{N,K}$,  and by  Lemma \ref{Zt}-(ii) with $r=1$, we obtain, on $\Omega_{N,K},$
\begin{align*}
\Et[|\varUpsilon^{3,3}_{t,N,K,u}|]
\le&\frac{1}{Kt^{2}}\sum_{i=1}^{K}\int_{t}^{t+\sqrt{u}t}|\ell_N(i)|\Big|\mathbb{E}_{\theta}[Z_{s}^{i,N}-Z_{t}^{i,N}-\mu(s-t)\ell_{N}(i)]\Big|ds\\
\le& \frac{C}{Kt^{2}}\sum_{i=1}^{K}\int_{t}^{t+\sqrt{u}t}\Big|\mathbb{E}_{\theta}[Z_{s}^{i,N}-Z_{t}^{i,N}-\mu(s-t)\ell_{N}(i)]\Big|ds\le \frac{C}{t^q}.
\end{align*}

Finally, we set $\mathbb{N}_u^{t,i,N}:=M_{t+ut}^{i,N}-M_{t}^{i,N}$. 
Then  $\mathbb{N}_u^{t,i,N}$ is a martingale for the filtration  $\mathcal{F}^{N}_{t+ut}$  
with parameter $0\le u\le 1$.  Therefore, by \eqref{ee3},
$$[\mathbb{N}_.^{t,i,N},\mathbb{N}_.^{t,j,N}]_u=\indiq_{\{i=j\}}(Z^{i,N}_{t+ut}-Z^{i,N}_{t}).$$
On $\Omega_{N,K}$, using the change of variables $s=t+at$, we obtain
\begin{align*}
 \Et[(\varUpsilon^{3,1}_{t,N,K,u})^2]
\! =&\frac{2}{K^{2}t^{2}}\mathbb{E}_{\theta}\Big[\Big(\sum_{i=1}^{K}\mu\ell_{N}(i)\int_{0}^{\sqrt{u}}\int_{t}^{t+at}(M_{r-}^{i,N}-M_{t}^{i,N})dM_{r}^{i,N}da\Big)^{2}\Big]\\
 =&\frac{1}{K^{2}t^{2}}\mathbb{E}_{\theta}\Big[\Big(\sum_{i=1}^{K}\mu\ell_{N}(i)\int_{0}^{\sqrt{u}}\int_{0}^{a}\mathbb{N}_{b-}^{t,i,N}d\mathbb{N}_b^{t,i,N}da\Big)^{2}\Big]\\
    =& \frac{\mu^2}{K^{2}t^{2}}\sum_{i=1}^{K}\sum_{i'=1}^{K}\ell_N(i)\ell_N(i')\int_{0}^{\sqrt{u}}\!\int_{0}^{\sqrt{u}}\!\mathbb{E}_{\theta}\Big[\int_{0}^{a}\mathbb{N}_{b-}^{t.i,N}d\mathbb{N}_{b}^{t,i,N}\int_{0}^{a'}\mathbb{N}_{b'-}^{t,i',N}d\mathbb{N}_{b'}^{t,i',N}\Big]dada'\\
    \le&\frac{C}{K^{2}t^{2}}\sum_{i=1}^{K}\int_{0}^{1}\int_{0}^{1}\mathbb{E}_{\theta}\Big[\Big(\int_{0}^{a\land a'}\mathbb{N}_{b-}^{t.i,N}d\mathbb{N}_{b}^{t,i,N}\Big)^2\Big]dada'.
\end{align*}
Using the  Burkholder–Davis–Gundy inequality, the above term  is bounded by 
\begin{align*}
&\frac{C}{K^{2}t^{2}}\sum_{i=1}^{K}\mathbb{E}_{\theta}\Big[\int_{0}^{1}(\mathbb{N}_{b-}^{t.i,N})^2dZ_{t+bt}^{i,N}\Big]\\
    \le&\frac{C}{K^{2}t^{2}}\sum_{i=1}^{K}\mathbb{E}_{\theta}\Big[\Big(\sup_{0\le b\le 1}(\mathbb{N}_{b}^{t.i,N})^2\Big)Z_{2t}^{i,N}\Big].
\end{align*}

Hence, applying the Cauchy-Schwarz inequality  and the   Burkholder–Davis–Gundy  inequality, 
$$ 
\Et[(\varUpsilon^{3,1}_{t,N,K,u})^2]\leq
\frac{C}{K^{2}t^2}\sum_{i=1}^{K}\mathbb{E}_{\theta}\Big[\sup_{0\le b\le 1}(\mathbb{N}_{b}^{t.i,N})^4\Big]^\frac{1}{2}
\Et[(Z_{2t}^{i,N})^2]^\frac{1}{2}\le \frac{C}{K^{2}t^2}\sum_{i=1}^{K}
\mathbb{E}_{\theta}[(Z_{2t}^{i,N})^2]\le \frac{C}{K},
$$
where the last inequality follows from Lemma \ref{lambar}-(iii). This  completes the proof.

\end{proof}

We are now fully equipped to prove the limit for the second estimator.

\begin{proof}[Proof of Theorem \ref{VVNK}]
Recall that we operate with $(N,K,t)\to (\infty,\infty,\infty)$ such that
 $\frac{t\sqrt{K}}{N}(\frac{N}{t^q}+\sqrt{\frac{N}{Kt}})+Ne^{-c_{p,\Lambda}K}\to 0$.
At the beginning of the section, we decomposed the difference as
$$
\cV^{N,K}_t-\cV^{N,K}_\infty=J^{N,K,1}_t+J^{N,K,211}_t+J^{N,K,212}_t+J^{N,K,213}_t+J^{N,K,22}_t+
J^{N,K,23}_t+J^{N,K,3}_t.
$$
As shown in Lemma \ref{Delta1}, all terms except $J^{N,K,211}_t$, when multiplied by
$t\sqrt K / N$, converge to $0$. To complete the proof, it remains to show that under  $(t,N,K)\to (\infty,\infty,\infty)$ and $Ne^{-c_{p,\Lambda}K}\to 0,$
$$
\indiq_{\Omega_{N,K}} \frac{t\sqrt K}N J^{N,K,211}_t=\!
\indiq_{\Omega_{N,K}}\frac{1}{t\sqrt{K}}\sum_{i=1}^{K}\Big\{(U_{2t}^{i,N}-U_{t}^{i,N})^{2}
-\mathbb{E}_{\theta}[(U_{2t}^{i,N}-U_{t}^{i,N})^{2}]\Big\}\stackrel{d}{\longrightarrow}
\mathcal{N}\Big(0, \frac{2\mu^2}{(1-\Lambda p)^2}\Big),
$$
which will establish the desired result.

\vip
We now work on  $\Omega_{N,K}.$ Recalling \eqref{UPro}, 
we write
$$
(U_{2t}^{i,N}-U_{t}^{i,N})^{2}=(M_{2t}^{i,N}-M_{t}^{i,N})^{2}+2T_{t}^{i,N}(M_{2t}^{i,N}-M_{t}^{i,N})+(T_{t}^{i,N})^{2},
$$
where 
$$T_{t}^{i,N}:=\sum_{n\ge 1} \sum_{j=1}^{N}\int_{0}^{2t}\phi^{*n}(2t-s)A_{N}^{n}(i,j)M_{s}^{j,N}ds-\sum_{n\ge 1} \sum_{j=1}^{N}\int_{0}^{t}\phi^{*n}(t-s)A_{N}^{n}(i,j)M_{s}^{j,N}ds.$$
These terms will be treated individually in what follows. Here, as usual, we set $\phi(s)=0$ for $s\le 0$.

\vip

{\bf Step 1.} In this step,  we verify that
$$
\lim_{(N,K,t)\to (\infty,\infty,\infty)} \indiq_{\Omega_{N,K}}\frac1{t\sqrt K}\Et\Big[\sum_{i=1}^K \Big|(T^{i,N}_{t})^2- \Et[(T^{i,N}_{t})^2] \Big| \Big]=0.
$$
By the triangle inequality, it suffices to show that  for all $i=1,\dots,K$, $\Et[(T^{i,N}_{t})^2]\leq C t /N$.
\vip
Setting  $\beta_n(s,t,r)=\phi^{* n}(t-r)-\phi^{* n}(s-r)$, we  rewrite
\begin{equation}\label{frf}
T^{i,N}_{t}=\sum_{n\geq 1} \int_{0}^{2t}\beta_n(t,2t,u)\sum_{j=1}^{N} A_{N}^{n}(i,j) M^{j,N}_{u} du.
\end{equation}
Hence,
\begin{align*}
\mathbb{E}_{\theta}[(T^{i,N}_{t})^2]= \sum_{m,n\geq 1} \int_{0}^{2t}\int_{0}^{2t}\beta_m(t,2t,u)\beta_n(t,2t,v)\sum_{j,k=1}^{N} A_{N}^{m}(i,j) A_{N}^{n}(i,k)
\mathbb{E}_{\theta}[M^{j,N}_{u} M^{k,N}_{v}] dvdu.
\end{align*}
Note  that $\int_{0}^{2t}\beta_n(t,2t,u)\le 2\Lambda^n$ for any $ n\ge 0.$
Using \eqref{ee3} and Lemma \ref{Zt}-(i) with $r=\infty$ on $\Omega_{N,K}$, we obtain $\Et[M^{j,N}_u M^{k,N}_v]=\indiq_{\{j=k\}}\Et[Z^{j,N}_{u\land v}]\leq C (u\land v)$. Therefore, 
\begin{align*}
\mathbb{E}_{\theta}[(T^{i,N}_{t})^2]=& \sum_{m,n\geq 1} \int_{0}^{2t}\int_{0}^{2t}\beta_m(t,2t,u)\beta_n(t,2t,v)\sum_{j,k=1}^{N} A_{N}^{m}(i,j) A_{N}^{n}(i,k)
\mathbb{E}_{\theta}[M^{j,N}_{u} M^{k,N}_{v}] dvdu\\
\le& Ct \sum_{m,n\geq 1} \int_{0}^{2t}\int_{0}^{2t}\beta_m(t,2t,u)\beta_n(t,2t,v)dvdu\sum_{j=1}^{N}
A_{N}^{m}(i,j) A_{N}^{n}(i,j)\\ 
\le& Ct \sum_{m,n\geq 1} \Lambda^{m+n}\sum_{j=1}^{N} A_{N}^{m}(i,j) A_{N}^{n}(i,j) \\ 
\le& Ct\sum_{j=1}^{N}(Q_{N}(i,j)-\boldsymbol{1}_{\{i=j\}})^{2}\le \frac{Ct}{N}.
\end{align*}
The last step follows from \cite[(8)]{A}, which ststes that 
on $\Omega_{N,K}\subset \Omega_{N}^1$,  $\indiq_{\{i=j\}} \leq Q_N(i,j)\leq \indiq_{\{i=j\}} + \Lambda C N^{-1}.$

\vip

{\bf Step 2.} In this step,  we verify that
$$
\lim_{(N,K,t)\to (\infty,\infty,\infty)} \indiq_{\Omega_{N,K}}\frac1{t\sqrt K}\Et\Big[\Big| \sum_{i=1}^K \Big(T^{i,N}_{t}(M^{i,N}_{2t}-M^{i,N}_{t})- \Et[
T^{i,N}_{t}(M^{i,N}_{2t}-M^{i,N}_{t})] \Big)\Big| \Big]=0.
$$
Actually, this follows from the variance estimate on $\Omega_{N,K}$:
$$
x:=\mathbb{V}ar_{\theta}\Big[\sum_{i=1}^{K}(T_{t}^{i,N}(M_{2t}^{i,N}-M_{t}^{i,N}))\Big] \leq C \frac{K t^2}N.
$$
We begin with
\begin{align*}
x=&\mathbb{E}_{\theta}\Big[\sum_{i,j=1}^{K}\Big(T_{t}^{i,N}(M_{2t}^{i,N}-M_{t}^{i,N}) -
\mathbb{E}_{\theta}[T_{t}^{i,N}(M_{2t}^{i,N}-M_{t}^{i,N})]\Big) \\
& \hskip3cm 
\Big(T_{t}^{j,N}(M_{2t}^{j,N}-M_{t}^{j,N}) -
\mathbb{E}_{\theta}[T_{t}^{j,N}(M_{2t}^{j,N}-M_{t}^{j,N})]\Big)\Big].
\end{align*}
Recalling \eqref{frf} and setting $\alpha_{N}(u,t,i,j):=\sum_{n\ge 1}\beta_n(t,2t,u)A_{N}^{n}(i,j),$ we obtain 
\begin{align*}
x&\le \sum_{i,j=1}^{K}\int_{0}^{2t}\int_{0}^{2t}\sum_{k,m=1}^{N}|\alpha_{N}(s,t,i,k)\alpha_{N}(u,t,j,m)|\\
&\hskip3cm|\mathbb{C}ov_{\theta}[(M_{2t}^{i,N}-M_{t}^{i,N})M_{s}^{k,N},(M_{2t}^{j,N}-M_{t}^{j,N})M_{u}^{m,N}]|dsdu.
\end{align*}
Moreover, since  $\int_{0}^{2t}\beta_n(t,2t,u)\le 2\Lambda^n$ for any $ n\ge 0$, we have 
\begin{align*}
\int_{0}^{2t}|\alpha_{N}(s,t,i,k)| ds \leq \sum_{n\geq 1} A_N^n(i,k)  \int_{0}^{2t}|\beta_n(t,2t,s)|ds 
\leq 2 \sum_{n\geq 1} A_N^n(i,k) \Lambda^n \leq 2 (Q_N(i,k)-\indiq_{\{i=k\}}),
\end{align*}
which, as noted at the end of Step 1,  is bounded by $C/N$ according to \cite[(8)]{A}.
In addition, by \cite[Lemma 22]{A}, we have on $\Omega_{N,K}$  that for $s,u\in[0,2t]$,
$$
|\mathbb{C}ov_{\theta}[(M_{2t}^{i,N}-M_{t}^{i,N})M_{s}^{k,N},(M_{2t}^{j,N}-M_{t}^{j,N})M_{u}^{m,N}]|
\le C(\indiq_{\#\{k,i,j,m\}=3}N^{-2}t+\indiq_{\#\{k,i,j,m\}\le 2} t^2).
$$
Therefore,  we conclude that
$$
x\leq \frac C {N^2} \sum_{i,j=1}^K \sum_{k,m=1}^N (\indiq_{\#\{k,i,j,m\}=3}N^{-2}t+\indiq_{\#\{k,i,j,m\}\le 2} t^2)
\leq \frac C {N^2} \Big(N^2 K  \times N^{-2}t + N K \times t^2\Big),
$$
which is bounded by $C K t^2 / N$ as desired.

\vip

{\bf Step 3.}
It remains to show that
\begin{align}\label{Mart}
\indiq_{\Omega_{N,K}}\frac{1}{t\sqrt{K}}\Big[\sum_{i=1}^{K}(M_{2t}^{i,N}-M_{t}^{i,N})^{2}-\sum_{i=1}^{K}\mathbb{E}_{\theta}[(M_{2t}^{i,N}-M_{t}^{i,N})^{2}]\Big]
\end{align}
 converges to a Gaussian random variable with variance $2\mu^2/(1-\Lambda p)^2$.
Applying It\^{o}'s formula, we obtain
$$(M_{2t}^{i,N}-M_{t}^{i,N})^{2}=2\int_{t}^{2t}(M_{s-}^{i,N}-M_{t}^{i,N})dM_{s}^{i,N}+Z_{2t}^{i,N}-Z_{t}^{i,N}.$$
Therefore, \eqref{Mart} becomes
\begin{align*}
\indiq_{\Omega_{N,K}}\frac{1}{t\sqrt{K}}\sum_{i=1}^{K}\left[2\int_{t}^{2t}(M_{s-}^{i,N}-M_{t}^{i,N})dM_{s}^{i,N}+\Big\{(Z_{2t}^{i,N}-Z_{t}^{i,N})-\mathbb{E}_\theta[Z_{2t}^{i,N}-Z_{t}^{i,N}]\Big\}\right].
\end{align*}
From \cite[Lemma 7.2-(ii)]{D}, we know that $\boldsymbol{1}_{\Omega_{N,K}}\mathbb{E}_{\theta}[|\bar{U}_{t}^{N,K}|^{2}]\le\frac{Ct}{K}$. This immediately implies 
 $$\boldsymbol{1}_{\Omega_{N,K}}\frac{1}{t\sqrt{K}}\sum_{i=1}^{K}\{(Z_{2t}^{i,N}-Z_{t}^{i,N})-\mathbb{E}_\theta[Z_{2t}^{i,N}-Z_{t}^{i,N}]\}=\boldsymbol{1}_{\Omega_{N,K}} \frac{\sqrt K}t [\bar{U}_{2t}^{N,K}-
\bar{U}_{t}^{N,K}]\to  0.
$$
Recalling  $N_{u}^{t,i,N}=\int_{t}^{t+\sqrt{u}t}(M_{s-}^{i,N}-M_{t}^{i,N})dM_{s}^{i,N}$ defined in Lemma \ref{Gauss}. 
Since \eqref{ababa} is established in  Lemma \ref{Gauss} and $\frac{1}{t\sqrt{K}}\sum_{i=1}^{K}\int_{t}^{2t}(M_{s-}^{i,N}-M_{t}^{i,N})dM_{s}^{i,N}= \frac{1}{t\sqrt{K}}\sum_{i=1}^{K}  N_{1}^{t,i,N}$,  we thus complete the proof.

\end{proof}

\section{Limit theorem for the third estimator}\label{sec5}
In this section, our goal is to establish the asymptotic behavior of the third estimator $\mathcal{X}_{\Delta,t}^{N,K}$, introduced in Section \ref{TRISC} (see Theorem \ref{corX}).  First,  recall that for $\Delta\ge1$ such that $t/(2\Delta)\in \mathbb{N}^{*}$, 
\begin{align*}
&\mathcal{X}_{\Delta,t}^{N,K}=\mathcal{W}_{\Delta,t}^{N,K}-\frac{N-K}{K}\varepsilon_{t}^{N,K},\quad 
\mathcal{W}_{\Delta,t}^{N,K}=2\mathcal{Z}_{2\Delta,t}^{N,K}-\mathcal{Z}_{\Delta,t}^{N,K}, \\
&\mathcal{Z}^{N,K}_{\Delta,t}=\frac{N}{t}\sum_{a=\frac{t}{\Delta}+1}^{\frac{2t}{\Delta}}(\bar{Z}_{a\Delta}^{N,K}-\bar{Z}_{(a-1)\Delta}^{N,K}-\Delta\varepsilon_{t}^{N,K})^{2},\quad  \varepsilon _{t}^{N,K}=(\bar{Z}_{2t}^{N,K}-\bar{Z}_{t}^{N,K})/t,
\end{align*}
where $\bar{Z}^{N}_{t}=N^{-1}\sum_{i=1}^N Z^{i,N}_t$ and $\bar{Z}^{N,K}_{t}=K^{-1}\sum_{i=1}^K tZ^{i,N}_t.$ 
Moreover,  $Q_N$, $(\ell_{N}(i))_{i=1,\cdots,N}$ and $\bar{\ell}_N^K$ are defined in Section \ref{subimn}. We also introduce $c^K_{N}(j):=\sum_{i=1}^{K}Q_{N}(i,j), \, j=1,\cdots,N$.
\vip
 It was shown in \cite{D} that $\mathcal{X}_{\Delta,t}^{N,K}$ converges to $\cX^{N,K}_{\infty,\infty}=\cW^{N,K}_{\infty,\infty}-\frac{(N-K)\mu}{K}\bar{\ell}_N^K$, where $\cW^{N,K}_{\infty,\infty}=\mu\frac{N}{K^{2}}A^{N,K}_{\infty,\infty}$, $A^{N,K}_{\infty,\infty}=\sum_{j=1}^{N}\Big(\sum_{i=1}^{K}Q_{N}(i,j)\Big)^{2}\ell_{N}(j)$. To establish the  central limit theorem stated in Theorem \ref{corX},  we decompose $\mathcal{X}_{\Delta,t}^{N,K}-\mathcal{X}_{\infty,\infty}^{N,K}$ into
\begin{align*}
&\mathcal{X}_{\Delta,t}^{N,K}-\mathcal{X}_{\infty,\infty}^{N,K}\\
=& (\mathcal{W}_{\Delta,t}^{N,K}-\mathcal{W}_{\infty,\infty}^{N,K}) - \frac{N-K}{K}\Big(\varepsilon_t^{N,K}-\mu\bar{\ell}_N^K\Big)\\
=& -\underbrace{D_{\Delta,t}^{N,K,1}+2D_{2\Delta,t}^{N,K,1}-D_{\Delta,t}^{N,K,2}+2D_{2\Delta,t}^{N,K,2}}_{\textit{small error}}-\underbrace{D_{\Delta,t}^{N,K,3}+2D_{2\Delta,t}^{N,K,3}}_{principle}+\underbrace{D_{\Delta,t}^{N,K,4} - \frac{N-K}{K}\Big(\varepsilon_t^{N,K}-\mu \bar{\ell}_N^K\Big)}_{\textit{small error}},
\end{align*}
where
\begin{align*}
&D_{\Delta,t}^{N,K,1}=\frac{N}{t}\Big\{\sum_{a=\frac{t}{\Delta}+1}^{\frac{2t}{\Delta}}\Big(\bar{Z}_{a\Delta}^{N,K}-\bar{Z}_{(a-1)\Delta}^{N,K}-\Delta\varepsilon_{t}^{N,K}\Big)^{2}-\sum_{a=\frac{t}{\Delta}+1}^{\frac{2t}{\Delta}}\Big(\bar{Z}_{a\Delta}^{N,K}-\bar{Z}_{(a-1)\Delta}^{N,K}-\Delta \mu\bar{\ell}_{N}^{K}\Big)^{2}\Big\}, \\
&D_{\Delta,t}^{N,K,2}=\frac{N}{t}\Big\{\sum_{a=\frac{t}{\Delta}+1}^{\frac{2t}{\Delta}}\Big(\bar{Z}_{a\Delta}^{N,K}-\bar{Z}_{(a-1)\Delta}^{N,K}-\Delta\mu\bar{\ell}_N^{K}\Big)^{2}\\
&\hskip5cm-\sum_{a=\frac{t}{\Delta}+1}^{\frac{2t}{\Delta}}\Big(\bar{Z}_{a\Delta}^{N,K}-\bar{Z}_{(a-1)\Delta}^{N,K}-\mathbb{E}_{\theta}[\bar{Z}_{a\Delta}^{N,K}-\bar{Z}_{(a-1)\Delta}^{N,K}]\Big)^{2}\Big\},
\end{align*}
\begin{align*}
&D_{\Delta,t}^{N,K,3}=\frac{N}{t}\Big\{\sum_{a=\frac{t}{\Delta}+1}^{\frac{2t}{\Delta}}\Big(\bar{Z}_{a\Delta}^{N,K}-\bar{Z}_{(a-1)\Delta}^{N,K}-\mathbb{E}_{\theta}[\bar{Z}_{a\Delta}^{N,K}-\bar{Z}_{(a-1)\Delta}^{N,K}]\Big)^{2}\\
&\hskip3.5cm-\mathbb{E}_{\theta}\Big[\sum_{a=\frac{t}{\Delta}+1}^{\frac{2t}{\Delta}}\Big(\bar{Z}_{a\Delta}^{N,K}-\bar{Z}_{(a-1)\Delta}^{N,K}-\mathbb{E}_{\theta}[\bar{Z}_{a\Delta}^{N,K}-\bar{Z}_{(a-1)\Delta}^{N,K}]\Big)^{2}\Big]\Big\},
\end{align*}
and finally
\begin{align*}
D_{\Delta,t}^{N,K,4}=&\Big\{\frac{2N}{t}\mathbb{E}_{\theta}\Big[\sum_{a=\frac{t}{2\Delta}+1}^{\frac{t}{\Delta}}\Big(\bar{Z}_{2a\Delta}^{N,K}-\bar{Z}_{2(a-1)\Delta}^{N,K}-\mathbb{E}_{\theta}[\bar{Z}_{2a\Delta}^{N,K}-\bar{Z}_{2(a-1)\Delta}^{N,K}]\Big)^{2}\Big]\\&-\frac{N}{t}\mathbb{E}_{\theta}\Big[\sum_{a=\frac{t}{\Delta}+1}^{\frac{2t}{\Delta}}\Big(\bar{Z}_{a\Delta}^{N,K}-\bar{Z}_{(a-1)\Delta}^{N,K}
-\mathbb{E}_{\theta}[\bar{Z}_{a\Delta}^{N,K}-\bar{Z}_{(a-1)\Delta}^{N,K}]\Big)^{2}\Big]-\mathcal{W}^{N,K}_{\infty,\infty}\Big\}.
\end{align*}
The principle term in this decomposition arises from  $D_{\Delta,t}^{N,K,3}$ (see Lemma \ref{D124}), which is approximated by the martingale difference combination $\mathbb{X}^{N,K}_{\Delta,t,v}$ defined in \eqref{mathbX} (see Proposition \ref{D3}).  We then prove in Proposition \ref{D33} that 
$\mathbb{X}^{N,K}_{\Delta,t,v}$ satisfies a central limit theorem, thereby establishing Theorem \ref{corX}.

\vip

The proof of Proposition \ref{D3} relies on Lemmas \ref{covc}, \ref{lem 6.5}, and \ref{mathY}, which together establish the convergence of each component in the decomposition of  $\Big|D_{\Delta,t}^{N,K,3}-\frac{N}{t}\mathbb{X}_{\Delta,t,1}^{N,K}\Big|$. The  proof of Proposition \ref{D33}, on the other hand, proceeds in two steps: we first establish Lemma  \ref{LmathbX}) and then Proposition \ref{main2}, which itself follows from Lemmas \ref{jum} and \ref{bro}. We now state Theorem \ref{corX}, which is proved in Section \ref{pf:corX}.
  
\begin{theorem}\label{corX}
Assume  \eqref{H(q)}  for some $q> 3$, $K\le N$ and $\lim_{(N,K)\to(\infty,\infty)} \frac{K}{N}=\gamma\le 1,$ $\Delta_t= t/(2 \lfloor t^{1-4/(q+1)}\rfloor) \sim t^{4/(q+1)}/2$ \emph{(for $t$ large)}. If  $(N,K,t)\to (\infty,\infty,\infty)$  and $\frac 1{\sqrt K} + \frac NK \sqrt{\frac{\Delta_t}t}+ \frac{N}{t\sqrt K}+Ne^{-c_{p,\Lambda}K} \to 0$,
$$
\omg\frac{K}{N}\sqrt{\frac{t}{\Delta_t}}\Big(\mathcal{X}_{\Delta_t,t}^{N,K}-\cX^{N,K}_{\infty,\infty}\Big)\longrightarrow \mathcal{N}\Big(0,6\mu^2\Big(\frac{1-\gamma}{(1-\Lambda p)}+\frac{\gamma}{(1-\Lambda p)^3}\Big)^2\Big).
$$

\end{theorem}
\subsection{Some small terms of the estimator}
First, we are going to prove  the terms $D_{\Delta,t}^{N,K,1},$ $D_{\Delta,t}^{N,K,2},$ $D_{\Delta,t}^{N,K,4}$ and $\frac{N}{K}|\varepsilon_t^{N,K}-\mu\bar{\ell}_N^K|$ are small. 
\begin{lemma}\label{D124}
Assume  \eqref{H(q)}  for some $q> 3$. If we choose $\Delta_t= t/(2 \lfloor t^{1-4/(q+1)}\rfloor) \sim t^{4/(q+1)}/2$ \emph{(for $t$ large)},   then, If  $(N,K,t)\to (\infty,\infty,\infty)$ and  $\frac 1{\sqrt K} + \frac NK \sqrt{\frac{\Delta_t}t}+ \frac{N}{t\sqrt K}+Ne^{-c_{p,\Lambda}K} \to 0$, we have the convergence in probability that 
$$
\boldsymbol{1}_{\Omega_{N,K}}\frac{K}{N}\sqrt{\frac{t}{\Delta_t}}\Big\{|D_{\Delta_t,t}^{N,K,1}|+|D_{\Delta_t,t}^{N,K,2}|+|D_{\Delta_t,t}^{N,K,4}|+\frac{N}{K}|\varepsilon_t^{N,K}-\mu\bar{\ell}_N^K|\Big\}\longrightarrow 0.
$$
\end{lemma}
\begin{proof}
    It is a directly corollary of \cite[Lemmas 7.3, 9.2, 9.3 and 9.5]{D}.
\end{proof}
 
 Next, for $\Delta\ge1$, we consider the term $D_{\Delta,t}^{N,K,3}$ and  prove that it is close to  $\mathbb{X}^{N,K}_{\Delta,t,v}.$ Recall  $c^K_{N}(j)=\sum_{i=1}^{K}Q_{N}(i,j)$, and  for $0\le v\le 1,$  define
\begin{align}\label{mathbX}
   \mathbb{X}^{N,K}_{\Delta,t,v}:=\sum_{a=[\frac{vt}{\Delta}]+1}^{[\frac{2vt}{\Delta}]}\Big\{(\mathcal{Y}_{(a-1)\Delta,a\Delta}^{N,K})^2-\mathbb{E}_\theta[(\mathcal{Y}_{(a-1)\Delta,a\Delta}^{N,K})^2]\Big\},
\end{align}
where
\begin{align}
\label{Ya1}\mathcal{Y}_{(a-1)\Delta,a\Delta}^{N,K}:=\frac{1}{K}\sum_{j=1}^{N}c_{N}^{K}(j)(M^{j,N}_{a\Delta}-M^{j,N}_{(a-1)\Delta}).
\end{align}

\begin{prop}\label{D3}
Assume  \eqref{H(q)}  for some $q>3$, for $\Delta\ge1$, then we have
\begin{align*}
\frac{K}{N}\sqrt{\frac{t}{\Delta}}\mathbb{E}\Big[\indiq_{\Omega_{N,K}}\Big|D_{\Delta,t}^{N,K,3}-\frac{N}{t}\mathbb{X}_{\Delta,t,1}^{N,K}\Big|\Big]
\le \frac{CK}{N\Delta}+\frac{C\sqrt{K}}{\sqrt{N\Delta}}+\frac{Ct^\frac{3}{4}\sqrt{K}}{\Delta^{1+\frac{q}{2}}}.
\end{align*}
\end{prop}

Before presenting the proof, we require some preparatory steps. Recall $U^{i,N}$ defined in \eqref{UPro}. For $a\in \{t/(2\Delta)+1,...,2t/\Delta\}$, we write  for $i=1,\cdots, K$
$$U^{i,N}_{a\Delta}- U^{i,N}_{(a-1)\Delta}=\Gamma_{(a-1)\Delta,a\Delta}^{i,N}+X_{(a-1)\Delta,a\Delta}^{i,N},$$ 
where
 \begin{align}\label{defGamma}
\Gamma_{(a-1)\Delta,a\Delta}^{i,N}:=&\sum_{n\ge 1}\Big\{\int_{0}^{a\Delta}\phi^{*n}(a\Delta-s)\sum_{j=1}^{N}A_{N}^{n}(i,j)[M^{j,N}_{s}-M^{j,N}_{a\Delta}]ds\\
&\qquad -\int_{0}^{(a-1)\Delta}\phi^{*n}((a-1)\Delta-s)\sum_{j=1}^{N}A_{N}^{n}(i,j)[M^{j,N}_{s}-M^{j,N}_{(a-1)\Delta}]ds\Big\},\notag\\
\label{defXa}
X_{(a-1)\Delta,a\Delta}^{i,N}:=&\sum_{n\ge 0}\sum_{j=1}^{N}\Big\{\int_{0}^{a\Delta}\phi^{*n}(a\Delta-s)ds A_{N}^{n}(i,j)M^{j,N}_{a\Delta}\\
&\qquad -\int_{0}^{(a-1)\Delta}\phi^{*n}((a-1)\Delta-s)ds A_{N}^{n}(i,j)M^{j,N}_{(a-1)\Delta}\Big\}.\notag
 \end{align}
Accordingly, we define 
  $\bar{\Gamma}_{(a-1)\Delta,a\Delta}^{N,K} :=\frac{1}{K} \sum_{i=1}^K \Gamma_{(a-1)\Delta,a\Delta}^{i,N}$,
$\bar{X}_{(a-1)\Delta,a\Delta}^{N,K}:= \frac{1}{K} \sum_{i=1}^K X_{(a-1)\Delta,a\Delta}^{i,N}$.
We decompose the difference into three terms:
\begin{align*}
    \Big| D_{\Delta,t}^{N,K,3}-\frac{N}{t}\mathbb{X}_{\Delta,t,1}^{N,K}\Big|
    =&\frac{N}{t}\Big|\Big\{\sum_{a=\frac{t}{\Delta}+1}^{\frac{2t}{\Delta}}\Big(\bar{Z}_{a\Delta}^{N,K}-\bar{Z}_{(a-1)\Delta}^{N,K}-\mathbb{E}_{\theta}[\bar{Z}_{a\Delta}^{N,K}-\bar{Z}_{(a-1)\Delta}^{N,K}]\Big)^{2}\\
 &-\mathbb{E}_{\theta}\Big[\sum_{a=\frac{t}{\Delta}+1}^{\frac{2t}{\Delta}}\Big(\bar{Z}_{a\Delta}^{N,K}-\bar{Z}_{(a-1)\Delta}^{N,K}-\mathbb{E}_{\theta}[\bar{Z}_{a\Delta}^{N,K}-\bar{Z}_{(a-1)\Delta}^{N,K}]\Big)^{2}\Big]\Big\}\\
&-\sum_{a=[\frac{vt}{\Delta}]+1}^{[\frac{2vt}{\Delta}]}\Big\{(\mathcal{Y}_{(a-1)\Delta,a\Delta}^{N,K})^2-\mathbb{E}_\theta[(\mathcal{Y}_{(a-1)\Delta,a\Delta}^{N,K})^2]\Big\}\Big|,
\end{align*}
which is bounded by 
\begin{align*}
& \frac{2N}{t}\Big[\Big(\Big|\sum_{a=\frac{t}{\Delta}+1}^{\frac{2t}{\Delta}}\Big\{(\bar{\Gamma}^{N,K}_{(a-1)\Delta,a\Delta})^{2}-\mathbb{E}_{\theta}[(\bar{\Gamma}^{N,K}_{(a-1)\Delta,a\Delta})^{2}]\Big\}\Big| \\
&\qquad+\Big|\sum_{a=\frac{t}{\Delta}+1}^{\frac{2t}{\Delta}}\bar{\Gamma}_{(a-1)\Delta,a\Delta}^{N,K}\bar{X}_{(a-1)\Delta,a\Delta}^{N,K}-\mathbb{E}_{\theta}[\sum_{a=\frac{t}{\Delta}+1}^{\frac{2t}{\Delta}}\bar{\Gamma}_{(a-1)\Delta,a\Delta}^{N,K}\bar{X}_{(a-1)\Delta,a\Delta}^{N,K}]\Big|\\
&\qquad+\sum_{a=\frac{t}{\Delta}+1}^{\frac{2t}{\Delta}}\Big|\mathcal{Y}_{(a-1)\Delta,a\Delta}^{N,K}\Big|\Big|\mathcal{Y}_{(a-1)\Delta,a\Delta}^{N,K}-\bar{X}_{(a-1)\Delta,a\Delta}^{N,K}\Big|+\sum_{a=\frac{t}{\Delta}+1}^{\frac{2t}{\Delta}}\Big|\mathcal{Y}_{(a-1)\Delta,a\Delta}^{N,K}-\bar{X}_{(a-1)\Delta,a\Delta}^{N,K}\Big|^{2}\Big)
\Big]\\
+&\frac{2N}{t}\E_\theta\Big[\sum_{a=\frac{t}{\Delta}+1}^{\frac{2t}{\Delta}}\Big|\mathcal{Y}_{(a-1)\Delta,a\Delta}^{N,K}\Big|\Big|\mathcal{Y}_{(a-1)\Delta,a\Delta}^{N,K}-\bar{X}_{(a-1)\Delta,a\Delta}^{N,K}\Big|+\sum_{a=\frac{t}{\Delta}+1}^{\frac{2t}{\Delta}}\Big|\mathcal{Y}_{(a-1)\Delta,a\Delta}^{N,K}-\bar{X}_{(a-1)\Delta,a\Delta}^{N,K}\Big|^2\Big]. 
\end{align*}
We now prove that all the terms above are negligible. The first term is treated in Lemma \ref{lem 6.5}, and the rest are handled in Lemma \ref{mathY}.
\vip
We begin by recalling $\Gamma_{(a-1)\Delta,a\Delta}^{i,N},\, i=1,\cdots,K$ defined in  \eqref{defGamma}, which allows us to rewrite their average $\bar{\Gamma}_{(a-1)\Delta,a\Delta}^{N,K}=\frac{1}{K} \sum_{i=1}^K \Gamma_{(a-1)\Delta,a\Delta}^{i,N}$ as 
$\bar{\Gamma}_{(a-1)\Delta,a\Delta}^{N,K}=C_{a\Delta}^{N,K}+B_{a\Delta}^{N,K}-C_{(a-1)\Delta}^{N,K}-B_{(a-1)\Delta}^{N,K},$
where
\begin{align}
\label{Cdelta}C_{a\Delta}^{N,K}:=&\frac{1}{K}\sum_{i=1}^{K}\sum_{j=1}^{N}\sum_{n\ge 1}\int^{\Delta}_{0}\phi^{*n}(s)A_{N}^{n}(i,j)(M_{(a\Delta-s)}^{j,N}-M^{j,N}_{a\Delta})ds,\\
\label{Bdelta}B_{a\Delta}^{N,K}:=&\frac{1}{K}\sum_{i=1}^{K}\sum_{j=1}^{N}\sum_{n\ge 1}\int^{a\Delta}_{\Delta}\phi^{*n}(s)A_{N}^{n}(i,j)(M^{j,N}_{(a\Delta-s)}-M^{j,N}_{a\Delta})ds.
\end{align}

We now establish the following bounds  for $C_{a\Delta}^{N,K}$ and $B_{a\Delta}^{N,K}$. The proofs are deferred to Appendix \ref{app: prof: 6.4-6}.

\begin{lemma}\label{covc}
 Assume \eqref{H(q)}  for some $q\ge 1$. For $\Delta\ge1$ and $a, b\in\{\frac{t}{\Delta}+1,...,\frac{2t}{\Delta}\}.$ Then a.s. on  $\Omega_{N,K}$,
\vip
$(i)$ $\mathbb{E}_{\theta}[(B_{a\Delta}^{N,K})^{2}]\le \frac{C}{N}\Delta^{1-2q},$
\vip
$(ii)$ $\mathbb{E}_{\theta}[(C_{a\Delta}^{N,K})^{4}]\le \frac{C}{N^{2}},$
\vip
$(iii)$ $\mathbb{C}ov_{\theta}[(C_{a\Delta}^{N,K}-C_{(a-1)\Delta}^{N,K})^{2},(C_{b\Delta}^{N,K}-C_{(b-1)\Delta}^{N,K})^{2}]\le \frac{C\sqrt{t}}{N\Delta^{q-1}},\ |a-b|\ge 4$.
\end{lemma}

We are now in position to  give the estimate of 
$\sum_{a=\frac{t}{\Delta}+1}^{\frac{2t}{\Delta}}\Big\{(\bar{\Gamma}^{N,K}_{(a-1)\Delta,a\Delta})^{2}-\mathbb{E}_{\theta}[(\bar{\Gamma}^{N,K}_{(a-1)\Delta,a\Delta})^{2}]\Big\}.$

\begin{lemma}\label{lem 6.5}
Assume  \eqref{H(q)}  for some $q\ge 1$,  then a.s. on the set $\Omega_{N,K}$, for $\Delta\ge1$,
 $$\frac{K}{N}\sqrt{\frac{t}{\Delta}}\frac{N}{t}\mathbb{E}_{\theta}\Big[\Big|\sum_{a=\frac{t}{\Delta}+1}^{\frac{2t}{\Delta}}\Big\{(\bar{\Gamma}^{N,K}_{(a-1)\Delta,a\Delta})^{2}-\mathbb{E}_{\theta}[(\bar{\Gamma}^{N,K}_{(a-1)\Delta,a\Delta})^{2}]\Big\}\Big|\Big]
\le \frac{CK\sqrt{t}}{N\Delta^{(q+1)}}+\frac{CK}{N\Delta}+\frac{CKt^{\frac{3}{4}}}{\Delta^{(1+\frac{q}{2})}\sqrt{N}}.$$
\end{lemma}
\begin{proof}
    We start from 
$$(\bar{\Gamma}^{N,K}_{(a-1)\Delta,a\Delta})^{2}= (C_{a\Delta}^{N,K}-C_{(a-1)\Delta}^{N,K})^{2}+2(B_{a\Delta}^{N,K}-B_{(a-1)\Delta}^{N,K})(C_{a\Delta}^{N,K}-C_{(a-1)\Delta}^{N,K})+(B_{a\Delta}^{N,K}-B_{(a-1)\Delta}^{N,K})^{2}.$$
Then, 
\begin{align*}
    &\frac{K}{N}\sqrt{\frac{t}{\Delta}}\frac{N}{t}\mathbb{E}_{\theta}\Big[\Big|\sum_{a=\frac{t}{\Delta}+1}^{\frac{2t}{\Delta}}\Big\{(\bar{\Gamma}^{N,K}_{(a-1)\Delta,a\Delta})^{2}-\mathbb{E}_{\theta}[(\bar{\Gamma}^{N,K}_{(a-1)\Delta,a\Delta})^{2}]\Big\}\Big|\Big]\\
\le& \frac{K}{\sqrt{t\Delta}}\Big\{\mathbb{E}_{\theta}\Big[\Big|\sum_{a=\frac{t}{\Delta}+1}^{\frac{2t}{\Delta}}(B_{a\Delta}^{N,K}-B_{(a-1)\Delta}^{N,K})^{2}-\mathbb{E}_{\theta}\Big[\sum_{a=\frac{t}{\Delta}+1}^{\frac{2t}{\Delta}}(B_{a\Delta}^{N,K}-B_{(a-1)\Delta}^{N,K})^{2}\Big]\Big|\Big]\\
    &+\mathbb{E}_{\theta}\Big[\Big|\sum_{a=\frac{t}{\Delta}+1}^{\frac{2t}{\Delta}}(C_{a\Delta}^{N,K}-C_{(a-1)\Delta}^{N,K})^{2}-\mathbb{E}_{\theta}\Big[\sum_{a=\frac{t}{\Delta}+1}^{\frac{2t}{\Delta}}(C_{a\Delta}^{N,K}-C_{(a-1)\Delta}^{N,K})^{2}\Big]\Big|\Big]\\
    &+2\mathbb{E}_{\theta}\Big[\Big|\sum_{a=\frac{t}{\Delta}+1}^{\frac{2t}{\Delta}}(B_{a\Delta}^{N,K}-B_{(a-1)\Delta}^{N,K})(C_{a\Delta}^{N,K}-C_{(a-1)\Delta}^{N,K})\\
&    \hskip5cm -\mathbb{E}_{\theta}\Big[(B_{a\Delta}^{N,K}-B_{(a-1)\Delta}^{N,K})(C_{a\Delta}^{N,K}-C_{(a-1)\Delta}^{N,K})\Big]\Big|\Big]\Big\}.
\end{align*}
Applying Lemma \ref{covc}-(i) gives  
\begin{align*}
&\mathbb{E}_{\theta}\Big[\Big|\sum_{a=\frac{t}{\Delta}+1}^{\frac{2t}{\Delta}}(B_{a\Delta}^{N,K}-B_{(a-1)\Delta}^{N,K})^{2}-\mathbb{E}_{\theta}\Big[\sum_{a=\frac{t}{\Delta}+1}^{\frac{2t}{\Delta}}(B_{a\Delta}^{N,K}-B_{(a-1)\Delta}^{N,K})^{2}\Big]\Big|\Big]\\
\le& 2\mathbb{E}_{\theta}\Big[\sum_{a=\frac{t}{\Delta}+1}^{\frac{2t}{\Delta}}(B_{a\Delta}^{N,K}-B_{(a-1)\Delta}^{N,K})^{2}\Big]\\
\le& \frac{Ct}{N\Delta^{2q}}.
\end{align*}
By  lemma \ref{covc}-(ii)\&(iii), we have 
\begin{align*}
 &\mathbb{E}_{\theta}\Big[\Big|\sum_{a=\frac{t}{\Delta}+1}^{\frac{2t}{\Delta}}(C_{a\Delta}^{N,K}-C_{(a-1)\Delta}^{N,K})^{2}-\mathbb{E}_{\theta}[\sum_{a=\frac{t}{\Delta}+1}^{\frac{2t}{\Delta}}(C_{a\Delta}^{N,K}-C_{(a-1)\Delta}^{N,K})^{2}]\Big|^2\Big]\\
=&\mathbb{V}ar_{\theta}\Big[\sum_{a=\frac{vt}{\Delta}+1}^{\frac{2vt}{\Delta}}(C_{a\Delta}^{N,K}-C_{(a-1)\Delta}^{N,K})^{2}\Big]\\
\le& \sum_{\substack{t/\Delta+1\le a,b \le 2t/\Delta\\ |a-b|\le 3}}
\mathbb{E}_\theta\Big[(C_{a\Delta}^{N,K}-C_{(a-1)\Delta}^{N,K})^{4}\Big]^\frac{1}{2}\mathbb{E}_\theta\Big[(C_{b\Delta}^{N,K}-C_{(b-1)\Delta}^{N,K})^{4}\Big]^\frac{1}{2}\\
& \quad+ \sum_{\substack{t/\Delta+1\le a,b \le 2t/\Delta\\ |a-b|\ge 4}}
\mathbb{C}ov_\theta\Big[(C_{a\Delta}^{N,K}-C_{(a-1)\Delta}^{N,K})^{2},(C_{b\Delta}^{N,K}-C_{(b-1)\Delta}^{N,K})^{2}\Big]\\
\le& C\Big[\frac{t}{\Delta}\frac{1}{N^{2}}+\frac{t^{\frac{5}{2}}}{\Delta^{q+1}N}\Big].
\end{align*}
Moreover,  by the Cauchy-Schwarz inequality and Lemmas  \ref{covc}-(i)\&(ii), we have 
\begin{align*}
   &\mathbb{E}_{\theta}\Big[\Big|\sum_{a=\frac{t}{\Delta}+1}^{\frac{2t}{\Delta}}(B_{a\Delta}^{N,K}-B_{(a-1)\Delta}^{N,K})(C_{a\Delta}^{N,K}-C_{(a-1)\Delta}^{N,K})
   -\mathbb{E}_{\theta}[(B_{a\Delta}^{N,K}-B_{(a-1)\Delta}^{N,K})(C_{a\Delta}^{N,K}-C_{(a-1)\Delta}^{N,K})]\Big|\Big]\\
\le& 4\sum_{a=\frac{t}{\Delta}+1}^{\frac{2t}{\Delta}}\Big\{\mathbb{E}_{\theta}\Big[\Big|B_{a\Delta}^{N,K}C_{a\Delta}^{N,K}\Big|\Big]+\mathbb{E}_{\theta}\Big[\Big|B_{(a-1)\Delta}^{N,K}C_{a\Delta}^{N,K}\Big|\Big]+\mathbb{E}_{\theta}\Big[\Big|B_{a\Delta}^{N,K}C_{(a-1)\Delta}^{N,K}\Big|\Big]\\
& \hskip5cm +\mathbb{E}_{\theta}\Big[\Big|B_{(a-1)\Delta}^{N,K}C_{(a-1)\Delta}^{N,K}\Big|\Big]\Big\}\\
\le& 4\sum_{a=\frac{t}{\Delta}+1}^{\frac{2t}{\Delta}}\Big\{\mathbb{E}_{\theta}\Big[\Big|B_{a\Delta}^{N,K}\Big|^{2}\Big]^{\frac{1}{2}}+\Big|B_{(a-1)\Delta}^{N,K}\Big|^{2}\Big]^{\frac{1}{2}}\Big\}\Big\{\mathbb{E}_{\theta}\Big[\Big|C_{a\Delta}^{N,K}\Big|^{2}\Big]^{\frac{1}{2}}+\mathbb{E}_{\theta}\Big[\Big|C_{(a-1)\Delta}^{N,K}\Big|^{2}\Big]^{\frac{1}{2}}\Big\}\\
\le& \frac{Ct}{\Delta^{q+\frac{1}{2}}}\frac{1}{N}.
\end{align*}
 \vip
Overall, we have
\begin{align*}
    &\frac{K}{N}\sqrt{\frac{t}{\Delta}}\frac{N}{t}\mathbb{E}_{\theta}\Big[\Big|\sum_{a=\frac{t}{\Delta}+1}^{\frac{2t}{\Delta}}\Big\{(\bar{\Gamma}^{N,K}_{(a-1)\Delta,a\Delta})^{2}-\mathbb{E}_{\theta}[(\bar{\Gamma}^{N,K}_{(a-1)\Delta,a\Delta})^{2}]\Big\}\Big|\Big]\\
    \le&  \frac{K}{\sqrt{t\Delta}}\Big\{\frac{1}{N}\frac{Ct}{\Delta^{2q}}+\Big[\frac{t}{\Delta}\frac{1}{N^{2}}+\frac{t^{\frac{5}{2}}}{\Delta^{q+1}N}\Big]^{\frac{1}{2}}+\frac{1}{N}\frac{Ct}{\Delta^{q+\frac{1}{2}}}\Big\}\\
    \le& C\Big\{\frac{K}{N\Delta}+\frac{Kt^{\frac{3}{4}}}{\Delta^{(1+\frac{q}{2})}\sqrt{N}}+\frac{K\sqrt{t}}{N\Delta^{(q+1)}}\Big\}.
\end{align*}
The proof is finished.
\end{proof}

Recall that $c^K_{N}(j)=\sum_{i=1}^{K}Q_{N}(i,j)$ and that  $X_{(a-1)\Delta,a\Delta}^{i,N}$ is defined in \eqref{defXa}.  Our next step is to show that the term
$\bar{X}_{(a-1)\Delta,a\Delta}^{N,K}= \frac{1}{K} \sum_{i=1}^K X_{(a-1)\Delta,a\Delta}^{i,N}$
is close to $\mathcal{Y}_{(a-1)\Delta,a\Delta}^{N,K}$ defined in \eqref{Ya1}.

\begin{lemma}\label{mathY}
Assume  \eqref{H(q)}  for some $q\ge 2$. For $\Delta\ge1$, then we have the following inequalities,
\vip
$(i)$ $\mathbb{E}_{\theta}[(\mathcal{Y}_{(a-1)\Delta,a\Delta}^{N,K}-\bar{X}_{(a-1)\Delta,a\Delta}^{N,K})^{2}] \le \frac{C}{N}\Big[\frac{1}{(a\Delta)^{2q-1}}+\frac{1}{((a-1)\Delta)^{2q-1}}\Big].
$
\vip
$(ii)$ $\frac{K}{\sqrt{t\Delta}}\mathbb{E}_{\theta}\Big[\sum_{a=\frac{t}{\Delta}+1}^{\frac{2t}{\Delta}}\Big|\mathcal{Y}_{(a-1)\Delta,a\Delta}^{N,K}-\bar{X}_{(a-1)\Delta,a\Delta}^{N,K}\Big|^{2}\Big]\le \frac{CK}{N\sqrt{t}\Delta^{2q-\frac{1}{2}}}$.
\vip
$(iii)$ $\mathbb{E}\Big[\indiq_{\Omega_{N,K}}\Big|\mathcal{Y}_{(a-1)\Delta,a\Delta}^{N,K}\Big|^{4}\Big]\le \frac{C\Delta^2}{K^2}$.\
 \vip
$(iv)$ $\frac{K}{\sqrt{\Delta t}}\mathbb{E}\Big[\boldsymbol{1}_{\Omega_{N,K}}\sum_{a=\frac{t}{\Delta}}^{\frac{2t}{\Delta}}\Big|\mathcal{Y}_{(a-1)\Delta,a\Delta}^{N,K}\Big|\Big|\mathcal{Y}_{(a-1)\Delta,a\Delta}^{N,K}-\bar{X}_{(a-1)\Delta,a\Delta}^{N,K}\Big|\Big]\le \frac{C\sqrt{K}}{\Delta^{q-\frac{1}{2}}\sqrt{Nt}}$.
\vip
$(v)$ $\mathbb{E}\Big[\boldsymbol{1}_{\Omega_{N,K}}\frac{K}{N}\sqrt{\frac{t}{\Delta}}\frac{N}{t}\Big|\sum_{a=\frac{t}{\Delta}+1}^{\frac{2t}{\Delta}}\bar{\Gamma}_{(a-1)\Delta,a\Delta}^{N,K}\bar{X}_{(a-1)\Delta,a\Delta}^{N,K}-\mathbb{E}_{\theta}[\sum_{a=\frac{t}{\Delta}+1}^{\frac{2t}{\Delta}}\bar{\Gamma}_{(a-1)\Delta,a\Delta}^{N,K}\bar{X}_{(a-1)\Delta,a\Delta}^{N,K}]\Big|\Big]$
$$\hskip-6cm \le \frac{CK}{N\Delta^q\sqrt{t}}+\frac{C\sqrt{tK}}{\Delta^{q+\frac{1}{2}}\sqrt{N}}+\frac{C\sqrt{K}}{\sqrt{N\Delta}}+\frac{Ct^\frac{3}{4}\sqrt{K}}{\Delta^{1+\frac{q}{2}}}.$$
\end{lemma}
We place the proof of Lemma \ref{mathY}  in Appendix \ref{app: prof: 6.4-6}.  Now, we can give the proof of Proposition \ref{D3}.

\begin{proof}[Proof of Proposition \ref{D3}]
Recalling  \eqref{mathbX} and \eqref{Ya1}, as well as 
Lemmas \ref{covc}, \ref{lem 6.5}, \ref{mathY}-(i),(ii),(iv)\&(v), we have
\begin{align*}
&\frac{K}{N}\sqrt{\frac{t}{\Delta}}\mathbb{E}\Big[\indiq_{\Omega_{N,K}}\Big|D_{\Delta,t}^{N,K,3}-\frac{N}{t}\mathbb{X}_{\Delta,t,1}^{N,K}\Big|\Big]\\
\le& \frac{CK}{N\Delta}+\frac{CKt^{\frac{3}{4}}}{\Delta^{(1+\frac{q}{2})}\sqrt{N}}+\frac{CK\sqrt{t}}{N\Delta^{(q+1)}}+ \frac{CK}{N\Delta^q\sqrt{t}}+\frac{C\sqrt{tK}}{\Delta^{q+\frac{1}{2}}\sqrt{N}}\\
&\qquad+\frac{C\sqrt{K}}{\sqrt{N\Delta}}+\frac{Ct^\frac{3}{4}\sqrt{K}}{\Delta^{1+\frac{q}{2}}}+\frac{C\sqrt{K}}{\Delta^{q-\frac{1}{2}}\sqrt{Nt}}+\frac{CK}{N\sqrt{t}\Delta^{2q-\frac{1}{2}}}\\
\le& \frac{CK}{N\Delta}+\frac{CKt^{\frac{3}{4}}}{\Delta^{1+\frac{q}{2}}\sqrt{N}}+\frac{C\sqrt{K}}{\sqrt{N\Delta}}+\frac{Ct^\frac{3}{4}\sqrt{K}}{\Delta^{1+\frac{q}{2}}}\\
\le&\frac{CK}{N\Delta}+\frac{C\sqrt{K}}{\sqrt{N\Delta}}+\frac{Ct^\frac{3}{4}\sqrt{K}}{\Delta^{1+\frac{q}{2}}},
\end{align*}
which completes the proof.
\end{proof}

\subsection{The convergence of $\mathbb{X}_{\Delta_t,t,v}^{N,K}$ }\label{PD3}
Recall the process $\mathbb{X}_{\Delta_t,t,v}^{N,K}$ defined in \eqref{mathbX}. The goal of this subsection is to prove the following proposition, which states that the normalized version of 
$(\mathbb{X}_{\Delta_t,t,v}^{N,K})_{v\ge 0}$  converges to a Gaussian process.

\begin{prop}\label{D33}
Assume  \eqref{H(q)}  for some $q> 3$. For $t\geq 1$, set $\Delta_t= t/(2 \lfloor t^{1-4/(q+1)}\rfloor) \sim t^{4/(q+1)}/2$
\emph{(as $t\to \infty$)}. When $(N,K,t)\to (\infty,\infty,\infty)$ such that
$\frac{K}{N} \to \gamma\le 1$ and 
$\frac 1{\sqrt K} + \frac NK \sqrt{\frac{\Delta_t}t}+ \frac{N}{t\sqrt K}+Ne^{-c_{p,\Lambda}K} \to 0$, it holds that
$$\left(\frac{K}{N}\sqrt{\frac{t}{\Delta_t}}\frac{N}{t}\mathbb{X}_{\Delta_t,t,v}^{N,K}\right)_{v\ge 0}\stackrel{d}{\longrightarrow}\sqrt{2}\mu\Big(\frac{1-\gamma}{(1-\Lambda p)}+\frac{\gamma}{(1-\Lambda p)^3}\Big)(B_{2v}-B_v)_{v\ge 0},$$
for the Skorohod topology, where $B$ is a standard Brownian motion.
\end{prop}

Recalling the definition of the martingale summation $\mathcal{Y}_{(a-1)\Delta,a\Delta}^{N,K}$ defined in \eqref{Ya1}, we apply It\^o's formula to obtain
\begin{align}\label{KYQZ}
    (K\mathcal{Y}_{(a-1)\Delta,a\Delta}^{N,K})^2= Q_{a,N,K}+\sum_{j=1}^{N}\Big(c_{N}^{K}(j)\Big)^{2}\Big(Z_{a\Delta}^{j,N}-Z^{j,N}_{(a-1)\Delta}\Big),
    \end{align}
    where $$Q_{a,N,K}=2\int_{(a-1)\Delta}^{a\Delta}\sum_{j=1}^{N}c_{N}^{K}(j)(M^{j,N}_{s-}-M^{j,N}_{(a-1)\Delta})\sum_{j=1}^{N}c_{N}^{K}(j)dM^{j,N}_{s}.$$
We can then decompose $$\frac{K}{\sqrt{\Delta t}}\mathbb{X}^{N,K}_{\Delta,t,v}=\mathcal{L}^{t,\Delta}_{N,K}(2v)-\mathcal{L}^{t,\Delta}_{N,K}(v)+rest,$$ where 
$$\mathcal{L}^{t,\Delta}_{N,K}(u):=\frac{1}{K\sqrt{\Delta t}}\sum_{a=1}^{[\frac{t}{\Delta}u]}Q_{a,N,K},\quad
\text{ for $0\le u\le 2$}.$$ 
The proof of Proposition \ref{D33} proceeds in two steps: we first show that the $``rest"$ term is negligible (see Lemma  \ref{LmathbX}), and then prove that  $\mathcal{L}^{t,\Delta}_{N,K}$ satisfies a central limit theorem (see Proposition \ref{main2}).
\begin{lemma}\label{LmathbX}
Assume  \eqref{H(q)}  for some $q\ge 1$, then for $0\le v\le 1,$ and $\Delta\ge1$,
$$
\mathbb{E}\Big[\omg\Big|\mathcal{L}^{t,\Delta}_{N,K}(2v)-\mathcal{L}^{t,\Delta}_{N,K}(v)-\frac{K}{\sqrt{\Delta t}}\mathbb{X}^{N,K}_{\Delta,t,v}\Big|\Big]\le \frac{C}{\sqrt{\Delta t}}.
$$
\end{lemma}

\begin{proof}
Noting that $\mathbb{E}[Q_{a,N,K}|\mathcal{F}_{(a-1)\Delta}]=0$, the process  $\mathcal{L}^{t,\Delta}_{N,K}(u)$ is a martingale with respect to  the filtration $\mathcal{F}_{[\frac{t}{\Delta}u]\Delta}.$ 
In view of equality (\ref{KYQZ}) and definition of $\mathbb{X}^{N,K}_{\Delta,t,v}$ in \eqref{mathbX}, it remains to verify that
\begin{align*}
    \mathbb{E}_{\theta}\Big[\Big|\frac{1}{K\sqrt{\Delta t}}\sum_{a=[\frac{vt}{\Delta}]+1}^{[\frac{2vt}{\Delta}]}\sum_{j=1}^{N}\Big(c_{N}^{K}(j)\Big)^{2}\Big(Z_{a\Delta}^{j,N}-Z_{(a-1)\Delta}^{j,N}-\mathbb{E}_{\theta}[Z_{a\Delta}^{j,N}-Z_{(a-1)\Delta}^{j,N}]\Big)\Big|\Big]\le \frac{C}{\sqrt{\Delta t}}.
\end{align*}
We first decompose
\begin{align*}
&\frac{1}{K\sqrt{\Delta t}}\sum_{a=[\frac{vt}{\Delta}]+1}^{[\frac{2vt}{\Delta}]}\sum_{j=1}^{N}\Big(c_{N}^{K}(j)\Big)^{2}\Big(Z_{a\Delta}^{j,N}-Z_{(a-1)\Delta}^{j,N}-\mathbb{E}_{\theta}[Z_{a\Delta}^{j,N}-Z_{(a-1)\Delta}^{j,N}]\Big)\\
&=\frac{1}{K\sqrt{\Delta t}}\Big\{\sum_{j=1}^{N} \Big(c_{N}^{K}(j)\Big)^{2}\Big(Z_{2vt}^{j,N}-Z_{vt}^{j,N}-\mu vt\ell_{N}(j)\Big)\\
&\hskip2cm+\sum_{j=1}^{N} \mathbb{E}_{\theta}\Big[\Big(c_{N}^{K}(j)\Big)^{2}\Big(\mu vt\ell_{N}(j)-Z_{2vt}^{j,N}+Z_{vt}^{j,N}\Big)\Big]\Big\}.
\end{align*}
From \cite[(8)]{A}, on the event $ \Omega_{N,K}\subset \Omega_N^1$,  we  have
$\boldsymbol{1}_{\{i=j\}}\le Q_{N}(i,j)\le \boldsymbol{1}_{\{i=j\}}+\frac{C}{N}$ for all $i,j=1,...,N.$
Recalling that $c_{N}^{K}(i)=\sum_{j=1}^{K}Q_{N}(j,i)$, $i=1,\cdots,N$, we obtain
\begin{align}\label{CNK}
1\le c_{N}^{K}(i)\le 1+\frac{CK}{N}\hbox{  when } 1\le i\le K \hbox{ and }
0\le c_{N}^{K}(i)\le \frac{CK}{N} \hbox{ when } (K+1)\le i\le N.
\end{align}
Moreover, by \cite[Lemma 16-(ii)]{A}, we have
$$
\max_{j=1,...,N}\mathbb{E}_{\theta}\Big[ \Big|\Big(Z_{2vt}^{j,N}-Z_{vt}^{j,N}-\mu v t\ell_{N}(j)\Big)\Big|\Big]\le C.
$$
Consequently,
\begin{align*}
    &\mathbb{E}_{\theta}\Big[\frac{1}{K\sqrt{\Delta t}}\sum_{j=1}^{N} \Big|\Big(c_{N}^{K}(j)\Big)^{2}\Big(Z_{2vt}^{j,N}-Z_{vt}^{j,N}-\mu v t\ell_{N}(j)\Big)\Big|\Big]\\
   & \le \frac{C}{K\sqrt{\Delta t}}\Big[\sum_{j=1}^{K} \Big(c_{N}^{K}(j)\Big)^{2}+\sum_{j=K}^{N} \Big(c_{N}^{K}(j)\Big)^{2}\Big]
    \le \frac{C}{\sqrt{\Delta t}}.
\end{align*}
Therefore,
\begin{align*}
    \mathbb{E}_{\theta}\Big[\Big|\frac{1}{K\sqrt{\Delta t}}\sum_{a=[\frac{vt}{\Delta}]+1}^{[\frac{2vt}{\Delta}]}\sum_{j=1}^{N}\Big(c_{N}^{K}(j)\Big)^{2}\Big(Z_{a\Delta}^{j,N}-Z_{(a-1)\Delta}^{j,N}-\mathbb{E}_{\theta}[Z_{a\Delta}^{j,N}-Z_{(a-1)\Delta}^{j,N}]\Big)\Big|\Big]\le\frac{C}{\sqrt{\Delta t}},
\end{align*}
which ends  the proof.
\end{proof}

We now turn to prove the convergence of $\mathcal{L}^{t,\Delta}_{N,K}(u)$ to a Brownian motion.  
\begin{prop}\label{main2}
Assume $K\le N$. For $t\geq 1$, define $\Delta_t:= t/(2 \lfloor t^{1-4/(q+1)}\rfloor) \sim t^{4/(q+1)}/2$ \emph{(for $t$ large)}.
Let   $(N,K,t)\to (\infty,\infty,\infty)$ satisfy 
$\frac 1{\sqrt K} + \frac NK \sqrt{\frac{\Delta_t}t}+ \frac{N}{t\sqrt K}+Ne^{-c_{p,\Lambda}K} \to 0$ and
$\frac{K}{N}\to \gamma\le 1$. Then, in the Skorokhod topology,
$$(\mathcal{L}^{t,\Delta_t}_{N,K}(u))_{u\ge 0} \stackrel{d}{\longrightarrow}\mu\sqrt{2}\Big(\frac{1-\gamma}{(1-\Lambda p)}+\frac{\gamma}{(1-\Lambda p)^3}\Big)(B_{u})_{u\ge 0},$$
where $B$ is a standard Brownian motion.
\end{prop}

According to Lemma \ref{Agamma}, to prove Proposition \ref{main2},  it suffices, by \cite[Theorem VIII.3.8]{B}, to verify the following two conditions:\\

1. The jump size of $\mathcal{L}^{t,\Delta}_{N,K}(u)$ is not large (Lemma \ref{jum}).
\vip
2. Its quadratic variation increases linearly in time (Lemma \ref{bro}).
\vip
The first condition  is addressed by  the following lemma.

\begin{lemma}\label{jum}
Assume  \eqref{H(q)}  for some $q\ge 1$,  and for $\Delta\ge1$,
$$\boldsymbol{1}_{\Omega_{N,K}}\mathbb{E}_{\theta}\Big[\sup_{0\le u\le 2}\Big|\mathcal{L}^{t,\Delta}_{N,K}(u)-\mathcal{L}^{t,\Delta}_{N,K}(u-)\Big|\Big]\le C\Big(\frac{\Delta}{t}\Big)^{\frac{1}{4}}.$$
\end{lemma}

\begin{proof}
First, note that $\mathcal{L}^{t,\Delta}_{N,K}(u)$ is a pure jump process. Hence, at a jump time we have 
$$\mathcal{L}^{t,\Delta}_{N,K}(u)-\mathcal{L}^{t,\Delta}_{N,K}(u-)=\frac{1}{K\sqrt{\Delta t}}Q_{[\frac{t}{\Delta}u],N,K}.$$
Next, we are going to show that $$\mathbb{E}_{\theta}[(Q_{a,N,K})^{4}]\le   C(K\Delta)^4.$$
For $0\le u\le 1,$ we define
\begin{align*}
q_{a,N,K}(u):=\int_{(a-1)\Delta}^{[(a-1)+u]\Delta}\sum_{j=1}^{N}c_{N}^{K}(j)(M^{j,N}_{s}-M^{j,N}_{(a-1)\Delta})d\Big(\sum_{j=1}^{N}c_{N}^{K}(j)M^{j,N}_{s}\Big).
\end{align*}
Clearly $q_{a,N,K}(1)=\frac{1}{2}Q_{a,N,K}$ and its quadratic variation 
\begin{align*}
[q_{a,N,K}(.),q_{a,N,K}(.)]_{u}=\int_{(a-1)\Delta}^{(a-1+u)\Delta}\Big(\sum_{j=1}^{N}c_{N}^{K}(j)(M_{s}^{j,N}-M_{(a-1)\Delta}^{j,N})\Big)^{2}\sum_{j=1}^{N}\Big(c_{N}^{K}(j)\Big)^{2}d Z_{s}^{j,N}.
\end{align*}
Consequently, recalling $\bar{Z}^{N}_{t}=N^{-1}\sum_{i=1}^N Z^{i,N}_t$, $\bar{Z}^{N,K}_{t}=K^{-1}\sum_{i=1}^K Z^{i,N}_t$ and using \eqref{CNK}, we obtain 
$$
\sum_{j=1}^{N}\Big(c_{N}^{K}(j)\Big)^{2}dZ_{s}^{j,N}\le C\Big(Kd\bar{Z}^{N,K}_s+\frac{K^2}{N}d\bar{Z}^{N}_s\Big).
$$
On $\Omega_{N,K}$,  using the Burkholder–Davis–Gundy inequality, we obtain
\begin{align*}
&\mathbb{E}_{\theta}[(q_{a,N,K}(v))^{4}]\le    4\mathbb{E}_{\theta}[([q_{a,N,K}(.),q_{a,N,K}(.)]_{v})^{2}]
 \\
=&4\mathbb{E}_{\theta}\Big[\Big(\int_{(a-1)\Delta}^{(a-1+v)\Delta}\Big(\sum_{j=1}^{N}c_{N}^{K}(j)(M_{s}^{j,N}-M_{(a-1)\Delta}^{j,N})\Big)^{2}d\sum_{j=1}^{N}\Big(c_{N}^{K}(j)\Big)^{2}Z_{s}^{j,N}\Big)^{2}\Big]\\
\le& 4\mathbb{E}_{\theta}\Big[\sup_{0\le s\le v\Delta}\Big(\sum_{j=1}^Nc_N^K(j)(M_{(a-1)\Delta+s}^{j,N}-M_{(a-1)\Delta}^{j,N})\Big)^4\Big(\sum_{j=1}^{N}\Big(c_{N}^{K}(j)\Big)^{2}(Z_{(a-1+v)\Delta}^{j,N}-Z_{(a-1)\Delta}^{j,N})\Big)^{2}\Big]\\
\le& 8\mathbb{E}_{\theta}\Big[\sup_{0\le s\le v\Delta}\Big(\sum_{j=1}^Nc_N^K(j)(M_{(a-1)\Delta+s}^{j,N}-M_{(a-1)\Delta}^{j,N})\Big)^8+\Big(\sum_{j=1}^{N}\Big(c_{N}^{K}(j)\Big)^{2}(Z_{(a-1+v)\Delta}^{j,N}-Z_{(a-1)\Delta}^{j,N})\Big)^{4}\Big]\\
\le& C\mathbb{E}_{\theta}\Big[\Big(\sum_{j=1}^{N}\Big(c_{N}^{K}(j)\Big)^{2}(Z_{(a-1+v)\Delta}^{j,N}-Z_{(a-1)\Delta}^{j,N})\Big)^{4}\Big]\\
\le& C\mathbb{E}_{\theta}\Big[\Big(K(\bar{Z}_{(a-1+v)\Delta}^{N,K}-\bar{Z}_{(a-1)\Delta}^{N,K})+\frac{K^2}{N}(\bar{Z}_{(a-1+v)\Delta}^{N}-\bar{Z}_{(a-1)\Delta}^{N})\Big)^4\Big]\\
\le& C(Kv\Delta)^4.
\end{align*}
Here the fourth step uses the Burkholder–Davis–Gundy inequality, while the last bound follows from Lemma \ref{lambar}-(iv).
 
\vip
Finally, applying the Cauchy–Schwarz inequality at the third step,  we  obtain
\begin{align*}
&\mathbb{E}_{\theta}\Big[\sup_{0\le u\le 2}\Big|\mathcal{L}^{t,\Delta}_{N,K}(u)-\mathcal{L}^{t,\Delta}_{N,K}(u-)\Big|\Big]\\=&\frac{1}{K\sqrt{\Delta t}}\mathbb{E}_{\theta}\Big[\sup_{\{i=1...[\frac{2t}{\Delta}]\}}|Q_{i,N,K}|\Big]
\le \frac{1}{K\sqrt{\Delta t}}\mathbb{E}_{\theta}\Big[\Big(\sum_{i=1}^{[\frac{2t}{\Delta}]}|Q_{i,N,K}|^{4}\Big)^{\frac{1}{4}}\Big]\\
&\le \frac{1}{K\sqrt{\Delta t}}\mathbb{E}_{\theta}\Big[\sum_{i=1}^{[\frac{2t}{\Delta}]}|Q_{i,N,K}|^{4}\Big]^{\frac{1}{4}}
\le C\Big(\frac{\Delta}{t}\Big)^{\frac{1}{4}}.
\end{align*}
This completes the proof.
\end{proof}

\begin{lemma}\label{bro}
We assume  \eqref{H(q)}  for some $q\ge 1$. For $0\le u\le 2,$ and $\Delta\ge1$,
\begin{align*}
\mathbb{E}\Big[\boldsymbol{1}_{\Omega_{N,K}}\Big|[\mathcal{L}^{t,\Delta}_{N,K}(.),\mathcal{L}^{t,\Delta}_{N,K}(.)]_{u}-\frac{2u(\mu A^{N,K}_{\infty,\infty})^{2}}{K^{2}}\Big|\Big]\le C\Big(\frac{1}{K\Delta}+\frac{1}{\sqrt{N}}+\Big(\frac{K\sqrt{t}}{\Delta^{q+1}}\Big)^\frac{1}{2}+\sqrt{\frac{\Delta}{t}}\Big).
\end{align*}
\end{lemma}

\begin{proof}
For $s\geq 0$, we introduce $\phi_{t,\Delta}(s)=a\Delta$, where $a$ is the unique integer such that
$a\Delta\le s< (a+1)\Delta$. Then we have
$$\mathcal{L}^{t,\Delta}_{N,K}(u)=\frac{2}{K\sqrt{\Delta t}}\int_{0}^{tu}\sum_{j=1}^{N}c_{N}^{K}(j)(M_{s}^{j,N}-M_{\phi_{t,\Delta}(s)}^{j,N}) \sum_{i=1}^{N}c_{N}^{K}(i)dM_{s}^{i,N}. $$
Noting that $Z_t^{i,N}=M_t^{i,N}+\int_0^t \lambda_s^{i,N}ds$ for $i=1,\dots,N$, we have 
\begin{align*}
[\mathcal{L}^{t,\Delta}_{N,K}(.),\mathcal{L}^{t,\Delta}_{N,K}(.)]_{u}
=&\frac{4}{K^{2}\Delta t}\int_{0}^{tu}\Big(\sum_{j=1}^{N}c_{N}^{K}(j)(M_{s}^{j,N}-M^{j,N}_{\phi_{t,\Delta}(s)})\Big)^{2}\sum_{i=1}^{N}\Big(c_{N}^{K}(i)\Big)^{2}dZ_{s}^{i,N}\\
=& 4(\mathcal{A}^{u,1}_{N,K}+\mathcal{A}^{u,2}_{N,K}+\mathcal{A}^{u,3}_{N,K}),
\end{align*}
where
\begin{align*}
    \mathcal{A}^{u,1}_{N,K}:=&\frac{1}{K^{2}\Delta t}\int_{0}^{tu}\Big(\sum_{j=1}^{N}c_{N}^{K}(j)(M_{s}^{j,N}-M^{j,N}_{\phi_{t,\Delta}(s)})\Big)^{2} \sum_{i=1}^{N}\Big(c_{N}^{K}(i)\Big)^{2}dM_{s}^{i,N},\\
    \mathcal{A}^{u,2}_{N,K}:=&\frac{1}{K^{2}\Delta t}\int_{0}^{tu}\Big(\sum_{j=1}^{N}c_{N}^{K}(j)(M_{s}^{j,N}-M_{\phi_{t,\Delta}(s)}^{j,N})\Big)^{2}\sum_{i=1}^{N}\Big(c_{N}^{K}(i)\Big)^{2}\Big(\lambda_{s}^{i,N}-\mu\ell_{N}(i)\Big)ds,\\
    \mathcal{A}^{u,3}_{N,K}:=&\Big[\mu\sum_{i=1}^{N}\Big(c_{N}^{K}(i)\Big)^{2}\ell_{N}(i)\Big]\frac{1}{K^{2}\Delta t}\int_{0}^{tu}\Big(\sum_{j=1}^{N}c_{N}^{K}(j)(M_{s}^{j,N}-M_{\phi_{t,\Delta}(s)}^{j,N})\Big)^{2}ds.
\end{align*}

First,  we derive an upper-bound for $\mathcal{A}^{u,1}_{N,K}$. Recalling (\ref{ee3}), we obtain 
\begin{align*}
    &\mathbb{E}_{\theta}\Big[\Big(\mathcal{A}^{u,1}_{N,K}\Big)^2\Big]
     =\frac{1}{K^{4}(\Delta t)^{2}}\mathbb{E}_{\theta}\Big[\int_{0}^{tu}\Big(\sum_{j=1}^{N}c_{N}^{K}(j)(M_{s}^{j,N}-M^{j,N}_{\phi_{t,\Delta}(s)})\Big)^{4} \sum_{i=1}^{N}\Big(c_{N}^{K}(i)\Big)^{4}dZ_{s}^{i,N}\Big]\\
    \le& \frac{C}{K^{4}(\Delta t)^{2}}\mathbb{E}_{\theta}\Big[\max_{a=1,...,[\frac{2t}{\Delta}]}\sup_{0\le s\le \Delta}\Big(\sum_{j=1}^Nc_N^K(j)(M_{(a-1)\Delta+s}^{j,N}-M_{(a-1)\Delta}^{j,N})\Big)^4\sum_{j=1}^{N}\Big(c_{N}^{K}(j)\Big)^{4}Z_{2t}^{j,N}\Big]\\
    \le& \frac{C}{K^{4}(\Delta t)^{2}}\mathbb{E}_{\theta}\Big[\max_{a=1,...,[\frac{2t}{\Delta}]}\sup_{0\le s\le \Delta}\Big(\sum_{j=1}^Nc_N^K(j)(M_{(a-1)\Delta+s}^{j,N}-M_{(a-1)\Delta}^{j,N})\Big)^8+\Big(\sum_{j=1}^{N}\Big(c_{N}^{K}(j)\Big)^{4}Z_{2t}^{j,N}\Big)^2\Big].
    \end{align*}
Using Doob's inequality, Lemma \ref{lambar}-(iii) together with \eqref{CNK}, the last expression is bounded by
\begin{align*}     
     & \frac{C}{K^{4}(\Delta t)^{2}}\mathbb{E}_{\theta}\Big[\sum^{[\frac{2t}{\Delta}]}_{a=1}\sup_{0\le s\le \Delta}\Big(\sum_{j=1}^Nc_N^K(j)(M_{(a-1)\Delta+s}^{j,N}-M_{(a-1)\Delta}^{j,N})\Big)^8\Big]+\frac{C}{K^2\Delta^2}\\
   &\le  \frac{C}{K^{4}(\Delta t)^{2}}\mathbb{E}_{\theta}\Big[\sum^{[\frac{2t}{\Delta}]}_{a=1}\Big(\sum_{j=1}^N(c_N^K(j))^2(Z_{a\Delta}^{j,N}-Z_{(a-1)\Delta}^{j,N})\Big)^4\Big]+\frac{C}{K^2\Delta^2}\\
   & \le \frac{C\Delta}{t}+\frac{C}{K^2\Delta^2}.
\end{align*}

\vip
For the second term, applying the Cauchy-Schwarz inequality,  the Burkholder–Davis–Gundy inequality and  \eqref{CNK}  yields that   on  $\Omega_{N,K}$,
\begin{align*}
    &\mathbb{E}_{\theta}\Big[\Big|\mathcal{A}^{u,2}_{N,K}\Big|\Big]\\
    \le& \frac{1}{K^{2}\Delta t}\int_{0}^{tu}\mathbb{E}_{\theta}\Big[\Big(\sum_{j=1}^{N}c_{N}^{K}(j)\Big(M_{s}^{j,N}-M_{\phi_{t,\Delta}(s)}^{j,N}\Big)\Big)^{4}\Big]^{\frac{1}{2}}\mathbb{E}_{\theta}\Big[\Big|\sum_{i=1}^{N}\Big(c_{N}^{K}(i)\Big)^{2}\Big(\lambda_{s}^{i,N}-\mu\ell_{N}(i)\Big)\Big|^{2}\Big]^{\frac{1}{2}}ds\\
    \le& \frac{1}{K^{2}\Delta t}\int_{0}^{tu}\Et\Big[\Big(\sum_{j=1}^{N}\Big(c_{N}^{K}(j)\Big)^2\Big(Z_{s}^{j,N}-Z_{\phi_{t,\Delta}(s)}^{j,N}\Big)\Big)^{2}\Big]^{\frac{1}{2}}\mathbb{E}_{\theta}\Big[\Big|\sum_{i=1}^{N}\Big(c_{N}^{K}(i)\Big)^{2}\Big(\lambda_{s}^{i,N}-\mu\ell_{N}(i)\Big)\Big|^{2}\Big]^{\frac{1}{2}}ds\\
    \le& \frac{C}{K^{2}\Delta t}\int_{0}^{tu}\Et\Big[\Big(K\Big(\bar{Z}_{s}^{N,K}-\bar{Z}_{\phi_{t,\Delta}(s)}^{N,K}\Big)
    +\frac{K^2}{N}\Big(\bar{Z}_{s}^{N}-\bar{Z}_{\phi_{t,\Delta}(s)}^{N}\Big)\Big)^{2}\Big]^{\frac{1}{2}}\\ &\qquad\qquad\qquad\times\mathbb{E}_{\theta}\Big[\Big|\sum_{i=1}^{N}\Big(c_{N}^{K}(i)\Big)^{2}\Big(\lambda_{s}^{i,N}-\mu\ell_{N}(i)\Big)\Big|^{2}\Big]^{\frac{1}{2}}ds.
\end{align*}
Now, applying  Lemma \ref{Zt}-(ii), which states that  on $\Omega_{N,K}$ 
$$\max_{i=1,\dots,N}\Et[Z^{i,N}_{t}-Z^{i,N}_{s}]\le C(t-s),$$ together with \eqref{CNK}  and Lemma \ref{lambar}-(ii),  we  further obtain 
 \begin{align*}
  \mathbb{E}_{\theta}\Big[\Big|\mathcal{A}^{u,2}_{N,K}\Big|\Big]
    \le& \frac{C}{K t}\int_{0}^{tu}\sum_{i=1}^{N}\Big(c_{N}^{K}(i)\Big)^{2}\mathbb{E}_{\theta}\Big[\Big|(\lambda_{s}^{i,N}-\mu\ell_{N}(i))\Big|^{2}\Big]^{\frac{1}{2}}ds\\
    \le& \frac{C}{Kt}\int_{0}^{tu}\sum_{i=1}^{K}\mathbb{E}_{\theta}\Big[\Big|\lambda_{s}^{i,N}-\mu\ell_{N}(i)\Big|^{2}\Big]^{\frac{1}{2}}ds+\frac{C}{Nt}\int_{0}^{tu}\sum_{i=1}^{N}\mathbb{E}_{\theta}\Big[\Big|\lambda_{s}^{i,N}-\mu\ell_{N}(i)\Big|^{2}\Big]^{\frac{1}{2}}ds\\
     \le& \frac{C}{\sqrt{N}}+\frac{C}{t^{q}}.
\end{align*}
For the third term, we first recall $A^{N,K}_{\infty,\infty}=\sum_{i=1}^{N}\Big(c_{N}^{K}(i)\Big)^{2}\ell_{N}(i)$ with $c_{N}^{K}(i)=\sum_{j=1}^{K}Q_{N}(j,i)$. Then we rewrite
\begin{align*}
    &\mathbb{V}ar_{\theta}(\mathcal{A}^{u,3}_{N,K})\\
    &=\frac{(\mu A^{N,K}_{\infty,\infty})^2}{K^{4}\Delta^{2}t^{2}}\mathbb{V}ar_{\theta}\Big[\int_{0}^{ut}\Big(\sum_{j=1}^{N}c_{N}^{K}(j)(M_{s}^{j,N}-M_{\phi_{t,\Delta}(s)}^{j,N})\Big)^{2}ds\Big]
    \\ &=\frac{(\mu A^{N,K}_{\infty,\infty})^2}{K^{4}\Delta^{2}t^{2}}\int_{0}^{ut}\!\!\int_{0}^{ut}\!\!
    \mathbb{C}ov_{\theta}\Big[\Big(\sum_{i=1}^{N}c_{N}^{K}(i)(M_{s}^{i,N}-M_{\phi_{t,\Delta}(s)}^{i,N})\Big)^{2},\Big(\sum_{j=1}^{N}c_{N}^{K}(j)(M_{s'}^{j,N}-M_{\phi_{t,\Delta}(s')}^{j,N})\Big)^{2}\Big]dsds'\\
    &=\frac{(\mu A^{N,K}_{\infty,\infty})^2}{K^{4}\Delta^{2}t^{2}}\int_{0}^{ut}\int_{0}^{ut}
    \sum_{1\le i,i',j,j'\le N}\mathbb{C}ov_{\theta}\Big[c_{N}^{K}(i)c_{N}^{K}(i')(M_{s}^{i,N}-M^{i,N}_{\phi_{t,\Delta}(s)})(M_{s}^{i',N}-M^{i',N}_{\phi_{t,\Delta}(s)}),\\
    &\hskip5cm c_{N}^{K}(j)c_{N}^{K}(j')(M_{s'}^{j,N}-M^{j,N}_{\phi_{t,\Delta}(s')})(M_{s'}^{j',N}-M^{j',N}_{\phi_{t,\Delta}(s')})\Big]dsds'\\
    &= \frac{(\mu A^{N,K}_{\infty,\infty})^2}{K^{4}\Delta^{2}t^{2}}\int_{0}^{ut}\int_{0}^{ut} \Big( \boldsymbol{1}_{\{|s-s'|> 3\Delta\}}+\boldsymbol{1}_{\{|s-s'|\le 3\Delta\}}\Big)\sum_{1\le i,i',j,j'\le N}\\
    &\hskip3cm \mathbb{C}ov_{\theta}\Big[c_{N}^{K}(i)c_{N}^{K}(i')(M_{s}^{i,N}-M^{i,N}_{\phi_{t,\Delta}(s)})(M_{s}^{i',N}-M^{i',N}_{\phi_{t,\Delta}(s)}),\\
    &\hskip5cm c_{N}^{K}(j)c_{N}^{K}(j')(M_{s'}^{j,N}-M^{j,N}_{\phi_{t,\Delta}(s')})(M_{s'}^{j',N}-M^{j',N}_{\phi_{t,\Delta}(s')})\Big]dsds'.
    \end{align*}
But on $\Omega_{N,K},$ we have
    \begin{align*}
    &\sum_{i,i',j,j'=1}^{N}\int_{0}^{ut}\int_{0}^{ut}\boldsymbol{1}_{\{|s-s'|\le 3\Delta\}}\mathbb{C}ov_{\theta}\Big[c_{N}^{K}(i)c_{N}^{K}(i')(M_{s}^{i,N}-M^{i,N}_{\phi_{t,\Delta}(s)})(M_{s}^{i',N}-M^{i',N}_{\phi_{t,\Delta}(s)}),\\
    &\hskip3cm c_{N}^{K}(j)c_{N}^{K}(j')(M_{s'}^{j,N}-M^{j,N}_{\phi_{t,\Delta}(s')})(M_{s'}^{j',N}-M^{j',N}_{\phi_{t,\Delta}(s')})\Big]dsds'\\
    \le &\int_{0}^{ut}\int_{0}^{ut}\boldsymbol{1}_{\{|s-s'|\le 3\Delta\}}\mathbb{E}_{\theta}\Big[\Big(\sum_{i=1}^{N}c_{N}^{K}(i)(M_{s}^{i,N}-M^{i,N}_{\phi_{t,\Delta}(s)})\Big)^{4}\Big]^\frac{1}{2}\\
    &\times\mathbb{E}_{\theta}\Big[\Big(\sum_{i=1}^{N}c_{N}^{K}(i)(M_{s'}^{i,N}-M^{i,N}_{\phi_{t,\Delta}(s')})\Big)^{4}\Big]^\frac{1}{2}dsds'\\
    \le &\int_{0}^{ut}\int_{0}^{ut}\boldsymbol{1}_{\{|s-s'|\le 3\Delta\}}\mathbb{E}_{\theta}\Big[\Big(\sum_{i=1}^{N}(c_{N}^{K}(i))^2(Z_{s}^{i,N}-Z^{i,N}_{\phi_{t,\Delta}(s)})\Big)^{2}\Big]^\frac{1}{2}\\
    &\times\mathbb{E}_{\theta}\Big[\Big(\sum_{i=1}^{N}(c_{N}^{K}(i))^2(Z_{s'}^{i,N}-Z^{i,N}_{\phi_{t,\Delta}(s')})\Big)^{2}\Big]^\frac{1}{2}dsds'\\
    \le& C\int_{0}^{ut}\int_{0}^{ut} \boldsymbol{1}_{\{|s-s'|\le 3\Delta\}}\Et\Big[\Big(K\Big(\bar{Z}_{s}^{N,K}-\bar{Z}_{\phi_{t,\Delta}(s)}^{N,K}\Big)+\frac{K^2}{N}\Big(\bar{Z}_{s}^{N}-\bar{Z}_{\phi_{t,\Delta}(s)}^{N}\Big)\Big)^{2}\Big]^{\frac{1}{2}}\\
    &\hskip3cm\times\Et\Big[\Big(K\Big(\bar{Z}_{s'}^{N,K}-\bar{Z}_{\phi_{t,\Delta}(s')}^{N,K}\Big)+\frac{K^2}{N}\Big(\bar{Z}_{s'}^{N}-\bar{Z}_{\phi_{t,\Delta}(s')}^{N}\Big)\Big)^{2}\Big]^{\frac{1}{2}}dsds'\\
    \le& Ct\Delta^{3}K^{2}.
\end{align*}
By \cite[Step 6 of the proof of Lemma 30]{A}, we already have,
when $|s-s'|\ge 3\Delta,$ that
\begin{align*}
    &\mathbb{C}ov_{\theta}[(M_{s}^{i,N}-M^{i,N}_{\phi_{t,\Delta}(s)})(M_{s}^{i',N}-M^{i',N}_{\phi_{t,\Delta}(s)}),
    (M_{s'}^{j,N}-M^{j,N}_{\phi_{t,\Delta}(s')})(M_{s'}^{j',N}-M^{j',N}_{\phi_{t,\Delta}(s')})]\\
    \le&  C (\indiq_{\{i=i'\}}+\indiq_{\{j=j'\}})  t^{1/2}\Delta^{1-q}.
\end{align*}
Hence,
\begin{align*}
     &\boldsymbol{1}_{\Omega_{N,K}}\sum_{i,i',j,j'=1}^{N}\int_{0}^{t}\int_{0}^{t}\boldsymbol{1}_{\{|s-s'|\ge 3\Delta\}}\mathbb{C}ov_{\theta}\Big[c_{N}^{K}(i)c_{N}^{K}(i')\Big(M_{s}^{i,N}-M^{i,N}_{\phi_{t,\Delta}(s)}\Big)\Big(M_{s}^{i',N}-M^{i',N}_{\phi_{t,\Delta}(s)}\Big),\\
&\hskip4cm c_{N}^{K}(j)c_{N}^{K}(j')\Big(M_{s'}^{j,N}-M^{j,N}_{\phi_{t,\Delta}(s')}\Big)\Big(M_{s'}^{j',N}-M^{j',N}_{\phi_{t,\Delta}(s')}\Big)\Big]dsds'
   \\
\le& \indiq_{\Omega_{N,K}}  C t^{5/2}\Delta^{1-q}\Big(\sum_{i=1}^{N}(c_{N}^{K}(i))^2\Big)\Big(\sum_{i=1}^{N}c_{N}^{K}(i)\Big)^{2}\le CK^{3} t^{5/2}\Delta^{1-q}.
\end{align*}
Overall,  we have, on $\Omega_{N,K}$
\begin{align*}
    &\mathbb{V}ar_{\theta}(\mathcal{A}^{u,3}_{N,K})
    \le \frac{1}{K^{4}\Delta^{2}t^{2}}\Big(\mu A^{N,K}_{\infty,\infty}\Big)^{2}\Big(\frac{K^{3} t^{5/2}}{\Delta^{q-1}}+t\Delta^{3}K^{2}\Big)
    \le C\Big(\frac{K\sqrt{t}}{\Delta^{q+1}}+\frac{\Delta}{t}\Big),
\end{align*}
due to the fact that on $\Omega_{N,K}$, $ |\ell_N(i)|\le C$ for all $1\le i\le N,$ and \eqref{CNK}.

Noting that $A^{N,K}_{\infty,\infty}=\sum_{j=1}^{N}\Big(c_{N}^{K}(j)\Big)^{2}\ell_{N}(j)$ and that $\int_{0}^{ut} (s-\phi_{t,\Delta}(s)) ds=\frac{u\Delta t}{2}$ , we have on $\Omega_{N,K},$
\begin{align*}
    & \Big|\mathbb{E}_{\theta}[\mathcal{A}^{u,3}_{N,K}]-\frac{u(\mu A^{N,K}_{\infty,\infty})^{2}}{2K^{2}}\Big|\\
    =&\frac{1}{\Delta tK^{2}}\Big|\mu A^{N,K}_{\infty,\infty}\int_{0}^{ut}\sum_{j=1}^{N}\Big\{\Big(c_{N}^{K}(j)\Big)^{2}\mathbb{E}_{\theta}\Big[Z_{s}^{j,N}-Z^{j,N}_{\phi_{t,\Delta}(s)}\Big]\Big\}ds-\frac{u\Delta t(\mu A^{N,K}_{\infty,\infty})^{2}}{2}\Big|\\
    =&\frac{1}{\Delta tK^{2}}\Big|\mu A^{N,K}_{\infty,\infty}\int_{0}^{ut}\sum_{j=1}^{N}\Big\{\Big(c_{N}^{K}(j)\Big)^{2}\mathbb{E}_{\theta}\Big[Z_{s}^{j,N}-Z^{j,N}_{\phi_{t,\Delta}(s)}-\mu(s-\phi_{t,\Delta}(s))\ell_{N}(j)\Big]\Big\}ds\Big|.
    \end{align*}
By  \eqref{CNK} and \cite[Lemma 16-(ii)]{A}, we obtain 
\[  \Big|\mathbb{E}_{\theta}[\mathcal{A}^{u,3}_{N,K}]-\frac{u(\mu A^{N,K}_{\infty,\infty})^{2}}{2K^{2}}\Big|\le \frac{C}{\Delta^q}.\]
Gathering all the previous computations, we obtain
\begin{align*}
    &\mathbb{E}\Big[\boldsymbol{1}_{\Omega_{N,K}}\Big|\big[\mathcal{L}^{t,\Delta}_{N,K}(.),\mathcal{L}^{t,\Delta}_{N,K}(.)\big]_{u}-\frac{2u(\mu A^{N,K}_{\infty,\infty})^{2}}{K^{2}}\Big|\Big]\\
    \le& 4 \mathbb{E}\Big[\boldsymbol{1}_{\Omega_{N,K}}\Big\{\Big|\mathcal{A}^{u,1}_{N,K}\Big|+\Big|\mathcal{A}^{u,2}_{N,K}\Big|+\mathbb{V}ar_{\theta}(\mathcal{A}^{u,3}_{N,K})^\frac{1}{2}+\Big|\mathbb{E}_{\theta}[\mathcal{A}^{u,3}_{N,K}]-\frac{u(\mu A^{N,K}_{\infty,\infty})^{2}}{2K^{2}}\Big|\Big\}\Big]\\
\le&  \frac{C}{K\Delta}+C\sqrt{\frac{\Delta}{t}}+\frac{C}{\sqrt{N}}+\frac{C}{t^q}+C\Big(\frac{K\sqrt{t}}{\Delta^{q+1}}+\frac{\Delta}{t}\Big)^{\frac{1}{2}}+\frac{C}{\Delta^q}\\ 
\le& C\Big(\frac{1}{K\Delta}+\frac{1}{\sqrt{N}}+\Big(\frac{K\sqrt{t}}{\Delta^{q+1}}\Big)^\frac{1}{2}+\sqrt{\frac{\Delta}{t}}\Big).
\end{align*}
The proof is finished.
\end{proof}

Next, we prove Proposition \ref{D33}.
\begin{proof}[Proof of Proposition \ref{D33}]

A direct application of Lemma  \ref{LmathbX}, Proposition \ref{main2} together with Lemma \ref{Agamma} gives 
$$
\frac{K}{\sqrt{t\Delta_t}}(\mathbb{X}_{\Delta_t,t,v}^{N,K})_{v\ge0} \stackrel{d}{\longrightarrow}\mu\sqrt{2}\Big(\frac{1-\gamma}{(1-\Lambda p)}+\frac{\gamma}{(1-\Lambda p)^3}\Big)(B_{2v}-B_v)_{v\ge0},$$
as desired.
\end{proof}

\subsection{Proof of Theorem \ref{corX}}\label{pf:corX}
 We recall that $\mathbb{X}_{\Delta_t,t,v}^{N,K}$ is defined in \eqref{mathbX} and note that $\mathbb{X}_{2\Delta_t,t,1}^{N,K}=\mathbb{X}_{\Delta_t,t,\frac{1}{2}}^{N,K}.$ By Proposition \ref{D33}, we have
$$
\frac{K}{N}\sqrt{\frac{t}{\Delta_t}}\frac{N}{t}\Big(-\mathbb{X}_{\Delta_t,t,1}^{N,K}+2\mathbb{X}_{\Delta_t,t,\frac{1}{2}}^{N,K}\Big)\stackrel{d}{\longrightarrow}\mathcal{N}\Big(0,6\mu^2\Big(\frac{1-\gamma}{(1-\Lambda p)}+\frac{\gamma}{(1-\Lambda p)^3}\Big)^2\Big).
$$
By Proposition \ref{D3}, we conclude that
\begin{align*}
\frac{K}{N}\sqrt{\frac{t}{\Delta_t}}\mathbb{E}\Big[\indiq_{\Omega_{N,K}}\Big|-D_{\Delta_t,t}^{N,K,3}+2D_{2\Delta_t,t}^{N,K,3}-2\frac{N}{t}\mathbb{X}_{2\Delta_t,t,1}^{N,K}+\frac{N}{t}\mathbb{X}_{\Delta_t,t,1}^{N,K}\Big|\Big]
\le \frac{CK}{N\Delta_t}+\frac{C\sqrt{K}}{\sqrt{N\Delta_t}}+\frac{Ct^\frac{3}{4}\sqrt{K}}{\Delta_t^{1+\frac{q}{2}}}.
\end{align*}
Consequently, by  Lemma \ref{D124}, we obtain the following convergence in probability:  as $(N,K,t)\to (\infty,\infty,\infty)$ such that
$\frac{K}{N} \to \gamma\le 1$ and 
$\frac 1{\sqrt K} + \frac NK \sqrt{\frac{\Delta_t}t}+ \frac{N}{t\sqrt K}+Ne^{-c_{p,\Lambda}K} \to 0$, the limit of 
$$\omg\frac{K}{N}\sqrt{\frac{t}{\Delta_t}}\Big(\mathcal{X}_{\Delta_t,t}^{N,K}-\cX^{N,K}_{\infty,\infty}\Big)$$ equals  the limit of $$\omg\frac{K}{N}\sqrt{\frac{t}{\Delta_t}}\Big\{-D_{\Delta_t,t}^{N,K,3}+2D_{2\Delta_t,t}^{N,K,3}\Big\},$$ which in turn equals the limit of
$$ \frac{K}{N}\sqrt{\frac{t}{\Delta_t}}\frac{N}{t}\Big(-\mathbb{X}_{\Delta_t,t,1}^{N,K}+2\mathbb{X}_{\Delta_t,t,\frac{1}{2}}^{N,K}\Big).$$
This finally converges in distribution to $
\mathcal{N}\Big(0,6\mu^2\Big(\frac{1-\gamma}{(1-\Lambda p)}+\frac{\gamma}{(1-\Lambda p)^3}\Big)^2\Big)$.
 
\section{Proof of the main result}\label{finsecsub}
In this section, we present the proofs of the main results stated in Section \ref{MR}. First, we recall the estimators $\varepsilon_{t}^{N,K}$, $\mathcal{V}_{t}^{N,K}$, and $\mathcal{X}_{\Delta_{t},t}^{N,K}$ defined in Section \ref{MR}, as well as the function $$\Psi^{(3)}(u,v,w)=\frac{u^2(1-\sqrt{\frac{u}{w}})^2}{v+u^2(1-\sqrt{\frac{u}{w}})^{2}}\quad \text{if $u>0$, $v>0$, $w>0$}\quad \text{and \ $\Psi^{(3)}(u,v,w)=0$ otherwise,}$$ and 
$$\hat p_{N,K,t}:= \Psi^{(3)}(\varepsilon_{t}^{N,K},\mathcal{V}_{t}^{N,K},\mathcal{X}_{\Delta_{t},t}^{N,K}), \quad\text {with  $\Delta_t=(2 \lfloor t^{1-4/(q+1)}\rfloor)^{-1}t$.}$$
We now proceed to the proof of  Theorem \ref{mainsubcr}.

\begin{proof}[Proof of Theorem \ref{mainsubcr}] 
It can be directly verified that
$\Psi^{(3)}\Big(\frac{\mu}{1-\Lambda p},\frac{(\mu\Lambda)^{2}p(1-p)}{(1-\Lambda p)^{2}},
\frac{\mu}{(1-\Lambda p)^{3}}\Big)=p.$
By the mean value theorem, there exist some vectors $\boldsymbol{C_{N,K,t}^{i}}$ for $i=1,2,3,$ lying on the segment between $(\varepsilon_{t}^{N,K},\mathcal{V}_{t}^{N,K},\mathcal{X}_{\Delta_t,t}^{N,K})$ and 
$\boldsymbol{C}:=\Big(\frac{\mu}{1-\Lambda p},\frac{(\mu\Lambda)^2 p(1-p)}{(1-\Lambda p)^2},\frac{\mu}{(1-\Lambda p)^3}\Big)$, such that
\begin{align*}
\hat p_{N,K,t}-p=&\Psi^{(3)}(\varepsilon_{t}^{N,K},\mathcal{V}_{t}^{N,K},\mathcal{X}_{\Delta_t,t}^{N,K})-p\\
     =&\Psi^{(3)}(\varepsilon_{t}^{N,K},\mathcal{V}_{t}^{N,K},\mathcal{X}_{\Delta_t,t}^{N,K})-\Psi^{(3)}\Big(\frac{\mu}{1-\Lambda p},\frac{(\mu\Lambda)^{2}p(1-p)}{(1-\Lambda p)^{2}},\frac{\mu}{(1-\Lambda p)^{3}}\Big)\\
    =&\frac{\partial\Psi^{(3)}}{\partial u}(\boldsymbol{C_{N,K,t}^{1}})\Big(\varepsilon_{t}^{N,K}-\frac{\mu}{1-\Lambda p}\Big)+
\frac{\partial\Psi^{(3)}}{\partial v}(\boldsymbol{C_{N,K,t}^{2}})\Big(\mathcal{V}_{t}^{N,K}-\frac{(\mu\Lambda)^{2}p(1-p)}{(1-\Lambda p)^{2}}\Big)\\
     &+\frac{\partial\Psi^{(3)}}{\partial w}(\boldsymbol{C_{N,K,t}^{3}})\Big(\mathcal{X}_{\Delta_t,t}^{N,K}-\frac{\mu}{(1-\Lambda p)^{3}}\Big).
     \end{align*}
From the first paragraph of \cite[Section 10]{D}, it is established that, when   
$(N,K,t)\to (\infty,\infty,\infty)$
and  $\frac 1{\sqrt K} + \frac NK \sqrt{\frac{\Delta_t}t}+ \frac{N}{t\sqrt K}+Ne^{-c_{p,\Lambda}K} \to 0$, $(\varepsilon_{t}^{N,K},\mathcal{V}_{t}^{N,K},\mathcal{X}_{\Delta_t,t}^{N,K})$ converges in probability to $\boldsymbol{C}$.  Consequently, the three vectors $\boldsymbol{C}_{N,K,t}^{i}$, $i=1,2,3$, all converge to 
$\boldsymbol{C}:=\Big(\frac{\mu}{1-\Lambda p},\frac{(\mu\Lambda)^2 p(1-p)}{(1-\Lambda p)^2},\frac{\mu}{(1-\Lambda p)^3}\Big)$ in probability.

\vip
We define the following functions from $D':= \{(u, v, w)\in \mathbb{R}^{3}: w >u> 0\  and\  v> 0\}$ to $\mathbb{R}^{3}$ by
$$\Psi^{(1)}(u,v,w)=u\sqrt{\frac{u}{w}},\quad  \Psi^{(2)}(u,v,w)=\frac{v+(u-\Psi^{(1)})^{2}}{u(u-\Psi^{(1)})}.$$
Then, on $D'$, we have $\Psi^{(3)}(u,v,w)=\frac{1-u^{-1}\Psi^{(1)}}{\Psi^{(2)}}$.
\vip
A series of tedious but straightforward calculations yields
\begin{gather*}
\quad\quad\frac{\partial\Psi^{(1)}}{\partial v}(\boldsymbol{C})=0,\quad \frac{\partial\Psi^{(1)}}{\partial w}(\boldsymbol{C})=\frac{-(1-\Lambda p)^3}{2},\quad  \frac{\partial\Psi^{(2)}}{\partial v}(\boldsymbol{C})=\frac{(1-\Lambda p)^2}{\mu^2\Lambda p},\\ 
\frac{\partial\Psi^{(2)}}{\partial w}(\boldsymbol{C})=\Big\{\frac{-2\frac{\partial\Psi^{(1)}}{\partial w}}{u}+\frac{\Psi^{(2)}\frac{\partial\Psi^{(1)}}{\partial w}}{(u-\Psi^{(1)})}\Big\}(\boldsymbol{C})=\frac{(1-\Lambda p)^4(2p-1)}{2\mu p},\\
\hskip-0.3cm \frac{\partial\Psi^{(3)}}{\partial v}(\boldsymbol{C})=-\frac{\frac{\Psi^{(2)}\frac{\partial\Psi^{(1)}}{\partial v}}{u}+(1-\frac{\Psi^{(1)}}{u})\frac{\partial\Psi^{(2)}}{\partial v}}{(\Psi^{(2)})^2}(\boldsymbol{C})=-\frac{(1-\Lambda p)^2}{(\mu\Lambda)^2},\\
\hskip0.4cm \frac{\partial\Psi^{(3)}}{\partial w}(\boldsymbol{C})=-\frac{\frac{\Psi^{(2)}\frac{\partial\Psi^{(1)}}{\partial w}}{u}+(1-\frac{\Psi^{(1)}}{u})\frac{\partial\Psi^{(2)}}{\partial w}}{(\Psi^{(2)})^2}(\boldsymbol{C})=\frac{(1-\Lambda p)^4(1-p)}{\mu \Lambda}.
\end{gather*}

{\it Case 1.}
The dominant term is $\frac 1{\sqrt K}$, i.e. when
$[\frac 1{\sqrt K}]/[ \frac NK \sqrt{\frac{\Delta_t}t}+ \frac{N}{t\sqrt K}]\to \infty$,
we write
\begin{align*}
\sqrt{K}\Big(\hat p_{N,K,t}-p\Big) =&\sqrt{K}\frac{\partial\Psi^{(3)}}{\partial u}(\boldsymbol{C_{N,K,t}^{1}})\Big(\varepsilon_{t}^{N,K}-\frac{\mu}{1-\Lambda p}\Big)\\
&+
\sqrt{K}\frac{\partial\Psi^{(3)}}{\partial v}(\boldsymbol{C_{N,K,t}^{2}})\Big(\mathcal{V}_{t}^{N,K}-\frac{(\mu\Lambda)^{2}p(1-p)}{(1-\Lambda p)^{2}}\Big)\notag\\
&+\sqrt{K}\frac{\partial\Psi^{(3)}}{\partial w}(\boldsymbol{C_{N,K,t}^{3}})\Big(\cX_{\Delta_{t},t}^{N,K}-\frac{\mu}{(1-\Lambda p)^{3}}\Big).\notag
\end{align*}
Based on Lemmas \ref{ellp}, \ref{barell} and \ref{W} and  Theorem \ref{corX}, we obtain
$$
\sqrt{K}\frac{\partial\Psi^{(3)}}{\partial u}(\boldsymbol{C_{N,K,t}^{1}})\Big(\varepsilon_{t}^{N,K}-\frac{\mu}{1-\Lambda p}\Big)+\sqrt K \frac{\partial\Psi^{(3)}}{\partial w}(\boldsymbol{C_{N,K,t}^{3}})\Big(\cX_{\Delta_{t},t}^{N,K}-\frac{\mu}{(1-\Lambda p)^{3}}\Big)\stackrel{d}{\longrightarrow} 0.
$$
Next, we observe that as $(N,K,t)\to(\infty,\infty,\infty)$,
$$\frac{\partial\Psi^{(3)}}{\partial v}(\boldsymbol{C_{N,K,t}^{2}}){\longrightarrow}-\frac{(1-\Lambda p)^{2}}{(\mu\Lambda)^{2}} \quad \textit {in probability}.$$
Therefore, by Theorems \ref{21} and \ref{VVNK}, we conclude that
$$
\sqrt{K}\frac{\partial\Psi^{(3)}}{\partial v}(\boldsymbol{C_{N,K,t}^{2}})\Big(\mathcal{V}_{t}^{N,K}-\frac{(\mu\Lambda)^{2}p(1-p)}{(1-\Lambda p)^{2}}\Big)\stackrel{d}{\longrightarrow} \mathcal{N}\Big(0,p^2(1-p)^{2}\Big),
$$
which in turn implies that 
\[\sqrt{K}\Big(\hat p_{N,K,t}-p\Big)\stackrel{d}{\longrightarrow} \mathcal{N}\Big(0,p^2(1-p)^{2}\Big).\]

{\it Case 2.}
The dominant term is $\frac{N}{t\sqrt K}$, i.e. when
$[\frac{N}{t\sqrt K}]/[\frac 1{\sqrt K}+\frac NK \sqrt{\frac{\Delta_t}t}]\to \infty$,
we write 
\begin{align*}
\frac{t\sqrt{K}}N\Big(\hat p_{N,K,t}-p\Big) =&\frac{t\sqrt{K}}N\frac{\partial\Psi^{(3)}}{\partial u}(\boldsymbol{C_{N,K,t}^{1}})\Big(\varepsilon_{t}^{N,K}-\frac{\mu}{1-\Lambda p}\Big)\\
&+
\frac{t\sqrt{K}}N\frac{\partial\Psi^{(3)}}{\partial v}(\boldsymbol{C_{N,K,t}^{2}})\Big(\mathcal{V}_{t}^{N,K}-\frac{(\mu\Lambda)^{2}p(1-p)}{(1-\Lambda p)^{2}}\Big)\notag\\
&+\frac{t\sqrt{K}}N\frac{\partial\Psi^{(3)}}{\partial w}(\boldsymbol{C_{N,K,t}^{3}})\Big(\cX_{\Delta_{t},t}^{N,K}-\frac{\mu}{(1-\Lambda p)^{3}}\Big).\notag
\end{align*}
Similarly, according to Lemmas \ref{ellp}, \ref{barell} and \ref{W} and Theorem \ref{corX}, we obtain
$$
\frac{t\sqrt{K}}N\frac{\partial\Psi^{(3)}}{\partial u}(\boldsymbol{C_{N,K,t}^{1}})\Big(\varepsilon_{t}^{N,K}-\frac{\mu}{1-\Lambda p}\Big)+\frac{t\sqrt{K}}N\frac{\partial\Psi^{(3)}}{\partial w}(\boldsymbol{C_{N,K,t}^{3}})\Big(\cX_{\Delta_{t},t}^{N,K}-\frac{\mu}{(1-\Lambda p)^{3}}\Big)\stackrel{d}{\longrightarrow} 0.
$$
Finally, using Theorems \ref{21} and \ref{VVNK}, we find that
\begin{align*}
&\frac{t\sqrt{K}}{N}\Big(\hat p_{N,K,t}-p\Big)\stackrel{d}{\longrightarrow} 
\mathcal{N}\Big(0,\frac{2 (1-\Lambda p)^2}{\mu^2 \Lambda^4}\Big).
\end{align*}

{\it Case 3.}
The  dominant term is $\frac{N}{K}\sqrt{\frac{\Delta_t}t}$, i.e. when
$[\frac{N}{K}\sqrt{\frac{\Delta_t}t}]/[\frac 1{\sqrt K}+ \frac{N}{t\sqrt K} ]\to \infty$ and $\frac{K}{N}\to\gamma\le 1$. Using Lemmas \ref{ellp},  \ref{barell} and Theorems \ref{21} and \ref{VVNK},  we have
\begin{align*}
\frac{K}{N}\sqrt{\frac{t}{\Delta_{t}}}\Big\{\frac{\partial\Psi^{(3)}}{\partial u}(\boldsymbol{C_{N,K,t}^{1}})\Big(\varepsilon_{t}^{N,K}-\frac{\mu}{1-\Lambda p}\Big)+\frac{\partial\Psi^{(3)}}{\partial v}(\boldsymbol{C_{N,K,t}^{2}})\Big(\cV_{t}^{N,K}-\frac{(\mu\Lambda)^{2}p(1-p)}{(1-\Lambda p)^{2}}\Big)\Big\}\stackrel{d}{\longrightarrow} 0.
\end{align*}
Applying Lemma \ref{W} gives
\[\frac{K}{N}\sqrt{\frac{t}{\Delta_{t}}}\frac{\partial\Psi^{(3)}}{\partial w}(\boldsymbol{C_{N,K,t}^{3}})\Big(\cX_{\infty,\infty}^{N,K}-\frac{\mu}{(1-\Lambda p)^{3}}\Big)\stackrel{d}{\longrightarrow} 0.\]
Consequently, it suffices to analyze
$$
\frac{K}{N}\sqrt{\frac{t}{\Delta_{t}}} \frac{\partial\Psi^{(3)}}{\partial w}(\boldsymbol{C_{N,K,t}^{3}})
\Big( \cX_{\Delta_t,t}^{N,K} - \cX_{\infty,\infty}^{N,K}\Big).
$$
Since
$$\frac{\partial\Psi^{(3)}}{\partial w}(\boldsymbol{C_{N,K,t}^{3}}){\longrightarrow}\frac{(1-\Lambda p)^4(1-p)}{\mu \Lambda}\ \textit{ in probability},$$
 and by Theorem \ref{corX}, we finally conclude that
\begin{align*}
&\frac{K}{N}\sqrt{\frac{t}{\Delta_{t}}}\Big(\hat p_{N,K,t}-p\Big)\stackrel{d}{\longrightarrow} \mathcal{N}\Bigg(0,\frac{6(1-p)^2}{\Lambda^2}\Big((1-\gamma)(1-\Lambda p)^{3}+\gamma(1-\Lambda p)\Big)^2\Bigg).
\end{align*}
\end{proof}

Next, we move to prove  Proposition \ref{pzero}.

\begin{proof}[Proof of Proposition \ref{pzero}]
We note that for the case $p=0,$  the conclusion of Theorem \ref{corX} remains valid (and the limit of $\frac{K}{N}$ is no longer required). One can verify directly that $\bar{\ell}_{N}^{K}=1$, $\cV_\infty^{N,K}=0$ and  $\cX^{N,K}_{\infty,\infty}=\mu.$ Define 
 \begin{align*}
f(u,v,w):=\frac{u(w-u)}{w+\sqrt{wu}}\quad when \text{ $u>0$, $w>0$}\qquad \text{and  $f=0$ otherwise.}
  \end{align*}
 By \cite[Lemma 7.3]{D}, 
 $$\omg\frac{K}{N}\sqrt{\frac{t}{\Delta_{t}}}(\varepsilon_{t}^{N,K}-\mu)\longrightarrow 0  \quad \textit{in probability.}$$
 Hence,  applying  Theorem \ref{corX},  we obtain
 \begin{align*}
    \omg\frac{K}{N}\sqrt{\frac{t}{\Delta_{t}}}(\mathcal{X}_{\Delta_t,t}^{N,K}-\varepsilon_{t}^{N,K})\stackrel{d}{\longrightarrow}\mathcal{N}\Big(0,2\mu^2\Big).
 \end{align*}
 From \cite[Lemma 7.3, Corollary 9.9]{D}, when $(N,K,t)\to(\infty,\infty,\infty)$ and $ \frac NK \sqrt{\frac{\Delta_t}t}+ \frac{N}{t\sqrt K}+Ne^{-c_{p,\Lambda}K} \to 0$, both
 $\varepsilon_{t}^{N,K}$ and $\mathcal{X}_{\Delta_t,t}^{N,K}$  converge to $ \mu$ in probability.
Consequently,
$$
\omg\frac{K}{N}\sqrt{\frac{t}{\Delta_{t}}}f(\varepsilon_{t}^{N,K},\mathcal{V}_{t}^{N,K},\mathcal{X}_{\Delta_{t},t}^{N,K})\stackrel{d}{\longrightarrow} \mathcal{N}\Big(0,\frac{\mu^2}{2}\Big).
$$ 
By Theorem \ref{VVNK}, we obtain
$$\omg\frac{t\sqrt{K}}{N}\cV_{t}^{N,K}\stackrel{d}{\longrightarrow}\mathcal{N}
\Big(0,2\mu^2\Big).$$
Therefore, if  $[\frac{N}{t\sqrt K}]/[\frac NK \sqrt{\frac{\Delta_t}t}]^2\to \infty$,  then 
$$
 \Big[\frac NK \sqrt{\frac{\Delta_t}t}\Big]^{-2}\Big|\mathcal{V}_{t}^{N,K}\Big| \longrightarrow \infty \textit{ in probability.}$$
Since $\Psi^{(3)}(u,v,w)=\frac{f^2}{v+f^2} 1_{\{ v>0\}}$, it follows that
$$
\hat p_{N,K,t} = \Psi^{(3)}(\varepsilon_{t}^{N,K},\mathcal{V}_{t}^{N,K},\mathcal{X}_{\Delta_{t},t}^{N,K})\stackrel{d}{\longrightarrow} 0.
$$
On the other hand, when $\Big[\frac{N}{K}\sqrt{\frac{\Delta_t}t}\Big]^2/[ \frac{N}{t\sqrt K} ]\to \infty$, we have
$$
 \big[\frac NK \sqrt{\frac{\Delta_t}t}\Big]^{-2}\Big|\mathcal{V}_{t}^{N,K}\Big| \longrightarrow 0 \textit{ in probability.}$$
Thus,
$$ \frac{f^2(\varepsilon_{t}^{N,K},\mathcal{V}_{t}^{N,K},\mathcal{X}_{\Delta_{t},t}^{N,K})}{\mathcal{V}_{t}^{N,K}+ f^2(\varepsilon_{t}^{N,K},\mathcal{V}_{t}^{N,K},\mathcal{X}_{\Delta_{t},t}^{N,K})} \longrightarrow 1\,  \textit{in probability},$$
which holds whenever
$$
\mathbb{P}\Big(\mathcal{V}_{t}^{N,K}>0\Big)\longrightarrow\frac{1}{2}.
$$
Hence,
$\hat p_{N,K,t} \stackrel{d}{\longrightarrow} X.$
\end{proof}

\appendix

\section{Proof of Lemma \ref{lambar} }\label{prof: lambar}

\subsection{Proof of Lemma \ref{lambar} (i)} Observing from $\Et[Z^{i,N}_t]=\int_0^t \Et[\lambda^{i,N}_s] ds$, a direct computation yields that 
$$
\max_{i=1,...,N}\mathbb{E}_\theta[\lambda_t^{i,N}]=\mu +\max_{i=1,...,N}\Big\{\sum_{j=1}^NA_N(i,j)\int_0^t\phi(t-s)\mathbb{E}_\theta[\lambda_s^{j,N}] ds\Big\}.
$$
Define $a_N(t):=\sup_{0\le s\le t}\max_{i=1,...,N}\mathbb{E}_\theta[\lambda_s^{i,N}].$ 
On the event $\Omega_{N,K}$,  we have 
$\Lambda\max_{i=1,...,N}\{\sum_{j=1}^NA_N(i,j)\}\le a<1$.  Since $\Lambda=\int_0^\infty \phi(s)ds$, it follows that
$$
a_N(t)\le \mu +a_N(t)a,
$$
which immediately implies the desired result.
\vip 
For the second part, recalling the definition of  $M^{i,N}_t$ in Section \ref{aux}, we have $M^{i,N}_t=Z^{i,N}_t-\int_0^t \lambda^{i,N}_s ds$. We express the intensity process $\lambda_{t}^{i,N}$ defined in \eqref{sssy} as 
\[\lambda_{t}^{i,N}=\mu+\frac{1}{N}\sum_{j=1}^{N}\theta_{ij}\int_{0}^{t}\phi(t-s)dM_{s}^{j,N}+\frac{1}{N}\sum_{j=1}^{N}\theta_{ij} \int_{0}^{t}\phi(t-s)\lambda_{s}^{j,N}ds.\]
An application of Minkowski’s inequality then yields
\begin{align}\label{lam2}
\mathbb{E}_{\theta}\Big[(\lambda_{t}^{i,N})^{2}\Big]^{\frac{1}{2}}
\le \mu&+ \mathbb{E}_{\theta}\Big[\Big(\frac{1}{N}\sum_{j=1}^{N}\theta_{ij}\int_{0}^{t}\phi(t-s)dM_{s}^{j,N}\Big)^{2}\Big]^{\frac{1}{2}}\\
&\hskip2cm+\mathbb{E}_{\theta}\Big[\Big(\frac{1}{N}\sum_{j=1}^{N}\theta_{ij}\int_{0}^{t}\phi(t-s)\lambda_{s}^{j,N}ds\Big)^{2}\Big]^{\frac{1}{2}}.\notag
\end{align}
Using  \eqref{ee3} and $0\le\theta_{ij}\le1$, we obtain
\begin{align*}
 \mathbb{E}_{\theta}\Big[\Big(\frac{1}{N}\sum_{j=1}^{N}\theta_{ij}\int_{0}^{t}\phi(t-s)dM_{s}^{j,N}\Big)^{2}\Big]^{\frac{1}{2}}
=&\frac{1}{N}\mathbb{E}_{\theta}\Big[\sum_{j=1}^{N}\int_{0}^{t}\Big(\theta_{ij}\phi(t-s)\Big)^2dZ_{s}^{j,N}\Big]^{\frac{1}{2}}\\
=&\frac{1}{N}\mathbb{E}_{\theta}\Big[\sum_{j=1}^{N}\int_{0}^{t}\Big(\theta_{ij}\phi(t-s)\Big)^2\lambda_{s}^{j,N}ds\Big]^{\frac{1}{2}}\\
\le& \frac{1}{\sqrt{N}}\Big[\int_{0}^{t}\Big(\phi(t-s)\Big)^2\max_{j=1,...,N}\mathbb{E}_{\theta}[\lambda_{s}^{j,N}]ds\Big]^{\frac{1}{2}}.
\end{align*}
From the first part of Lemma \ref{lambar}-(i) and  assumption  \eqref{H(q)}, it follows that
\begin{align}\label{MES}
 \mathbb{E}_{\theta}\Big[\Big(\frac{1}{N}\sum_{j=1}^{N}\theta_{ij}\int_{0}^{t}\phi(t-s)dM_{s}^{j,N}\Big)^{2}\Big]^{\frac{1}{2}} \le \frac{C}{\sqrt{N}}. 
\end{align}
Now, applying  Minkowski’s inequality to the third term of \eqref{lam2} yields 
\begin{align*}
\mathbb{E}_{\theta}\Big[\Big(\frac{1}{N}\sum_{j=1}^{N}\theta_{ij}\int_{0}^{t}\phi(t-s)\lambda_{s}^{j,N}ds\Big)^{2}\Big]^{\frac{1}{2}}\le  \frac{1}{N}\sum_{j=1}^{N}\theta_{ij}\int_{0}^{t}\phi(t-s)\mathbb{E}_\theta\Big[(\lambda_{s}^{j,N})^{2}\Big]^{\frac{1}{2}}ds.
\end{align*}
Therefore,
\begin{align*}
\max_{i=1,...,N}\mathbb{E}_{\theta}\Big[(\lambda_{t}^{i,N})^{2}\Big]^{\frac{1}{2}}\le \mu+ \frac{C}{\sqrt{N}}+\max_{i=1,...,N}\Big\{\frac{1}{N}\sum_{j=1}^{N}\theta_{ij}\int_{0}^{t}\phi(t-s)\mathbb{E}_\theta\Big[(\lambda_{s}^{j,N})^{2}\Big]^{\frac{1}{2}}ds\Big\}.
\end{align*}
Define $b_N(t):=\sup_{0\le s \le t}\max_{i=1,...,N}\mathbb{E}_{\theta}\big[(\lambda_{s}^{i,N})^{2}\big]^{\frac{1}{2}}$,  then we have 
\[b_N(t)\le \mu+\frac{C}{\sqrt{N}}+\Lambda\max_{i=1,...,N}\{\frac{1}{N}\sum_{j=1}^{N}\theta_{ij}\}b_N(t).\] 
Recalling that $A_N(i,j)=\frac{1}{N}\theta_{ij}$, and that  on  $\Omega_{N,K}$, 
$\Lambda\max_{i=1,...,N}\{\sum_{j=1}^NA_N(i,j)\}\le a<1$,  we conclude that 
\[b_N(t)\le \mu+C+ab_N(t),\] 
which completes the proof.

\subsection{Proof of Lemma \ref{lambar} (ii)}
Starting from the definition $\bl_{N}=Q_{N}\boldsymbol{1}_{N}=(I-\Lambda A_N)^{-1}\indiq_N$, we obtain  $\bl_{N}=\indiq_{N}+\Lambda A_{N}\bl_{N}.$ Recalling \eqref{sssy} and  writing $\Lambda=\int_{0}^{t}\phi(t-s)ds + \int_t^\infty \phi(s)ds$, we obtain
$$\lambda_{t}^{i,N}-\mu\ell_{N}(i)=\frac{1}{N}\sum_{j=1}^{N}\theta_{ij}\Big(\int_{0}^{t}\phi(t-s)dZ_{s}^{j,N}-\mu\ell_N(j)\int_{0}^{t}\phi(t-s)ds\Big)-\frac{\mu}{N}\sum_{j=1}^{N}\theta_{ij}\ell_N(j)\int_{t}^{\infty}\phi(s)ds.$$
Applying Minkowski’s inequality, we obtain
\begin{align*}
\mathbb{E}_{\theta}\Big[\Big(\lambda_{t}^{i,N}-\mu\ell_{N}(i)\Big)^{2}\Big]^{\frac{1}{2}}
\le& \mathbb{E}_{\theta}\Big[\Big(\frac{1}{N}\sum_{j=1}^{N}\theta_{ij}\Big(\int_{0}^{t}\phi(t-s)dZ_{s}^{j,N}-\mu\ell_{N}(j)\int_{0}^{t}\phi(t-s)ds\Big)\Big)^{2}\Big]^{\frac{1}{2}}\\
&\qquad+\mu\mathbb{E}_{\theta}\Big[\Big(\frac{1}{N}\sum_{j=1}^{N}\theta_{ij}\ell_N(j)\int_{t}^{\infty}\phi(s)ds\Big)^{2}\Big]^{\frac{1}{2}}.
\end{align*}
As in the proof of Lemma \ref{lambar}-(i), we reformulate the first right hand side term of the above inequality via the process  $M^{i,N}_t$ defined in Section \ref{aux}. In addition, since $\ell_N(j)$ is uniformly bounded on $\Omega_{N,K}$, it follows that 
\begin{align*}
\mathbb{E}_{\theta}\Big[\Big(\lambda_{t}^{i,N}-\mu\ell_{N}(i)\Big)^{2}\Big]^{\frac{1}{2}}\le& \mathbb{E}_{\theta}\Big[\Big(\frac{1}{N}\sum_{j=1}^{N}\theta_{ij}\int_{0}^{t}\phi(t-s)dM_{s}^{j,N}\Big)^{2}\Big]^{\frac{1}{2}}
\\&+\mathbb{E}_{\theta}\Big[\Big(\frac{1}{N}\sum_{j=1}^{N}\theta_{ij}\int_{0}^{t}\phi(t-s)|\lambda_{s}^{j,N}-\mu\ell_{N}(j)|ds\Big)^{2}\Big]^{\frac{1}{2}}
+C\int_{t}^{\infty}\phi(s)ds.
\end{align*}
Define $F^{N,K}_{t}:=\frac{1}{K}\sum_{i=1}^{K}\mathbb{E}_{\theta}[(\lambda_{t}^{i,N}-\mu\ell_{N}(i))^{2}]^{\frac{1}{2}}$. Using\eqref{MES} and  Minkowski’s inequality, we obtain
\begin{align*}
F^{N,K}_{t}\le&  \frac{1}{KN}\sum_{j=1}^{N}\sum_{i=1}^{K}\theta_{ij}\int_{0}^{t}\phi(t-s)\mathbb{E}_\theta\Big[\Big|\lambda_{s}^{j,N}-\mu\ell_N(j)\Big|^{2}\Big]^{\frac{1}{2}}ds+C\int_{t}^{\infty}\phi(s)ds+\frac{C}{\sqrt{N}}\\
\le& \int_{0}^{t}\frac{N}{K}|||I_{K}A_{N}|||_{1}\phi(t-s)F^{N,N}_{s}ds+C\int_{t}^{\infty}\phi(s)ds+\frac{C}{\sqrt{N}}.
\end{align*}
Using  $N |||I_{K}A_{N}|||_{1}=\max_{j=1,\dots,N}\sum_{i=1}^K \theta_{ij}$ and  the bound  $\frac{N}{K}|||I_KA_{N}|||_{1}\leq a/\Lambda $ on $\Omega_{N,K},$ we obtain 
\begin{align*}
F^{N,K}_{t}
\le \int_{0}^{t}\frac{a}{\Lambda}\phi(t-s)F^{N,N}_{s}ds+C\int_{t}^{\infty}\phi(s)ds+\frac{C}{\sqrt{N}}.
\end{align*}
\vip
Defining $g_N(t):=C\int_{t}^{\infty}\phi(s)ds+\frac{C}{\sqrt{N}}$, we have  on $\Omega_{N,K}$, 
for all $K=1,\dots,N$,
\begin{equation}\label{ggg}
F^{N,K}_t\leq \int_{0}^{t}\frac{a}{\Lambda}\phi(t-s)F^{N,N}_{s}ds+g_N(t).
\end{equation}
From assumption   \eqref{H(q)}, namely $\int_{0}^{\infty}(1+s^{q})\phi(s)ds<\infty$ , it follows that $g_N(t)\leq C(\frac{1}{t^{q}}\land 1)+C N^{-1/2}$.
Moreover, due to Lemma \ref{lambar}-(i) and  the uniform boundness of $\ell_N(j)$ on $\Omega_{N,K},$,  we deduce that
 $\sup_{t\ge0}F^{N,N}_{t}\le C$, so that  $\int_0^{t}(\frac{a}{\Lambda})^n\phi^{*n}(t-s)F^{N,N}_{s}ds\leq Ca^n \to 0$ as $n\to\infty.$ 
Hence, iterating \eqref{ggg} (using it once with some fixed $K\in \{1,\dots,N\}$ and then always with $K=N$), 
one concludes that on  $\Omega_{N,K}$, 
\begin{align*}
F^{N,N}_t\leq& \sum_{n\ge 1}\int_0^{t}\Big( \frac{a}{\Lambda}\Big)^n\phi^{*n}(t-s)g_N(s)ds+g_N(t)\\
\le& \sum_{n\ge 1}\int_0^{\frac{t}{2}}\Big(\frac{a}{\Lambda}\Big)^n\phi^{*n}(t-s)g_N(s)ds+\sum_{n\ge 1}\int_{\frac{t}{2}}^{t}\Big(\frac{a}{\Lambda}\Big)^n\phi^{*n}(t-s)g_N(s)ds+g_N(t)\\
\le& C \sum_{n\ge 1}\int_{\frac{t}{2}}^t\Big(\frac{a}{\Lambda}\Big)^n\phi^{*n}(s)ds+g_N\Big(\frac{t}{2}\Big)\sum_{n\ge 1}\int_{0}^{\infty}\Big(\frac{a}{\Lambda}\Big)^n\phi^{*n}(s)ds+g_N(t),
\end{align*}
because $g_N$ is non-increasing and bounded. Recalling that 
$\int_0^\infty \phi^{*n}(s)ds =\Lambda^n$ and,  as shown in  \cite[Proof of Lemma 15-(ii)]{A}, that
$$
\int_r^\infty \phi^{* n}(u)du \leq C n^q\Lambda^n  r^{-q},
$$ 
we conclude that (since $a \in (0,1)$)
\begin{align*}
F^{N,N}_t
\le& C \Big(\frac{t}{2}\Big)^{-q}\sum_{n\ge 1}n^qa^n+ g_N\Big(\frac{t}{2}\Big)\frac{a}{1-a}+g_N(t)
\le \frac{C}{t^q}+\frac{C}{\sqrt{N}}.
\end{align*}
This completes the proof.

\subsection{Proof of Lemma \ref{lambar} (iii)\&(iv)} We restrict our proof to part (iii), since the argument for part (iv) is virtually identical. Recall from \eqref{ee2} that
$$
U^{i,N}_t =  \sum_{n\geq 0} \intot \phi^{*n}(t-s) \sum_{j=1}^N A_N^n(i,j)M^{j,N}_sds.
$$ 
We set $\phi(s)=0$ for $s\le 0$. Separating the cases $n=0$ and $n\geq 1$,  
using $A_N^0(i,j)=\boldsymbol{1}_{\{i=j\}}$ and Minkowski's  inequality implies that on $\Omega_{N,K}$,
\begin{align*}
\mathbb{E}_{\theta}[(U^{i,N}_{t}-U_{s}^{i,N})^4]^\frac{1}{4}
\leq& \mathbb{E}_{\theta}[(M^{i,N}_{t}-M_{s}^{i,N})^4]^\frac{1}{4}\\
&+ \sum_{n\geq 1} \int_0^{t} \Big(\phi^{*n}(t-u)- \phi^{*n}(s-u)\Big)
\Et\Big[\Big(\sum_{j=1}^N A_N^n(i,j)M^{j,N}_u\Big)^4\Big]^\frac{1}{4} du.
\end{align*}
For the first term ($n=0$),  an application of  (\ref{ee3}) and Burkholder's inequality gives
\begin{align*}
\mathbb{E}_{\theta}[(M^{i,N}_{t}-M_{s}^{i,N})^4]
\le& C\mathbb{E}_{\theta}\Big[(Z^{i,N}_{t}-Z_{s}^{i,N})^2\Big].
\end{align*}
By \cite[Lemma 16-(iii)]{A}, on $\Omega_{N,K}$, we have,
$\max_{i=1,\dots,N} \Et [ (Z^{i,N}_{t}-Z^{i,N}_s)^2]\le C(t-s)^2,$
and therefore
\begin{align}\label{M4}
\mathbb{E}_{\theta}[(M^{i,N}_{t}-M_{s}^{i,N})^4]\le C (t-s)^2.
\end{align}
For the second term ($n\ge 1$), another application of (\ref{ee3}) and  Burkholder's inequality yields  
\begin{align*}
    \Et\Big[\Big(\sum_{j=1}^N A_N^n(i,j)M^{j,N}_u\Big)^4\Big]\le&  C\Et\Big[\Big(\big[ \sum_{j=1}^N A_N^n(i,j)M^{j,N},\sum_{j=1}^N A_N^n(i,j)M^{j,N} \big]_u\Big)^2\Big] \\
    \le& C\Et\Big[\Big(\sum_{j=1}^N( A_N^n(i,j))^2Z^{j,N}_u\Big)^2\Big]\\
=&C \sum_{j,j'=1}^N (A_N^n(i,j))^2(A_N^n(i,j'))^2 \Et[ Z^{j,N}_u Z^{j',N}_u].
\end{align*}
By the Cauchy-Schwarz inequality, $\Et[ Z^{j,N}_u Z^{j',N}_u]\le \sqrt{\Et[ (Z^{j,N}_u)^2]\Et[ (Z^{j',N}_u)^2]}$,  and from  \cite[Lemma 16-(iii)]{A}, we have $\max_{i=1,\dots,N} \Et [ (Z^{i,N}_u)^2]\le Cu^2$.  Therefore,
\begin{align*}
\Et\Big[\Big(\sum_{j=1}^N A_N^n(i,j)M^{j,N}_u\Big)^4\Big]\le  C\Big(\sum_{j=1}^N( A_N^n(i,j))^2\Big)^2u^2
    \le C\Big(\sum_{j=1}^NA_N^n(i,j)\Big)^4u^2
    \le C|||A_{N}|||^{4n}_{\infty}u^2.
\end{align*}
This implies that 
\begin{align*}
    & \sum_{n\geq 1} \int_0^{\infty} \Big(\phi^{*n}(t-u)- \phi^{* n}(s-u)\Big)
\Et\Big[\Big(\sum_{i=1}^K\sum_{j=1}^N A_N^n(i,j)M^{j,N}_u\Big)^4\Big]^\frac{1}{4} du\\
\le& C \sum_{n\geq 1} |||A_N|||_\infty^{n}\intot \sqrt{u} \Big(\phi^{*n}(t-u)- \phi^{*n}(s-u)\Big) du\\
\leq& C(t-s)^{1/2} \sum_{n\geq 1} \Lambda^n |||A_N|||_\infty^{n} \leq C (t-s)^{1/2}.
\end{align*}
To justify the second inequality, we use  the following estimate, which holds for all $n \ge 1$:
\begin{align*}
\intot \sqrt{u} (\phi^{*n}(t-u)- \phi^{* n}(s-u)) du=&\intot \sqrt{t-u} \phi^{* n}(u) du - 
\int_0^s \sqrt{s-u}\phi^{* n}(u) du\\
\leq & \int_0^s [\sqrt{t-u}-\sqrt{s-u}] \phi^{* n}(u) du +\int_s^t \sqrt{t-u} \phi^{* n}(u) du \\
\leq & 2 \sqrt{t-s} \int_0^\infty \phi^{* n}(u) du \leq 2 \Lambda^n \sqrt{t-s}.
\end{align*}
Moreover, on $\Omega_{N,K}$, we have $\Lambda |||A_N|||_\infty \leq a<1$. This completes the proof of of  the first part of (iii).

\vip

For the second part, we recall from Lemma \ref{Zt}-(ii) with $K=N$ and $r=\infty$ that we have
$\max_{i=1,\dots,N}\Et[Z^{i,N}_{t}-Z^{i,N}_{s}]\le C(t-s)$ on $\Omega_{N,N}\supset \Omega_{N,K}$,  it follows that
$$
\max_{i=1,\dots,N}\Et[(Z^{i,N}_{t}-Z^{i,N}_{s})^4]\le 8\Big\{\max_{i=1,\dots,N}\Et[Z^{i,N}_{t}-Z^{i,N}_{s}]^4
+\max_{i=1,\dots,N}\Et[(U^{i,N}_{t}-U^{i,N}_{s})^4]\Big\}\le C (t-s)^{4}
$$
as desired.

\section{Proof of Lemma \ref{simple}}\label{app: prof: 4.4}
Recall that $\bX_{N}^{K}=(X_N^K(i))_{i=1,\dots,N}$ with $X_{N}^{K}(i)=(L_{N}(i)-\bar{L}^{K}_{N})\indiq_{\{i\leq K\}}$, where $L_{N}(i):=\sum_{j=1}^{N}A_{N}(i,j)=\frac{1}{N}\sum_{j=1}^{N}\theta_{ij}$ and that $\bar{L}^K_N:=\frac{1}{K}\sum_{i=1}^KL_N(i)$ and $\boldsymbol{X}_{N}:=\boldsymbol{X}_{N}^{N}$ defined in Section \ref{subimn}. Here, $(\theta_{ij})_{i,j=1,...,N}$ is a family  of i.i.d. Bernoulli($p$) random variables, $\boldsymbol{x}^{K}_{N}=
(x_N^K(i))_{i=1,\dots,N}$ with $x_{N}^{K}(i)=(\ell_{N}(i)-\bar{\ell}^{K}_{N})\indiq_{\{i\leq K\}}$, and  $\boldsymbol{x}_{N}$ are defined in Section \ref{subimn}.

\subsection{ Proof of Lemma \ref{simple} (i)} 
Since $(\theta_{ij})_{i,j=1,...,N}$ are  i.i.d. Bernoulli($p$) random variables.  By symmetry, we have
\begin{align*}
\mathbb{E}[\|(I_{K}A_{N})^{T}\bX_{N}^{K}\|_{2}^{2}]
=&\frac{K}{N^{2}}\mathbb{E}\Big[\Big(\sum_{j=1}^{K}\theta_{j1}(L_{N}(j)-\bar{L}_{N}^{K})\Big)^{2}\Big]\\
\le& \frac{2K}{N^{2}}\Big\{\mathbb{E}\Big[\Big(\sum_{j=1}^{K}\theta_{j1}(L_{N}(j)-p)\Big)^{2}\Big]+\mathbb{E}\Big[\Big(\sum_{j=1}^{K}\theta_{j1}(p-\bar{L}_{N}^{K})\Big)^{2}\Big]\Big\}.
\end{align*}
On the one hand, since $\bar{L}^K_N=\frac{1}{NK}\sum_{i=1}^K\sum_{j=1}^N \theta_{ij}$ with 
$(\theta_{ij})_{i,j=1,\dots,N}$ being i.i.d. Bernoulli($p$) random variables and since $\theta_{j1} \leq 1$, it directly follows that
$$
\frac{K}{N^{2}}\mathbb{E}\Big[\Big(\sum_{j=1}^{K}\theta_{j1}(p-\bar{L}_{N}^{K})\Big)^{2}\Big]\le \frac{K^{3}}{N^{2}}\mathbb{E}[(p-\bar{L}_{N}^{K})^{2}]\le \frac{CK^{2}}{N^{3}}.
$$
On the other hand, writing  
$L_{N}(j)=\frac{1}{N}\sum_{i=1}^{N}\theta_{ji}=\frac{1}{N}\big(\sum_{i=2}^{N}\theta_{ji}+\theta_{j1}\big)$, we obtain
\begin{align*}
    &\frac{K}{N^{2}}\mathbb{E}\Big[\Big(\sum_{j=1}^{K}\theta_{j1}(L_{N}(j)-p)\Big)^{2}\Big]\\
    \le& \frac{2K}{N^{4}}\Big\{\mathbb{E}\Big[\Big(\sum_{j=1}^{K}\sum_{i=2}^{N}\theta_{j1}(\theta_{ji}-p)\Big)^{2}\Big]
    +\mathbb{E}\Big[\Big(\sum_{j=1}^{K}\theta_{j1}(\theta_{j1}-p)\Big)^{2}\Big]\Big\}\\
    \le& \frac{4K}{N^{4}}\Big\{\mathbb{E}\Big[\Big(\sum_{j=1}^{K}\sum_{i=2}^{N}(\theta_{j1}-p)(\theta_{ji}-p)\Big)^{2}\Big]+p^{2}\mathbb{E}\Big[\Big(\sum_{j=1}^{K}\sum_{i=2}^{N}(\theta_{ji}-p)\Big)^{2}\Big]+\mathbb{E}\Big[\Big(\sum_{j=1}^{K}\theta_{j1}(\theta_{j1}-p)\Big)^{2}\Big]\Big\}.
\end{align*}
Applying the family $\{(\theta_{ji}-p), i=2,\dots,N,j=1,\dots,K\}$ is independent and centered, it yields that 
\begin{align*}
&\mathbb{E}\Big[\Big(\sum_{j=1}^{K}\sum_{i=2}^{N}(\theta_{j1}-p)(\theta_{ji}-p)\Big)^{2}\Big]\\
&=\mathbb{E}\Big[\sum_{j,j'=1}^{K}\sum_{i,i'=2}^{N}(\theta_{j1}-p)(\theta_{j'1}-p)(\theta_{ji}-p)(\theta_{ji'}-p)\Big]\\
&=\mathbb{E}\Big[\sum_{j=1}^{K}\sum_{i=2}^{N}(\theta_{j1}-p)^2(\theta_{ji}-p)^2\Big]\le CNK.
\end{align*}
Similarly, we have $\mathbb{E}[(\sum_{j=1}^{K}\sum_{i=2}^{N}(\theta_{ji}-p))^{2}]\le CNK$. Furthermore,  
since $\theta_{j1} \leq 1$ and $|\theta_{j1}-p| \leq 1$, we obtain  $\mathbb{E}[(\sum_{j=1}^{K}\theta_{j1}(\theta_{j1}-p))^{2}]\le CK^2$. Consequently,  
 $\frac{K}{N^{2}}\mathbb{E}\Big[\Big(\sum_{j=1}^{K}\theta_{j1}(L_{N}(j)-p)\Big)^{2}\Big]\le C K^2/N^3$, as desired.
 
\subsection{ Proof of Lemma \ref{simple} (ii)}

By the definition of $\bX_{N}$ and $\bX^K_{N}$ in Section \ref{subimn}, we have 
\begin{align*}
(I_{K}A_{N}\bX_{N},\bX_{N}^{K})=\sum_{i,j=1}^{K}A_N(i,j)X_{N}(j)X_{N}^{K}(i).
\end{align*}
Since $A_N(i,j)=\frac{1}{N}\theta_{ij}$ and $\sum_{i=1}^KX_{N}^{K}(i)=\sum_{i=1}^K(L_{N}(i)-\bar{L}^{K}_{N})=0$, we have 
\begin{align*}
 (I_{K}A_{N}\bX_{N},\bX_{N}^{K})=&\frac{1}{N}\sum_{i,j=1}^{K}(\theta_{ij}-p)X_{N}(j)X_{N}^{K}(i)\\
 =&\frac{1}{N}\Big[\sum_{i,j=1}^{K}(\theta_{ij}-p)(L_{N}(j)-p)X_{N}^{K}(i)+(p-\bar{L}_{N})\sum_{i,j=1}^{K}(\theta_{ij}-p)X_{N}^{K}(i)\Big]\\
 =&\frac{1}{N}\Big[\sum_{i,j=1}^{K}(\theta_{ij}-p)(L_{N}(j)-p)(L_{N}(i)-p)+(p-\bar{L}_{N}^{K})\sum_{i,j=1}^{K}(\theta_{ij}-p)(L_{N}(j)-p)\\
 &\qquad+(p-\bar{L}_{N})\sum_{i,j=1}^{K}(\theta_{ij}-p)(L_{N}(i)-p)+(p-\bar{L}_{N})(p-\bar{L}_{N}^{K})\sum_{i,j=1}^{K}(\theta_{ij}-p)\Big].
\end{align*}

We start by analyzing the first term. Recalling that $L_{N}(i):=\sum_{j=1}^{N}A_{N}(i,j)=\frac{1}{N}\sum_{j=1}^{N}\theta_{ij}$, we obtain 
\begin{align*}
    &\mathbb{E}\Big[\Big(\sum_{i,j=1}^{K}(\theta_{ij}-p)(L_{N}(j)-p)(L_{N}(i)-p)\Big)^{2}\Big]\\
    =&\frac{1}{N^{4}}\mathbb{E}\Big[\Big(\sum_{i,j=1}^{K}\sum_{m,n=1}^{N}
    (\theta_{ij}-p)(\theta_{jm}-p)(\theta_{in}-p)\Big)^{2}\Big]\\
   =&\frac{1}{N^{4}}\mathbb{E}\Big[\sum_{i,j,i',j'=1}^{K}\sum_{m,n,m',n'=1}^{N}(\theta_{ij}-p)(\theta_{jm}-p)(\theta_{in}-p)(\theta_{i'j'}-p)(\theta_{j'm'}-p)(\theta_{i'n'}-p)\Big] \le \frac{CK^{2}}{N^{2}},
\end{align*}
since the family $\{(\theta_{ij}-p),i,j=1,\dots,N\}$ is i.i.d., centered, and bounded.
\vip
For the second term, applying the Cauchy–Schwarz inequality yields
\begin{align*}
\mathbb{E}\Big[\Big|(p-\bar{L}_{N}^{K})\sum_{i,j=1}^{K}(\theta_{ij}-p)(L_{N}(j)-p)\Big|\Big]
\le \frac{1}{N}\mathbb{E}[(p-\bar{L}_{N}^{K})^{2}]^\frac{1}{2}
\mathbb{E}\Big[\Big(\sum_{i,j=1}^{K}\sum_{k=1}^{N}(\theta_{ij}-p)(\theta_{jk}-p)\Big)^{2}\Big]^{\frac{1}{2}}.
\end{align*}
This quantity is bounded  by $\frac{\sqrt{K}}{N}$, since on the one hand, we have
\begin{align}\label{varL}
\mathbb{E}[(p-\bar{L}_{N}^{K})^{2}]=\frac{1}{N^2K^2}\mathbb{E}[(\sum_{i=1}^K\sum_{j=1}^N(\theta_{ij}-p))^2]
=\frac{\mathbb{E}[(\theta_{11}-p)^2]}{NK}\le\frac{C}{NK},
\end{align}
 and on the other hand,
\begin{align*}
    &\mathbb{E}\Big[\Big(\sum_{i,j=1}^{K}\sum_{k=1}^{N}(\theta_{ij}-p)(\theta_{jk}-p)\Big)^{2}\Big]\\
    =&\mathbb{E}\Big[\sum_{i,j,i',j'=1}^{K}\sum_{k,k'=1}^{N}(\theta_{ij}-p)(\theta_{i'j'}-p)(\theta_{jk}-p)(\theta_{j'k'}-p)\Big]\le CNK^2.
\end{align*}

For the third term, by \eqref{varL},
we have $\E[(p-\bar{L}_{N})^{2}]=\E[(p-\bar{L}_{N}^N)^{2}]\leq \frac C{N^2}$.  Applying the Cauchy-Schwarz inequality, we obtain
\begin{align*}
&\mathbb{E}\Big[\Big|(p-\bar{L}_{N})\sum_{i,j=1}^{K}(\theta_{ij}-p)(L_{N}(i)-p)\Big|\Big]\\
\le& \frac{1}{N}\mathbb{E}[(p-\bar{L}_{N})^{2}]^{\frac{1}{2}}\mathbb{E}\Big[\Big(\sum_{i,j=1}^{K}\sum_{k=1}^{N}(\theta_{ij}-p)(\theta_{ik}-p)\Big)^{2}\Big]^{\frac{1}{2}}\\
\le& \frac{1}N \sqrt\frac{C}{N^2} \mathbb{E}\Big[\sum_{i,j,i',j'=1}^{K}\sum_{k,k'=1}^{N}(\theta_{ij}-p)(\theta_{ik}-p)(\theta_{i'j'}-p)(\theta_{i'k'}-p)\Big]^{\frac{1}{2}}\\
\le&  \frac{1}N \sqrt\frac{C}{N^2} \sqrt{K^2 N+K^4}= C \frac K{N^{3/2}}+C \frac {K^2}{N^2}\le C \frac K{N}. 
\end{align*}

\vip

Finally, we analyze the last term. Note that
$\mathbb{E}[(\sum_{i,j=1}^{K}(\theta_{ij}-p))^{2}]=\mathbb{E}[\sum_{i,j=1}^{K}(\theta_{ij}-p)^{2}]=CK^2$ 
and $\mathbb{E}[(p-\bar{L}_{N}^{K})^{4}]=\frac{1}{N^4K^4}\E[(\sum_{i=1}^K\sum_{j=1}^N(\theta_{ij}-p))^4]
\le\frac{C}{N^2K^2}$. Therefore, applying the  Cauchy–Schwarz inequality again, we obtain
\begin{align*}
    &\mathbb{E}\Big[\Big|(p-\bar{L}_{N})(p-\bar{L}_{N}^{K})\sum_{i,j=1}^{K}(\theta_{ij}-p)\Big|\Big]\\
    \le& \mathbb{E}[(p-\bar{L}_{N})^{4}]^{\frac{1}{4}}\mathbb{E}[(p-\bar{L}_{N}^{K})^{4}]^{\frac{1}{4}}\mathbb{E}\Big[\Big(\sum_{i,j=1}^{K}(\theta_{ij}-p)\Big)^{2}\Big]^{\frac{1}{2}}\\
    \le& C \Big(\frac1 {N^4} \Big)^{1/4}\Big(\frac1 {N^2 K^2} \Big)^{1/4} \sqrt{K^2}=\frac{\sqrt K}{ N \sqrt N}\le C\frac{K}{ N }.
\end{align*}
Together, the preceding arguments complete the proof.

\subsection{ Proof of Lemma \ref{simple} (iii)} Recall that 
$\boldsymbol{x}^{K}_{N}=
(x_N^K(i))_{i=1,\dots,N}$, where  $x_{N}^{K}(i)=(\ell_{N}(i)-\bar{\ell}^{K}_{N})\indiq_{\{i\leq K\}}$, and that $\bX^K_{N}$, $\ell_{N}(i)$,  $\bar{\ell}^{K}_{N}$ and $\bar{\ell}_{N}$ are defined in Section \ref{subimn}.  Here, $\cA_{N}$ is defined in \eqref{mA}. For any $\boldsymbol{x}, \boldsymbol{y}\in \R^N$, $a>0$, it's not hard to verify the following elementary equality
\[\big| \|\boldsymbol{x}\|_2^2-a^2\|\boldsymbol{y}\|_2^2-\|\boldsymbol{x}-a\boldsymbol{y}\|_2^2 \big|=2\big| a \big(\boldsymbol{x}-a\boldsymbol{y}, \boldsymbol{y}\big)\big|.\]
Then, putting  $\boldsymbol{x}=\boldsymbol{x}^{K}_{N}$, $\boldsymbol{y=}\bX^{K}_{N}$ and $a=\Lambda\bar{\ell}_{N}$, we have 
\begin{align*}
&\mathbb{E}\Big[\boldsymbol{1}_{\Omega_{N,K}\cap \cA_{N}}\Big|\Big(\|\bx^{K}_{N}\|_{2}^{2}-(\Lambda\bar{\ell}_{N})^{2}\|\bX_{N}^{K}\|^{2}_{2}\Big)-\|\bx_{N}^{K}-\Lambda\bar{\ell}_{N} \bX_{N}^{K}\|^{2}_{2}\Big|\Big]\\
&=2\mathbb{E}\Big[\boldsymbol{1}_{\Omega_{N,K}\cap \cA_{N}}\Big|\Lambda\bar{\ell}_{N}\big( \bx^{K}_{N}-\Lambda\bar{\ell}_{N}\bX_{N}^{K}, \bX_{N}^{K} \big)\Big|\Big].
\end{align*}
By \cite[Lemma 5.11]{D}, it holds that 
$$\boldsymbol{x}_{N}^{K}-\Lambda\bar{\ell}_{N}\boldsymbol{X}_{N}^{K}
=\Lambda I_{K}A_{N}(\boldsymbol{x}_{N}-\Lambda\bar{\ell}_{N}\boldsymbol{X}_{N})-\frac{\Lambda}{K}(I_{K}A_{N}\boldsymbol{x}_{N}, \boldsymbol{1}_{K})\boldsymbol{1}_{K}+\bar{\ell}_{N}\Lambda^{2} I_{K}A_{N}\boldsymbol{X}_{N}.
$$
Since $X_{N}^{K}(i)=(L_{N}(i)-\bar{L}^{K}_{N})\indiq_{\{i\leq K\}}$ and $\boldsymbol{1}_K$ is a vector with entries  $\boldsymbol{1}_K(i)=\indiq_{\{1\leq i\leq K\}}$,  it follows that  $(\boldsymbol{1_{K}},\bX_{N}^{K})=\sum_{i=1}^KX_{N}^{K}(i)=0$.  Consequently,
\begin{align*}
(\bx_{N}^{K}-\Lambda\bar{\ell}_{N}\bX_{N}^{K},\bX_{N}^{K})=&\Lambda (I_{K}A_{N}(\bx_{N}-\Lambda \bar{\ell}_{N}\bX_{N}),\bX_{N}^{K})
+\Lambda^2 \bar{\ell}_{N}(I_{K}A_{N}\bX_{N},\bX_{N}^{K}):= e_{N,K,1}+e_{N,K,2}.
\end{align*}
An application of the Cauchy–Schwarz inequality gives
\begin{align*}
\mathbb{E}[\boldsymbol{1}_{\Omega_{N,K}\cap \cA_{N}}e_{N,K,1} ]  =&\mathbb{E}\Big[\boldsymbol{1}_{\Omega_{N,K}\cap \cA_{N}}\Lambda \Big((\bx_{N}-\Lambda \bar{\ell}_{N}\bX_{N}),(I_{K}A_{N})^{T}\bX_{N}^{K}\Big)\Big]\\
    \le& \Lambda \mathbb{E}\Big[\boldsymbol{1}_{\Omega_{N,K}\cap \cA_{N}}\|\bx_{N}-\Lambda \bar{\ell}_{N}\bX_{N}\|^2_{2}\Big]^\frac{1}{2}\E\Big[\|(I_{K}A_{N})^{T}\bX_{N}^{K}\|^2_{2}\Big]^\frac{1}{2}.
\end{align*}
From \cite[Lemma 5.11]{D}, the first term $\mathbb{E}\Big[\boldsymbol{1}_{\Omega_{N,K}\cap \cA_{N}}\|\bx_{N}-\Lambda \bar{\ell}_{N}\bX_{N}\|^2_{2}\Big]$ is bounded by $C N^{-1}$, and Lemma \ref{simple}-(i) bounds the second, yielding
\[\mathbb{E}[\boldsymbol{1}_{\Omega_{N,K}\cap \cA_{N}}e_{N,K,1} ] \le \frac{CK}{N^2}. \]
Furthermore, $\bar{\ell}_N$ is bounded on $\Omega_{N,K}$ by Lemma \ref{lo}, and  Lemma  \ref{simple}-(ii) implies 
\begin{align*}
\mathbb{E}[\boldsymbol{1}_{\Omega_{N,K}}e_{N,K,2}]\leq \frac{CK}{N^2},
\end{align*}
which completes the proof.

\subsection{ Proof of Lemma \ref{simple} (iv)}

Recall  from \cite[Proposition 14]{A} that  $\E[\indiq_{\Omega_N^1} |\bar\ell_N - \frac{1}{1-\Lambda p} |^2] \leq \frac C {N^2}$. Furthermore, Lemma \ref{lo} guarantees that $\bar{\ell}_N$ is bounded by a constant $C$ on $\Omega_{N,K}$. In addition, one can verify (see, e.g., \cite[Equation (9)]{D}) that
$\mathbb{E}[\frac{N^2}{K^2} \|\boldsymbol{X}_{N}^{K}\|_{2}^{4}]\le C$. Together with the Cauchy–Schwarz inequality, these results  imply
\begin{align*}
&\sqrt{K}\frac{N}{K}\E\Big[\omg\Big|(\bar{\ell}_{N})^{2}-\Big(\frac{1}{1-\Lambda p}\Big)^{2}\Big|\|\bX_{N}^{K}\|_{2}^{2}\Big]\\
\le& C \sqrt{K}\mathbb{E}\Big[\frac{N^2}{K^2} \|\boldsymbol{X}_{N}^{K}\|_{2}^{4}\Big]^\frac{1}{2}\E\Big[\omg\Big|\bar{\ell}_{N}-\frac{1}{1-\Lambda p}\Big|^{2}\Big]^\frac{1}{2}\le\frac{C\sqrt{K}}{N} \leq \frac{C}{\sqrt{N}}.
\end{align*}
To complete the proof, it  suffices to show that
$\omg\sqrt{K}[\frac{N}{K}\|\bX_{N}^{K}\|_{2}^{2}-p(1-p)]\stackrel{d}{\longrightarrow} \mathcal{N}(0,p^2(1-p)^2).$
Since $\omg$ tends to $1$ in probability,  it is enough to verify that
$\sqrt{K}[\frac{N}{K}\|\bX_{N}^{K}\|_{2}^{2}-p(1-p)]\stackrel{d}{\longrightarrow} \mathcal{N}(0,p^2(1-p)^2).$
Now observe that
$$
\|\bX_{N}^{K}\|_{2}^{2}=\sum_{i=1}^{K}(L_{N}(i)-\bar{L}^{K}_{N})^{2}=\sum_{i=1}^{K}(L_{N}(i)-p)^{2}-K(p-\bar{L}_{N}^{K})^{2}.
$$
As shown in the proof of Lemma \ref{simple}-(ii), we have  $\mathbb{E}[(p-\bar{L}_{N}^{K})^{2}]\le \frac{C}{NK},$ 
so that $\sqrt K \frac NK \mathbb{E}[K(p-\bar{L}_{N}^{K})^{2}]\le \frac{C}{\sqrt K}$. Therefore, it remains to show that
$$
\xi_{N,K}:=\sqrt{K}\Big[\frac{N}{K}\sum_{i=1}^{K}(L_{N}(i)-p)^{2}-p(1-p)\Big]\stackrel{d}{\longrightarrow} 
\mathcal{N}(0,p^2(1-p)^2).
$$
Recalling that $L_{N}(i)=N^{-1}\sum_{j=1}^N \theta_{ij}$, a direct computation shows that
\begin{align*}
\xi_{N,K}= \frac 1{N\sqrt K}\sum_{i=1}^{K}\sum_{j=1}^{N}[(\theta_{ij}-p)^{2}-p(1-p)]
+ \frac 1{N\sqrt K} \sum_{i=1}^{K}\sum_{j=1}^N \sum_{j'=1,j'\neq j}^{N}(\theta_{ij}-p)(\theta_{ij'}-p).
\end{align*}
Since the variables ${(\theta_{ij}-p)^2 - p(1-p)}$ are i.i.d. with mean zero and finite variance, the central limit theorem implies the convergence in distribution of 
$\frac 1{\sqrt {NK}}\sum_{i=1}^{K}\sum_{j=1}^{N}[(\theta_{ij}-p)^{2}-p(1-p)]$. Therefore, the first term tends to $0$ in probability.  For the second term, applying the the central limit theorem again, we obtain
\begin{align*}
    \frac{1}{N\sqrt{K}}\sum_{i=1}^{K}\sum_{j=1}^N \sum_{j'=1,j'\neq j}^{N}(\theta_{ij}-p)(\theta_{ij'}-p)\stackrel{d}{\longrightarrow} \mathcal{N}(0,p^{2}(1-p)^{2}),
\end{align*}
which completes the proof.

\section{convolution of $\phi$}
We first present two lemmas concerning the convolution of the function $\phi$ introduced in Section \ref{notation}. These will be useful in proving Lemmas \ref{covc} and \ref{mathY}.

\begin{lemma}\label{phisq}
We consider 
$\phi:[0,\infty)\to [0,\infty)$ such that $\Lambda=\int_0^\infty \phi(s)ds <\infty$ and, for some $q\ge 1$, 
$\int_0^\infty s^q\phi(s)ds <\infty.$ 
Then, for all $n\ge 1$ and $r\ge 1,$
$$
\int_{r}^{\infty}\sqrt{s}\phi^{*n}(s)ds\le  C \Lambda^{n}n^qr^{\frac{1}{2}-q} 
\quad \hbox {and} \quad
\int^{\infty}_{0}\sqrt{s}\phi^{*n}(s)ds\le C\sqrt{n}\Lambda^{n}.
$$
\end{lemma}

\begin{proof}
We introduce some i.i.d. random variables $X_1,X_2,\dots$ with density $\Lambda^{-1}\phi$ and
set $S_0=0$ as well as $S_n=X_1+\dots+X_n$ for all $n\geq 1$. By the Minkowski inequality and since  $\E[X_1^q]=\Lambda^{-1}\int_0^\infty s^q \phi(s)ds<\infty$ by assumption, we obtain $\E[S_n^q]\le n^q\E[X_1^q]\le Cn^q$. Consequently,
\begin{align*}
\int_{r}^{\infty}\sqrt{s}\phi^{*n}(s)ds=
\Lambda^{n}\E[\sqrt{S_n}\indiq_{\{S_n\ge r\}}] \leq \Lambda^{n}r^{\frac{1}{2}-q}\E[S_n^q]\leq 
\Lambda^{n}n^qr^{\frac{1}{2}-q} \E[X_1^q]\leq C \Lambda^{n}n^qr^{\frac{1}{2}-q}.
\end{align*}
For the second part, we write
$$
\int^{\infty}_{0}\sqrt{s}\phi^{*n}(s)ds=\Lambda^{n}\E[\sqrt{S_n}] \le \sqrt{n}\Lambda^{n}\sqrt{\E[{X_1}]}\le C\sqrt{n}\Lambda^{n}
$$
by the Cauchy-Schwarz inequality.
\end{proof}
\begin{lemma}\label{phi}
Under the same conditions as in Lemma \ref{phisq}, we have,
for $n\in \mathbb{N}_+$ and $r\ge 1,$
$$\Big|\int_{0}^{t}\phi^{*n}(s)ds-\Lambda^{n}\Big|\le n\Lambda^{n-1} \int^\infty_{\frac{t}{n}}\phi(s)ds.$$
\end{lemma}

\begin{proof}
Consider $n$ i.i.d random variables $\{X_i\}_{i=1,...,n}$ with density $\phi(s)/\Lambda$ and write
\begin{align*}
\Big|\int_{0}^{t}\phi^{*n}(s)ds-\Lambda^{n}\Big|=\Lambda^n \mathbb{P} \Big(\sum_{i=1}^n X_i\ge t\Big)
\leq \Lambda^n \mathbb{P}  \Big(\max_{i=1,\dots,n} X_i\ge t/n\Big) \leq 
n \Lambda^{n} \mathbb{P}  (X_1 \geq t/n),
\end{align*}
which completes the proof.
\end{proof}

\section{Proof of Lemmas \ref{covc} \& \ref{mathY}}\label{app: prof: 6.4-6}

\subsection{Proof of Lemma \ref{covc} (i)}
We work on  the set $\Omega_{N,K}$.
Recall (\ref{ee3}). By \cite[Lemma 16-(iii)]{A} and the Cauchy-Schwarz inequality, we have
\begin{align*}
\mathbb{E}_{\theta}[(M^{j,N}_{(a\Delta-s)}-M^{j,N}_{a\Delta})(M_{(a\Delta-s')}^{j',N}-M_{a\Delta}^{j',N})]&\le \indiq_{\{j=j'\}}\mathbb{E}_{\theta}[Z^{j,N}_{a\Delta}-Z^{j,N}_{(a\Delta-s)}]^\frac{1}{2}\mathbb{E}_{\theta}[Z^{j,N}_{a\Delta}-Z^{j,N}_{(a\Delta-s')}]^\frac{1}{2}\\&\le \indiq_{\{j=j'\}}\sqrt{ss'}.
\end{align*}
From Lemma \ref{phisq}, we already have
$
\int_r^\infty \sqrt{u}\phi^{* n}(u)du \leq C \Lambda^{n}n^qr^{\frac{1}{2}-q}.$ Recalling \eqref{Bdelta}, we obtain
\begin{align*}
\mathbb{E}_{\theta}[(B_{a\Delta}^{N,K})^{2}]&= \frac{1}{K^{2}}\sum_{i,i'=1}^{K}\sum_{j,j'=1}^{N}\sum_{n,m\ge 1}\int^{a\Delta}_{\Delta}\int^{a\Delta}_{\Delta}\phi^{*n}(s)\phi^{*m}(s')A_{N}^{n}(i,j)A_{N}^{m}(i',j')\\
&\hskip4cm\times\mathbb{E}_{\theta}[(M^{j,N}_{(a\Delta-s)}-M^{j,N}_{a\Delta})(M_{(a\Delta-s')}^{j',N}-M_{a\Delta}^{j',N})]dsds'\\
&\le \frac{CN}{K^2}\Big(\sum_{n,m\ge 1} |||I_KA_{N}|||^{2}_1|||A_{N}|||^{n+m-2}_{1}\int_{\Delta}^{\infty}\int_{\Delta}^{\infty}\sqrt{ss'}\phi^{*n}(s)\phi^{*m}(s')dsds'\Big)\\
&\le \frac{C}{N}\Big(\sum_{n\ge 1} n^q\Lambda^{n}|||A_{N}|||^{n-1}_{1}\Big)^{2}\Delta^{1-2q}\le \frac{C}{N}\Delta^{1-2q},
\end{align*}
where the last inequality follows because on $\Omega_{N,K}$, $\Lambda |||A_{N}|||_{1}\le a<1$ and $\sum_{n\ge 1} n^qa^{n-1}<\infty$.

\subsection{Proof of Lemma \ref{covc} (ii)} Recalling \eqref{Cdelta}, we write 
$$
C_{a\Delta}^{N,K}=\sum_{n\ge 1}\int^{\Delta}_{0}\phi^{*n}(s)O_{s,s,a\Delta}^{N,K,n},
$$
where for $r\ge 0$ and $0\le s\le a\Delta,$
$$
O_{r,s,a\Delta}^{N,K,n}:=\frac{1}{K}\sum_{i=1}^{K}\sum_{j=1}^{N}A_{N}^{n}(i,j)(M_{(a\Delta-s)}^{j,N}-M^{j,N}_{a\Delta-s+r}).
$$
For fixed $s$, $\{O_{r,s,a\Delta}^{N,K,n}\}_{r\ge 0}$ is a family of martingale w.r.t the filtration $(\mathcal{F}_{a\Delta-s+r})_{r\ge 0}$.
By (\ref{ee3}), we have $[M^{i,N},M^{j,N}]_{t}=\boldsymbol{1}_{\{i=j\}}Z_{t}^{i,N}$. Hence, 
for $n\ge 1$, on $\Omega_{N,K},$
\begin{align*}
[O_{.,s,a\Delta}^{N,K,n},O_{.,s,a\Delta}^{N,K,n}]_{r}=&\frac{1}{K^{2}}\sum_{j=1}^{N}\Big(\sum_{i=1}^{K}A^{n}_{N}(i,j)\Big)^{2}(Z_{a\Delta-s+r}^{j,N}-Z_{a\Delta-s}^{j,N})
\\\le& \frac{N}{K^{2}}|||I_{K}A^n_{N}|||^{2}_{1}(\bar{Z}_{a\Delta-s+r}^{N}-\bar{Z}_{a\Delta-s}^{N})\\
\le& \frac{N}{K^{2}}|||I_{K}A_{N}|||^{2}_{1}|||A_{N}|||^{2n-2}_{1}(\bar{Z}_{a\Delta-s+r}^{N}-\bar{Z}_{a\Delta-s}^{N}).
\end{align*}
Since $|||I_{K}A_{N}|||^{2}_{1}\le \frac{1}{\Lambda^2}\frac{K^2}{N^2}$ on $\Omega_{N,K}$, we obtain 
\[[O_{.,s,a\Delta}^{N,K,n},O_{.,s,a\Delta}^{N,K,n}]_{r}\le \frac{C}{N}|||A_{N}|||^{2n-2}_{1}(\bar{Z}_{a\Delta-s+r}^{N}-\bar{Z}_{a\Delta-s}^{N}).\]
Applying the Burkholder-Davis-Gundy inequality yields, on $\Omega_{N,K}$,
\begin{align*}
    \mathbb{E}_\theta[(O_{r,s,a\Delta}^{N,K,n})^4]\le 4\mathbb{E}_\theta\Big[\Big([O_{,s,a\Delta}^{N,K,n},O_{,s.a\Delta}^{N,K,n}]_{r}\Big)^2\Big]\le \frac{C|||A_{N}|||^{4n-4}_{1}}{N^2}\mathbb{E}_\theta[(\bar{Z}_{a\Delta-s+r}^{N}-\bar{Z}_{a\Delta-s}^{N})^2].
\end{align*}
From \cite [lemma 16-(iii)]{A}, we already have $\sup_{i=1,\dots,N}\Et[(Z^{i,N}_t-Z^{i,N}_s)^2] \leq C (t-s)^2.$
Using the second part of Lemma \ref{phisq} together with  the Minkowski inequality, we obtain
\begin{align*}
\mathbb{E}_{\theta}[(C_{a\Delta}^{N,K})^{4}]^\frac{1}{4}\le& \sum_{n\ge 1}\int^{\Delta}_{0}\phi^{*n}(s)\mathbb{E}_\theta[(O_{s,s,a\Delta}^{N,K,n})^4]^\frac{1}{4}ds\\
\le& \frac{1}{\sqrt{N}}\sum_{n\ge 0}|||A_{N}|||^{n}_{1}\int^{\Delta}_{0}\sqrt{s}\phi^{*(n+1)}(s)ds\\
\le& \frac{1}{\sqrt{N}}\sum_{n\ge 0}\sqrt{n+1}\,\Lambda^{n+1}|||A_{N}|||^{n}_{1}\le \frac{C}{\sqrt{N}}.
\end{align*}
This completes the proof.

\subsection{Proof of Lemma \ref{covc} (iii)}

Because
\begin{align*}
&\mathbb{C}ov_{\theta}[(C_{a\Delta}^{N,K}-C_{(a-1)\Delta}^{N,K})^{2},(C_{b\Delta}^{N,K}-C_{(b-1)\Delta}^{N,K})^{2}]\\
=&\frac{1}{K^{4}}\sum_{i,k,i',k'=1}^{K}\sum_{j,l,j',l'=1}^{N}\sum_{m,n,m',n'\ge 1}
\int_{0}^{\Delta}\int_{0}^{\Delta}\int_{0}^{\Delta}\int_{0}^{\Delta}
\phi^{*n}(s)\phi^{*m}(s)\phi^{*n'}(s')\phi^{*m'}(s') \\
&\hskip2cm \times A_{N}^{n}(i,j)A_{N}^{m}(k,l)A_{N}^{n'}(i',j')A_{N}^{m'}(k',l')\\
&\hskip2cm \times\mathbb{C}ov_{\theta}[(M^{j,N}_{(a\Delta-s)}-M^{j,N}_{a\Delta}-M^{j,N}_{((a-1)\Delta-s)}+M^{j,N}_{(a-1)\Delta})\\
&\hskip4cm \times (M^{j',N}_{(a\Delta-s')}-M^{j',N}_{a\Delta}-M^{j',N}_{((a-1)\Delta-s')}+M^{j',N}_{(a-1)\Delta}),\\
&\hskip3.3cm  (M^{l,N}_{(b\Delta-r)}-M^{l,N}_{b\Delta}-M^{l,N}_{((b-1)\Delta-r)}+M^{l,N}_{(b-1)\Delta})\\
&\hskip4cm \times (M^{l',N}_{(b\Delta-r')}-M^{l',N}_{b\Delta}+(M^{l',N}_{((b-1)\Delta-r')}-M^{l',N}_{(b-1)\Delta})]dsdrds'dr'.
\end{align*}
Define $\zeta^{j,N}_{a\Delta,s}:=M^{j,N}_{(a\Delta-s)}-M^{j,N}_{a\Delta}$ for $0\le s\le \Delta.$ Then we  rewrite 
\begin{align*}
&\mathbb{C}ov_{\theta}[(M^{j,N}_{(a\Delta-s)}-M^{j,N}_{a\Delta}-M^{j,N}_{((a-1)\Delta-s)}+M^{j,N}_{(a-1)\Delta})\\
&\hskip4cm \times (M^{j',N}_{(a\Delta-s')}-M^{j',N}_{a\Delta}-M^{j',N}_{((a-1)\Delta-s')}+M^{j',N}_{(a-1)\Delta}),\\
&\hskip1cm (M^{l,N}_{(b\Delta-r)}-M^{l,N}_{b\Delta}-M^{l,N}_{((b-1)\Delta-r)}+M^{l,N}_{(b-1)\Delta})\\
&\hskip4cm \times (M^{l',N}_{(b\Delta-r')}-M^{l',N}_{b\Delta}+(M^{l',N}_{((b-1)\Delta-r')}-M^{l',N}_{(b-1)\Delta})]\\
=&\mathbb{C}ov_{\theta}[(\zeta^{j,N}_{a\Delta,s}-\zeta^{j,N}_{(a-1)\Delta,s})(\zeta^{j',N}_{a\Delta,s'}-\zeta^{j',N}_{(a-1)\Delta,s'}),
(\zeta^{l,N}_{b\Delta,r}-\zeta^{l,N}_{(b-1)\Delta,r})(\zeta^{l',N}_{b\Delta,r'}-\zeta^{l',N}_{(b-1)\Delta,r'})].
\end{align*}
Since $0\le s,s',r,r'\le \Delta,$ we have
$$
\Et[\zeta^{j,N}_{(a-1)\Delta,s}\zeta^{j',N}_{a\Delta,s'}]=\Et[\zeta^{j,N}_{a\Delta,s}\zeta^{j',N}_{(a-1)\Delta,s'}]=\Et[\zeta^{l,N}_{(b-1)\Delta,r}\zeta^{l',N}_{b\Delta,r'}]=\Et[\zeta^{l,N}_{b\Delta,r}\zeta^{l',N}_{(b-1)\Delta,r'}]=0.
$$ 
Without loss of generality, assume $a-b\ge 4$ and $s\le s'$. First, note that
\begin{align*}
    &\mathbb{C}ov_{\theta}[\zeta^{j,N}_{a\Delta,s}\zeta^{j',N}_{(a-1)\Delta,s'},(\zeta^{l,N}_{b\Delta,r}-\zeta^{l,N}_{(b-1)\Delta,r})(\zeta^{l',N}_{b\Delta,r'}-\zeta^{l',N}_{(b-1)\Delta,r'})]\\
    =&\Et\Bigg[\Et\Big[\zeta^{j,N}_{a\Delta,s}|\mathcal{F}_{(a-1)\Delta}\Big]\zeta^{j',N}_{(a-1)\Delta,s'}\\
    \times&\Big((\zeta^{l,N}_{b\Delta,r}-\zeta^{l,N}_{(b-1)\Delta,r})(\zeta^{l',N}_{b\Delta,r'}-\zeta^{l',N}_{(b-1)\Delta,r'})-\Et\Big[(\zeta^{l,N}_{b\Delta,r}-\zeta^{l,N}_{(b-1)\Delta,r})(\zeta^{l',N}_{b\Delta,r'}-\zeta^{l',N}_{(b-1)\Delta,r'})\Big]\Big)\Bigg]=0.
\end{align*}
For the same reason, we also obtain
$$
\mathbb{C}ov_{\theta}[\zeta^{j,N}_{(a-1)\Delta,s}\zeta^{j',N}_{a\Delta,s'},(\zeta^{l,N}_{b\Delta,r}-\zeta^{l,N}_{(b-1)\Delta,r})(\zeta^{l',N}_{b\Delta,r'}-\zeta^{l',N}_{(b-1)\Delta,r'})]=0.
$$
If $j\ne j'$, the covariance vanishes because 
$$\Et[(\zeta^{j,N}_{a\Delta,s}-\zeta^{j,N}_{(a-1)\Delta,s})(\zeta^{j',N}_{a\Delta,s'}-\zeta^{j',N}_{(a-1)\Delta,s'})|\mathcal{F}_{b\Delta}]=0.$$ 
Now assume $j=j',$ then
\begin{align*}
\mathcal{K}:=& \mathbb{C}ov_{\theta}[(\zeta^{j,N}_{a\Delta,s}-\zeta^{j,N}_{(a-1)\Delta,s})(\zeta^{j',N}_{a\Delta,s'}-\zeta^{j',N}_{(a-1)\Delta,s'}),
(\zeta^{l,N}_{b\Delta,r}-\zeta^{l,N}_{(b-1)\Delta,r})(\zeta^{l',N}_{b\Delta,r'}-\zeta^{l',N}_{(b-1)\Delta,r'})]\\
=& \mathbb{C}ov_{\theta}[(\zeta^{j,N}_{a\Delta,s}\zeta^{j',N}_{a\Delta,s'}+\zeta^{j,N}_{(a-1)\Delta,s}\zeta^{j',N}_{(a-1)\Delta,s'}),
(\zeta^{l,N}_{b\Delta,r}-\zeta^{l,N}_{(b-1)\Delta,r})(\zeta^{l',N}_{b\Delta,r'}-\zeta^{l',N}_{(b-1)\Delta,r'})].
\end{align*}
Since $\Et[\zeta^{j,N}_{a\Delta,s}\zeta^{j,N}_{a\Delta,s'}|\mathcal{F}_{a\Delta-s'}]=\Et[(M^{j,N}_{a\Delta})^2-(M^{j,N}_{a\Delta-s})^2|\mathcal{F}_{a\Delta-s'}]$, and  as usual 
$(M^{j,N}_{a\Delta})^2-(M^{j,N}_{a\Delta-s})^2=2\int^{a\Delta}_{a\Delta-s}M^{j,N}_{\tau-}
dM^{j,N}_\tau +Z^{j,N}_{a\Delta}-Z^{j,N}_{a\Delta-s}$.We note that  $\Et\big[\int^{a\Delta}_{a\Delta-s} M^{j,N}_{\tau-}dM^{j,N}_\tau|\mathcal{F}_{a\Delta-s}\big]=0$. Moreover, because $a-b\ge4$, we obtain \begin{align*}
&\mathbb{C}ov_{\theta}[(M^{j,N}_{a\Delta})^2-(M^{j,N}_{a\Delta-s})^2, (\zeta^{l,N}_{b\Delta,r}-\zeta^{l,N}_{(b-1)\Delta,r})(\zeta^{l',N}_{b\Delta,r'}-\zeta^{l',N}_{(b-1)\Delta,r'})]\\
&=\mathbb{C}ov_{\theta}[Z^{j,N}_{a\Delta}-Z^{j,N}_{a\Delta-s}, (\zeta^{l,N}_{b\Delta,r}-\zeta^{l,N}_{(b-1)\Delta,r})(\zeta^{l',N}_{b\Delta,r'}-\zeta^{l',N}_{(b-1)\Delta,r'})].
\end{align*}
Consequently,
\begin{align*}
    \mathcal{K}=&\mathbb{C}ov_{\theta}[Z^{j,N}_{a\Delta}-Z^{j,N}_{a\Delta-s}+Z^{j,N}_{(a-1)\Delta}-Z^{j,N}_{(a-1)\Delta-s},
(\zeta^{l,N}_{b\Delta,r}-\zeta^{l,N}_{(b-1)\Delta,r})(\zeta^{l',N}_{b\Delta,r'}-\zeta^{l',N}_{(b-1)\Delta,r'})]\\
=&\mathbb{C}ov_{\theta}[U^{j,N}_{a\Delta}-U^{j,N}_{a\Delta-s}+U^{j,N}_{(a-1)\Delta}-U^{j,N}_{(a-1)\Delta-s},
(\zeta^{l,N}_{b\Delta,r}-\zeta^{l,N}_{(b-1)\Delta,r})(\zeta^{l',N}_{b\Delta,r'}-\zeta^{l',N}_{(b-1)\Delta,r'})].
\end{align*}
Recall that $\beta_n(x,z,r)=\phi^{\star n}(z-r)-\phi^{\star n}(x-r)$. We can write 
$$
U^{i,N}_{a\Delta}-U^{i,N}_{a\Delta-s}=\sum_{n\geq 0} \int_0^{a\Delta} \beta_n(a\Delta-s,a\Delta,r) \sum_{j=1}^N A_N^n(i,j) M^{j,N}_rdr=R^{i,N}_{a\Delta,a\Delta-s}+T^{i,N}_{a\Delta,a\Delta-s},
$$
where
\begin{align*}
R^{i,N}_{a\Delta,a\Delta-s}=&  \sum_{n\geq 0} \int_{(a-1)\Delta-s}^{a\Delta} \beta_n(x,z,r) \sum_{j=1}^N A_N^n(i,j) 
(M^{j,N}_r  -M^{j,N}_{(a-1)\Delta-s})dr,\\
T^{i,N}_{a\Delta,a\Delta-s}=&  \sum_{n\geq 0} \Big(\int_{(a-1)\Delta-s}^{a\Delta} \beta_n(x,z,r) dr \Big)\sum_{j=1}^N A_N^n(i,j) 
M^{j,N}_{(a-1)\Delta-s} \\ 
&+ \sum_{n\geq 0} \int_0^{(a-1)\Delta-s}\beta_n(x,z,r) \sum_{j=1}^N A_N^n(i,j) M^{j,N}_rdr.
\end{align*}
The conditional expectation of $R^{i,N}_{a\Delta,a\Delta-s}$ given $\mathcal{F}_{b\Delta}$ is zero.
Therefore,
\begin{align*}
    \mathcal{K}=\mathbb{C}ov_{\theta}[T^{i,N}_{a\Delta,a\Delta-s}+T^{i,N}_{(a-1)\Delta,(a-1)\Delta-s},
(\zeta^{l,N}_{b\Delta,r}-\zeta^{l,N}_{(b-1)\Delta,r})(\zeta^{l',N}_{b\Delta,r'}-\zeta^{l',N}_{(b-1)\Delta,r'})].
\end{align*}
Referring to the proof of Lemma 30, Step 1 in \cite{A} (noting that $T^{i,N}_{a\Delta,a\Delta-s}$ 
coincides with  $X^{i,N}_{a\Delta-s,a\Delta}$ in \cite{A}), we have 
$\sup_{i=1,\dots,N} \Et[(T^{i,N}_{a\Delta,a\Delta-s})^4] \leq C t^2 \Delta^{-4q}.$ Since $r\le \Delta$ and by \cite[Lemma 16-(iii)]{A}, we obtain
\begin{align*}
    \Et[(\zeta^{l,N}_{b\Delta,r}-\zeta^{l,N}_{(b-1)\Delta,r})^4]^\frac{1}{4}&\le \Et[(M^{l,N}_{(b\Delta-r)}-M^{l,N}_{b\Delta})^4]^\frac{1}{4}+\Et[(M^{l,N}_{((b-1)\Delta-r)}-M^{l,N}_{(b-1)\Delta})^4]^\frac{1}{4}\\
    &\le C\sqrt{\Delta},
\end{align*}
Hence, applying \cite[Lemma 16-(iii)]{A} once more,
\begin{align*}
|\mathcal{K}|\le&\{\Et[(T^{i,N}_{a\Delta,a\Delta-s})^2]^\frac{1}{2}+\Et[(T^{i,N}_{(a-1)\Delta,(a-1)\Delta-s})^2]^\frac{1}{2}\}\\
&\qquad \times 
\Et[(\zeta^{l,N}_{b\Delta,r}-\zeta^{l,N}_{(b-1)\Delta,r})^4]^\frac{1}{4}\Et[(\zeta^{l',N}_{b\Delta,r'}-\zeta^{l',N}_{(b-1)\Delta,r'})^4]^\frac{1}{4}\\
\le&  C t^{1/2} \Delta^{-q} \Delta.
\end{align*}
Moreover, by symmetry, we conclude that for $|a-b|\ge 4$,
\begin{align*}
&\mathbb{C}ov_{\theta}[(\zeta^{j,N}_{a\Delta,s}-\zeta^{j,N}_{(a-1)\Delta,s})(\zeta^{j',N}_{a\Delta,s'}-\zeta^{j',N}_{(a-1)\Delta,s'}),(\zeta^{l,N}_{b\Delta,r}-\zeta^{l,N}_{(b-1)\Delta,r})(\zeta^{l',N}_{b\Delta,r'}-\zeta^{l',N}_{(b-1)\Delta,r'})]\\
\le& C\Big(\boldsymbol{1}_{\{l=l'\}}+\boldsymbol{1}_{\{j=j'\}}\Big)\sqrt{t}\Delta^{1-q}.
\end{align*}
Recalling the definition of $\Omega_{N,K}$, we have $$|||I_KA^n_{N}|||_1\le |||I_KA_{N}|||_1|||A_{N}|||^{n-1}_1\le \frac{CK}{N}|||A_{N}|||^{n-1}_1,$$ which implies
\begin{align*}
   &\mathbb{C}ov_{\theta}[(C_{a\Delta}^{N,K}-C_{(a-1)\Delta}^{N,K})^{2},(C_{b\Delta}^{N,K}-C_{(b-1)\Delta}^{N,K})^{2}]\\
   \le& \frac{C\sqrt{t}\Delta^{1-q}}{K^{4}}\sum_{i,k,i',k'=1}^{K}\sum_{j,l,j',l'=1}^{N}\sum_{m,n,m',n'\ge 1}\Lambda^{n+m+n'+m'}
A_{N}^{n}(i,j)A_{N}^{m}(k,l)A_{N}^{n'}(i',j')A_{N}^{m'}(k',l')\\
&\hskip11cm\times\Big(\boldsymbol{1}_{\{l=l'\}}+\boldsymbol{1}_{\{j=j'\}}\Big)\\
\le& \frac{C\sqrt{t}\Delta^{1-q}}{K^{4}}N^3\sum_{m,n,m',n'\ge 1}
\Lambda^{n+m+n'+m'}
|||I_KA^n_{N}|||_1|||I_KA^m_{N}|||_1|||I_KA^{n'}_{N}|||_1|||I_KA^{m'}_{N}|||_1
\\
   \le& \sum_{n,m,n',m'\ge 1}\frac{1}{K^{4}}N^{3}\Big(\frac{K}{N}\Big)^{4} \Lambda^4 (\Lambda|||A_{N}|||_1)^{n+m+n'+m'-4}
    \frac{C\sqrt{t}}{\Delta^{q-1}}
    \le \frac{C\sqrt{t}}{N\Delta^{q-1}}.
\end{align*}
This completes the proof.

\subsection{Proof of Lemma \ref{mathY} (i)\&(ii) }
Recall  $c^K_{N}(j)=\sum_{i=1}^{K}Q_{N}(i,j)$ with $Q_{N}=\sum_{n\ge0}\Lambda^n A_N^n$. Remind $\bar{X}_{(a-1)\Delta,a\Delta}^{N,K}= \frac{1}{K} \sum_{i=1}^K X_{(a-1)\Delta,a\Delta}^{i,N}$, $\mathcal{Y}_{(a-1)\Delta,a\Delta}^{N,K}$ and $X_{(a-1)\Delta,a\Delta}^{i,N}$ defined in \eqref{Ya1} and \eqref{defXa}, respectively. 
We first rewrite 
\begin{align*}
\mathcal{Y}_{(a-1)\Delta,a\Delta}^{N,K}=\frac{1}{K}\sum_{j=1}^{N}\sum_{i=1}^{K}\sum_{n\ge0}\Lambda^n A_N^n(i,j)(M^{j,N}_{a\Delta}-M^{j,N}_{(a-1)\Delta}).
\end{align*}
Then, using \eqref{ee3} and Lemma \ref{phi}, we have, on $\Omega_{N,K}$
\begin{align*}
&\boldsymbol{1}_{\Omega_{N,K}}
\mathbb{E}_{\theta}[(\mathcal{Y}_{(a-1)\Delta,a\Delta}^{N,K}-\bar{X}_{(a-1)\Delta,a\Delta}^{N,K})^{2}]\\
\le& 
\frac{2}{K^2}\Et\Big[\Big|\sum_{j=1}^{N}\Big\{\sum_{n\ge 0}\sum_{i=1}^{K}\Big(\int_{0}^{a\Delta}\phi^{*n}(a\Delta-s)ds-\Lambda^{n}\Big) A_{N}^{n}(i,j)\Big\}M^{j,N}_{a\Delta}\Big|^2\Big]\\
&+\frac{2}{K^2}\Et\Big[\Big|\sum_{j=1}^{N}\Big\{\sum_{n\ge 0}\sum_{i=1}^{K}\Big(\int_{0}^{(a-1)\Delta}\phi^{*n}((a-1)\Delta-s)ds-\Lambda^{n}\Big) A_{N}^{n}(i,j)\Big\}M^{j,N}_{(a-1)\Delta}\Big|^2\Big]\\
=& 
\frac{2}{K^2}\sum_{j=1}^{N}\Big\{\sum_{n\ge 0}\sum_{i=1}^{K}\Big(\int_{0}^{a\Delta}\phi^{*n}(a\Delta-s)ds-\Lambda^{n}\Big) A_{N}^{n}(i,j)\Big\}^2\Et[Z^{j,N}_{a\Delta}]\\
&+\frac{2}{K^2}\sum_{j=1}^{N}\Big\{\sum_{n\ge 0}\sum_{i=1}^{K}\Big(\int_{0}^{(a-1)\Delta}\phi^{*n}((a-1)\Delta-s)ds-\Lambda^{n}\Big) A_{N}^{n}(i,j)\Big\}^2\Et[Z^{j,N}_{(a-1)\Delta}]\\
\le& 
\frac{2}{K^2}\sum_{j=1}^{N}\Big\{\sum_{n\ge 1}n\int_{(a\Delta)/n}^{\infty}\phi(s)ds\Lambda^{n-1}|||I_{K}A^{n}_{N}|||_1\Big)\Big\}^2\Et[Z^{j,N}_{a\Delta}]\\
&+\frac{2}{K^2}\sum_{j=1}^{N}\Big\{\sum_{n\ge 1}n\int_{(a-1)\Delta/n}^{\infty}\phi(s)ds\Lambda^{n-1}|||I_{K}A^{n}_{N}|||_1\Big\}^2\Et[Z^{j,N}_{(a-1)\Delta}].
\end{align*}
 Noting that on $\Omega_{N,K}$, $|||I_KA^n_{N}|||_1\le |||I_KA_{N}|||_1|||A_{N}|||^{n-1}_1\le \frac{CK}{N}|||A_{N}|||^{n-1}_1$, we have
\begin{align*}
&\boldsymbol{1}_{\Omega_{N,K}}
\mathbb{E}_{\theta}[(\mathcal{Y}_{(a-1)\Delta,a\Delta}^{N,K}-\bar{X}_{(a-1)\Delta,a\Delta}^{N,K})^{2}]\\
&\le \frac{C}{N^2}\sum_{j=1}^{N}\Big\{\sum_{n\ge 1}n\int_{(a\Delta)/n}^{\infty}\phi(s)ds\Lambda^{n-1}|||A_{N}|||^{n-1}_1\Big)\Big\}^2\Et[Z^{j,N}_{a\Delta}]\\
&+\frac{C}{N^2}\sum_{j=1}^{N}\Big\{\sum_{n\ge 1}n\int_{(a-1)\Delta/n}^{\infty}\phi(s)ds\Lambda^{n-1}|||A_{N}|||^{n-1}_1\Big\}^2\Et[Z^{j,N}_{(a-1)\Delta}]\\
\le& \frac{C}{N^2}\sum_{j=1}^{N}\Big\{\sum_{n\ge 1}n^{1+q}(a\Delta)^{-q}\int_{0}^{\infty}s^q\phi(s)ds\Lambda^{n-1}|||A_{N}|||^{n-1}_1\Big)\Big\}^2\Et[Z^{j,N}_{a\Delta}]\\
&+\frac{C}{N^2}\sum_{j=1}^{N}\Big\{\sum_{n\ge 1}n^{1+q}[(a-1)\Delta]^{-q}\int_{0}^{\infty}s^q\phi(s)ds\Lambda^{n-1}|||A_{N}|||^{n-1}_1\Big\}^2\Et[Z^{j,N}_{(a-1)\Delta}]\\
\le& \frac{C}{N^2(a\Delta)^{2q}}\mathbb{E}_{\theta}\Big[\sum_{j=1}^{N}Z^{j,N}_{a\Delta}\Big]+\frac{C}{N^2((a-1)\Delta)^{2q}}\mathbb{E}_{\theta}\Big[\sum_{j=1}^{N}Z^{j,N}_{(a-1)\Delta}\Big]\\
    \le& \frac{C}{N}\Big[\frac{1}{(a\Delta)^{2q-1}}+\frac{1}{((a-1)\Delta)^{2q-1}}\Big].
\end{align*}
The last inequality uses $\max_{j=1,...,N}\Et[Z^{j,N}_t]\le Ct$, proved in Lemma \ref{Zt}-(i) with  $r=\infty$, which yields exactly (i).
For $q\ge 2$, the series $\sum_{a=1}^{\infty}a^{1-2q}<+\infty,$ which concludes (ii).

\subsection{Proof of Lemma \ref{mathY} (iii)\&(iv) }
Recalling the definition of $\mathcal{Y}_{(a-1)\Delta,a\Delta}^{N,K}$ in \eqref{Ya1} and (\ref{ee3}),
and applying the Burkholder-Davis-Gundy inequality, we have on $\Omega_{N,K}$
\begin{align*}
&\mathbb{E}\Big[\Big|\mathcal{Y}_{(a-1)\Delta,a\Delta}^{N,K}\Big|^{4}\Big]\\
    \le& \frac{4}{K^4}\mathbb{E}\Bigg[\E_\theta\Big[\sum_{j=1}^N\Big(c_N^K(j)\Big)^2(Z^{j,N}_{a\Delta}-Z^{j,N}_{(a-1)\Delta})\Big]^2\Bigg]\\
    =&\frac{4}{K^4}\mathbb{E}\Bigg[\sum_{j,j'=1}^N\Big(c_N^K(j)\Big)^2\Big(c_N^K(j')\Big)^2\E_\theta\Big[(Z^{j,N}_{a\Delta}-Z^{j,N}_{(a-1)\Delta})(Z^{j',N}_{a\Delta}-Z^{j',N}_{(a-1)\Delta})\Big]\Bigg]\\
    \le& \frac{4}{K^4}\mathbb{E}\Bigg[\sum_{j,j'=1}^N\Big(c_N^K(j)\Big)^2\Big(c_N^K(j')\Big)^2\E_\theta\Big[(Z^{j,N}_{a\Delta}-Z^{j,N}_{(a-1)\Delta})^2\Big]^{1/2}\E_\theta\Big[(Z^{j',N}_{a\Delta}-Z^{j',N}_{(a-1)\Delta})^2\Big]^{1/2}\Bigg].
\end{align*}
From \cite [lemma 16-(iii)]{A}, we already have on $\Omega_{N,K}$,  $\sup_{i=1,\dots,N}\Et[(Z^{i,N}_t-Z^{i,N}_s)^2] \leq C (t-s)^2$. 
Moreover, from \eqref{CNK}, we have, on $\Omega_{N,K}$, 
\begin{align*}
\sum_{j=1}^N\big(c_{N}^{K}(j)\big)^{2}=\Big(\sum_{j=1}^K\big(c_{N}^{K}(j)\big)^{2}+\sum_{j=K}^N\big(c_{N}^{K}(j)\big)^{2}\Big)\le CK.
\end{align*}
Consequently,
\[\mathbb{E}\Big[\indiq_{\Omega_{N,K}}\Big|\mathcal{Y}_{(a-1)\Delta,a\Delta}^{N,K}\Big|^{4}\Big] \le\frac{C\Delta^2}{K^2},\]
which completes the proof of (iii).
From Lemma \ref{mathY}-(i)\&(iii) and using the Cauchy-Schwarz inequality,  we obtain
\begin{align*}
    &\frac{K}{\sqrt{\Delta t}}\mathbb{E}\Big[\boldsymbol{1}_{\Omega_{N,K}}\sum_{a=\frac{t}{\Delta}+1}^{\frac{2t}{\Delta}}\Big|\mathcal{Y}_{(a-1)\Delta,a\Delta}^{N,K}\Big|\Big|\mathcal{Y}_{(a-1)\Delta,a\Delta}^{N,K}-\bar{X}_{(a-1)\Delta,a\Delta}^{N,K}\Big|\Big]\\
    \le& \frac{K}{\sqrt{\Delta t}}\sum_{a=\frac{t}{\Delta}+1}^{\frac{2t}{\Delta}}\mathbb{E}\Big[\boldsymbol{1}_{\Omega_{N,K}}\Big|\mathcal{Y}_{(a-1)\Delta,a\Delta}^{N,K}\Big|^{4}\Big]^\frac{1}{4}\mathbb{E}\Big[\boldsymbol{1}_{\Omega_{N,K}}\Big|\mathcal{Y}_{(a-1)\Delta,a\Delta}^{N,K}-\bar{X}_{(a-1)\Delta,a\Delta}^{N,K}\Big|^{2}\Big]^\frac{1}{2}\\
    \le& \frac{C\sqrt{K}}{\Delta^{q-\frac{1}{2}}\sqrt{Nt}}\sum_{a=\frac{t}{\Delta}}^{\frac{2t}{\Delta}}a^{\frac{1}{2}-q} \le \frac{C\sqrt{K}}{\Delta^{q-\frac{1}{2}}\sqrt{Nt}}.
\end{align*}
In the last step, we used that for  $q\ge 2$, the series $\sum_{a=1}^{\infty}a^{\frac{1}{2}-q}$ converges.

\subsection{Proof of Lemma \ref{mathY} (v) }
Recalling that $\bar{\Gamma}_{(a-1)\Delta,a\Delta}^{N,K}=C_{a\Delta}^{N,K}+B_{a\Delta}^{N,K}-C_{(a-1)\Delta}^{N,K}-B_{(a-1)\Delta}^{N,K},$ where $C_{a\Delta}^{N,K}$ and  $B_{a\Delta}^{N,K}$ are defined in \eqref{Cdelta} and \eqref{Bdelta}, respectively. We write
\begin{align*}
&\bar{\Gamma}_{(a-1)\Delta,a\Delta}^{N,K}\bar{X}_{(a-1)\Delta,a\Delta}^{N,K}\\
=&\bar{\Gamma}_{(a-1)\Delta,a\Delta}^{N,K}(\bar{X}_{(a-1)\Delta,a\Delta}^{N,K}-\mathcal{Y}_{(a-1)\Delta,a\Delta}^{N,K})+\mathcal{Y}_{(a-1)\Delta,a\Delta}^{N,K}\bar{\Gamma}_{(a-1)\Delta,a\Delta}^{N,K}\\
=&\bar{\Gamma}_{(a-1)\Delta,a\Delta}^{N,K}(\bar{X}_{(a-1)\Delta,a\Delta}^{N,K}-\mathcal{Y}_{(a-1)\Delta,a\Delta}^{N,K})+\mathcal{Y}_{(a-1)\Delta,a\Delta}^{N,K}(C_{a\Delta}^{N,K}+B_{a\Delta}^{N,K}-C_{(a-1)\Delta}^{N,K}-B_{(a-1)\Delta}^{N,K}).
\end{align*}
Applying the Cauchy–Schwarz inequality together with Lemmas \ref{mathY}-(i)  and \ref{covc}-(i)\&(ii), we obtain
\begin{align*}
    &\mathbb{E}\Big[\boldsymbol{1}_{\Omega_{N,K}}\Big|\bar{\Gamma}_{(a-1)\Delta,a\Delta}^{N,K}\Big(\bar{X}_{(a-1)\Delta,a\Delta}^{N,K}-\mathcal{Y}_{(a-1)\Delta,a\Delta}^{N,K}\Big)\Big|\Big]^{2}\\
    \le&\mathbb{E}\Big[\boldsymbol{1}_{\Omega_{N,K}}\Big|\mathcal{Y}_{(a-1)\Delta,a\Delta}^{N,K}-\bar{X}_{(a-1)\Delta,a\Delta}^{N,K}\Big|^{2}\Big]\mathbb{E}\Big[\boldsymbol{1}_{\Omega_{N,K}}(\bar{\Gamma}_{(a-1)\Delta,a\Delta}^{N,K})^{2}\Big]\\
 =&\mathbb{E}\Big[\boldsymbol{1}_{\Omega_{N,K}}\Big|\mathcal{Y}_{(a-1)\Delta,a\Delta}^{N,K}-\bar{X}_{(a-1)\Delta,a\Delta}^{N,K}\Big|^{2}\Big]\mathbb{E}\Big[\boldsymbol{1}_{\Omega_{N,K}}\Big(C_{a\Delta}^{N,K}+B_{a\Delta}^{N,K}-C_{(a-1)\Delta}^{N,K}-B_{(a-1)\Delta}^{N,K}\Big)^{2}\Big]\\
    \le& 4\mathbb{E}\Big[\boldsymbol{1}_{\Omega_{N,K}}\Big|\mathcal{Y}_{(a-1)\Delta,a\Delta}^{N,K}-\bar{X}_{(a-1)\Delta,a\Delta}^{N,K}\Big|^{2}\Big]\\
    &\times\mathbb{E}\Big[\boldsymbol{1}_{\Omega_{N,K}}\Big\{\Big(C_{a\Delta}^{N,K}\Big)^{2}+\Big(B_{a\Delta}^{N,K}\Big)^{2}+\Big(C_{(a-1)\Delta}^{N,K}\Big)^{2}+\Big(B_{(a-1)\Delta}^{N,K}\Big)^{2}\Big\}\Big]\\
    \le&  \frac{C}{N}\Big[(a\Delta)^{1-2q}+\Big((a-1)\Delta\Big)^{1-2q}\Big]\Big(\frac{1}{N}+\frac{1}{N}\Delta^{1-2q}\Big)\\
    \le& \Big[(a\Delta)^{1-2q}+\Big((a-1)\Delta\Big)^{1-2q}\Big]\frac{C}{N^2}.
\end{align*}
Similarly, using the Cauchy–Schwarz inequality  again and Lemmas \ref{covc}-(i) and \ref{mathY}-(iii), we have
\begin{align*}
    &\indiq_{\Omega_{N,K}}\mathbb{E}_{\theta}\Big[\Big|\mathcal{Y}_{(a-1)\Delta,a\Delta}^{N,K}(B_{a\Delta}^{N,K}-B_{(a-1)\Delta}^{N,K})\Big|\Big]^{2}\\
    \le& \indiq_{\Omega_{N,K}}\mathbb{E}_{\theta}\Big[\Big|\mathcal{Y}_{(a-1)\Delta,a\Delta}^{N,K}\Big|^{2}\Big]\mathbb{E}_{\theta}\Big[\Big(B_{a\Delta}^{N,K}-B_{(a-1)\Delta}^{N,K}\Big)^{2}\Big]\\
    \le& \indiq_{\Omega_{N,K}}\mathbb{E}_{\theta}\Big[\Big|\mathcal{Y}_{(a-1)\Delta,a\Delta}^{N,K}\Big|^{4}\Big]^{1/2}\mathbb{E}_{\theta}\Big[\Big(B_{a\Delta}^{N,K}-B_{(a-1)\Delta}^{N,K}\Big)^{2}\Big]
    \le  \frac{C}{NK\Delta^{2q-2}}.
\end{align*}
Next, we consider the term $\mathcal{Y}_{(a-1)\Delta,a\Delta}^{N,K}(C_{a\Delta}^{N,K}-C_{(a-1)\Delta}^{N,K}).$  Recalling $\mathcal{Y}_{(a-1)\Delta,a\Delta}^{N,K}$ defined in \eqref{Ya1} and $C_{a\Delta}^{N,K}$ defined  in \eqref{Cdelta}, we  write
\begin{align*}
    \mathcal{Y}_{(a-1)\Delta,a\Delta}^{N,K}(C_{a\Delta}^{N,K}-C_{(a-1)\Delta}^{N,K})=\frac{1}{K^{2}}\sum_{i=1}^{K}\sum_{j,j'=1}^N\sum_{n\ge 1}\int_{0}^{\Delta}\phi^{*n}(s)A_{N}^{n}(i,j)c_{N}^{K}(j')\hskip2cm\\
    (M^{j,N}_{a\Delta-s}-M^{j,N}_{a\Delta}-M^{j,N}_{(a-1)\Delta-s}+M^{j,N}_{(a-1)\Delta})(M^{j',N}_{a\Delta}-M^{j',N}_{(a-1)\Delta}).
\end{align*}
We set for $1\le j,j',l,l'\le N$ and $a,  b\in \{t/(2\Delta)+1,...,2t/\Delta\}$,
\begin{align*}
\mathcal{T}_{a,b}(j,j',l,l'):=
   &\mathbb{C}ov_{\theta}[(M^{j,N}_{a\Delta-s}-M^{j,N}_{a\Delta}-M^{j,N}_{(a-1)\Delta-s}+M^{j,N}_{(a-1)\Delta})(M^{j',N}_{a\Delta}-M^{j',N}_{(a-1)\Delta}),\\
   &\hskip1cm(M^{l,N}_{b\Delta-s}-M^{l,N}_{b\Delta}-M^{l,N}_{(b-1)\Delta-s}+M^{l,N}_{(b-1)\Delta})(M^{l',N}_{b\Delta}-M^{l',N}_{(b-1)\Delta})].
\end{align*}
Using \cite[Lemma 16-(iii)]{A}, it is obvious that without any condition on $(a,b)$, on $\Omega_{N,K}$ 
\begin{align*}
&|\mathcal{T}_{a,b}(j,j',l,l')|\\
\le& \Big\{\mathbb{E}_{\theta}\Big[\Big(M^{j,N}_{a\Delta}-M^{j,N}_{(a-1)\Delta}\Big)^4\Big]^\frac{1}{4}+\mathbb{E}_{\theta}\Big[\Big(M^{j,N}_{a\Delta-s}-M^{j,N}_{(a-1)\Delta-s}\Big)^4\Big]^\frac{1}{4}\Big\}\mathbb{E}_{\theta}\Big[\Big(M^{j',N}_{a\Delta}-M^{j',N}_{(a-1)\Delta}\Big)^4\Big]^\frac{1}{4}\\
&\Big\{\mathbb{E}_{\theta}\Big[\Big(M^{l,N}_{b\Delta}-M^{l,N}_{(b-1)\Delta}\Big)^4\Big]^\frac{1}{4}+\mathbb{E}_{\theta}\Big[\Big(M^{l,N}_{b\Delta-s}-M^{l,N}_{(b-1)\Delta-s}\Big)^4\Big]^\frac{1}{4}\Big\}\mathbb{E}_{\theta}\Big[\Big(M^{l',N}_{b\Delta}-M^{l',N}_{(b-1)\Delta}\Big)^4\Big]^\frac{1}{4}\\
\le& C\Delta^2,
\end{align*}
and $\boldsymbol{1}_{\{\#\{j,j',l,l'\}=4\}}|\mathcal{T}_{a,b}(j,j',l,l')|=0.$
\vip
We now consider the case when  $a-b\ge 4$.
Recall that $\zeta^{j,N}_{a\Delta,s}:=M^{j,N}_{(a\Delta-s)}-M^{j,N}_{a\Delta}$ for $0\le s\le \Delta$. Then,
\begin{align*}
   \mathcal{T}_{a,b}(j,j',l,l')=&\mathbb{C}ov_{\theta}[(\zeta^{j,N}_{a\Delta,s}-\zeta^{j,N}_{(a-1)\Delta,s})\zeta^{j',N}_{a\Delta,\Delta},(\zeta^{l,N}_{b\Delta,r}-\zeta^{l,N}_{(b-1)\Delta,r})\zeta^{l',N}_{b\Delta,\Delta}]\\
    =&\mathbb{C}ov_{\theta}[\zeta^{j,N}_{a\Delta,s}\zeta^{j',N}_{a\Delta,\Delta},(\zeta^{l,N}_{b\Delta,r}-\zeta^{l,N}_{(b-1)\Delta,r})\zeta^{l',N}_{b\Delta,\Delta}].
\end{align*}
Using the same strategy as the proof of Lemma \ref{covc}, we have
\begin{align*}
&|\mathbb{C}ov_{\theta}[\zeta^{j,N}_{a\Delta,s}\zeta^{j',N}_{a\Delta,\Delta},(\zeta^{l,N}_{b\Delta,r}-\zeta^{l,N}_{(b-1)\Delta,r})\zeta^{l',N}_{b\Delta,\Delta}]|\\
=&|\mathbb{C}ov_{\theta}[T^{j,N}_{a\Delta,(a-1)\Delta},
(\zeta^{l,N}_{b\Delta,r}-\zeta^{l,N}_{(b-1)\Delta,r})\zeta^{l',N}_{b\Delta,\Delta}]|\\
\le& \{\Et[(T^{j,N}_{a\Delta,(a-1)\Delta})^2]^\frac{1}{2}\}
\Et[(\zeta^{l,N}_{b\Delta,r}-\zeta^{l,N}_{(b-1)\Delta,r})^4]^\frac{1}{4}\Et[(\zeta^{l',N}_{b\Delta,\Delta})^4]^\frac{1}{4}\\
\le&  C t^{1/2} \Delta^{1-q}.
\end{align*}
Hence, by symmetry, for $|a-b|\ge 4$,  $|\mathcal{T}_{a,b}(j,j',l,l')|\le C\Big(\boldsymbol{1}_{\{l=l'\}}+\boldsymbol{1}_{\{j=j'\}}\Big)\sqrt{t}\Delta^{1-q}.$
Consequently, still for $|a-b|\ge 4,$
\begin{align*}
&\mathbb{E}\Big[\boldsymbol{1}_{\Omega_{N,K}}\Big|\mathbb{C}ov_{\theta}\Big[\cY_{(a-1)\Delta,a\Delta}^{N,K}(C_{a\Delta}^{N,K}-C_{(a-1)\Delta}^{N,K}),\cY_{(b-1)\Delta,b\Delta}^{N,K}(C_{b\Delta}^{N,K}-C_{(b-1)\Delta}^{N,K})\Big]\Big|\Big]\\
=&\frac{1}{K^{4}}\mathbb{E}\Big[\boldsymbol{1}_{\Omega_{N,K}}\Big|\sum_{i,i'=1}^{K}\sum_{l,l',j,j'=1}^{N}\sum_{n,n'\ge 1}\int_{0}^{\Delta}\int_{0}^{\Delta}\phi^{*n}(s)\phi^{*n'}(s')A_{N}^{n}(i,j)A_{N}^{n'}(i',l)\\
&\hskip7cm\times c_{N}^{K}(j')c_{N}^{K}(l')\mathcal{T}_{a,b}(j,j',l,l')dsds'\Big|\Big]\\
    \le& \frac{t^{1/2}}{K^{4}\Delta^{q-1}}\mathbb{E}\Big[\boldsymbol{1}_{\Omega_{N,K}}\Big|\sum_{j=1}^N\Big(c_{N}^{K}(j)-1\Big)c_{N}^{K}(j)\Big|\Big|\sum_{l=1}^Nc_{N}^{K}(l)\Big|\sum_{n\ge 1}N\Lambda^n|||I_KA_N|||_1|||A_N|||^{n-1}_1\Big]\\
    \le& \frac{Ct^{1/2}}{K\Delta^{q-1}}.
\end{align*}
The last step follows from \eqref{CNK}, which implies that on  $\Omega_{N,K}$, $\sum_{j=1}^N\big(c_{N}^{K}(j)\big)^{2}\le CK$, together with the facts that
on $\Omega_{N,K}$,
$|||I_KA_N|||_1\le \frac{K}{N},$ $|\sum_{l=1}^Nc_{N}^{K}(l)|=K|\bar{\ell}_N^K|\le CK.$ 

\vip

Next, when $|a-b|\le 4,$

\begin{align*}
    &\mathbb{E}\Big[\boldsymbol{1}_{\Omega_{N,K}}\Big|\mathbb{C}ov_{\theta}\Big[\cY_{(a-1)\Delta,a\Delta}^{N,K}(C_{a\Delta}^{N,K}-C_{(a-1)\Delta}^{N,K}),\cY_{(b-1)\Delta,b\Delta}^{N,K}(C_{b\Delta}^{N,K}-C_{(b-1)\Delta}^{N,K})\Big]\Big|\Big]\\
    \le& \mathbb{E}\Big[\indiq_{\Omega_{N,K}}\mathbb{V}ar_{\theta}\Big[\mathcal{Y}_{(a-1)\Delta,a\Delta}^{N,K}(C_{a\Delta}^{N,K}-C_{(a-1)\Delta}^{N,K})\Big]\Big]^{\frac{1}{2}}\mathbb{E}\Big[\indiq_{\Omega_{N,K}}\mathbb{V}ar_{\theta}\Big[\mathcal{Y}_{(b-1)\Delta,b\Delta}^{N,K}(C_{b\Delta}^{N,K}-C_{(b-1)\Delta}^{N,K})\Big]\Big]^{\frac{1}{2}}\\
    \le& \mathbb{E}\Big[\indiq_{\Omega_{N,K}}\mathbb{E}_{\theta}\Big[\Big|\mathcal{Y}_{(a-1)\Delta,a\Delta}^{N,K}\Big|^{4}\Big]^{\frac{1}{4}}\mathbb{E}_{\theta}\Big[\Big(C_{a\Delta}^{N,K}-C_{(a-1)\Delta}^{N,K}\Big)^{4}\Big]^{\frac{1}{4}}\mathbb{E}_{\theta}\Big[\Big|\mathcal{Y}_{(b-1)\Delta,b\Delta}^{N,K}\Big|^{4}\Big]^{\frac{1}{4}}\\
    &\hskip4cm\times \mathbb{E}_{\theta}\Big[\Big(C_{b\Delta}^{N,K}-C_{(b-1)\Delta}^{N,K}\Big)^{4}\Big]^{\frac{1}{4}}\Big]\\
    \le& \frac{C\Delta}{NK}.
\end{align*}
Finally, 
\begin{align*}
\mathbb{E}\Big[\indiq_{\Omega_{N,K}}\mathbb{V}ar_{\theta}\Big[\sum_{a=\frac{t}{\Delta}+1}^{\frac{2t}{\Delta}}\mathcal{Y}_{(a-1)\Delta,a\Delta}^{N,K}(C_{a\Delta}^{N,K}-C_{(a-1)\Delta}^{N,K})\Big]\Big]\le \frac{Ct}{NK}+\frac{Ct^{5/2}}{K\Delta^{q+1}}.
\end{align*}
Overall we conclude that 
\begin{align*}
&\mathbb{E}\Big[\boldsymbol{1}_{\Omega_{N,K}}\frac{K}{N}\sqrt{\frac{t}{\Delta}}\frac{N}{t}\Big|\sum_{a=\frac{t}{\Delta}+1}^{\frac{2t}{\Delta}}\bar{\Gamma}_{(a-1)\Delta,a\Delta}^{N,K}\bar{X}_{(a-1)\Delta,a\Delta}^{N,K}-\mathbb{E}_{\theta}[\sum_{a=\frac{t}{\Delta}+1}^{\frac{2t}{\Delta}}\bar{\Gamma}_{(a-1)\Delta,a\Delta}^{N,K}\bar{X}_{(a-1)\Delta,a\Delta}^{N,K}]\Big|\Big]\\
\le& 
C\frac{K}{\sqrt{\Delta t}}\Big\{\mathbb{E}\Big[\boldsymbol{1}_{\Omega_{N,K}}\sum_{a=\frac{t}{\Delta}+1}^{\frac{2t}{\Delta}}\Big(\Big|\bar{\Gamma}_{(a-1)\Delta,a\Delta}^{N,K}(\bar{X}_{(a-1)\Delta,a\Delta}^{N,K}-\cY_{(a-1)\Delta,a\Delta}^{N,K})\Big|\\
&+\Big|\mathcal{Y}_{(a-1)\Delta,a\Delta}^{N,K}(B_{a\Delta}^{N,K}-B_{(a-1)\Delta}^{N,K})\Big|\Big)\Big]
 +\mathbb{E}\Big[\boldsymbol{1}_{\Omega_{N,K}}\mathbb{V}ar_{\theta}\Big[\sum_{a=\frac{t}{\Delta}+1}^{\frac{2t}{\Delta}}\cY_{(a-1)\Delta,a\Delta}^{N,K}(C_{a\Delta}^{N,K}-C_{(a-1)\Delta}^{N,K})\Big]\Big]^{\frac{1}{2}}\Big\}\\
\le& \frac{CK}{N\Delta^q\sqrt{t}}+\frac{C\sqrt{tK}}{\Delta^{q+\frac{1}{2}}\sqrt{N}}+\frac{C\sqrt{K}}{\sqrt{N\Delta}}+\frac{Ct^\frac{3}{4}\sqrt{K}}{\Delta^{1+\frac{q}{2}}}.
\end{align*}
The proof is finished.

\section*{Acknowledgements}
We would like to express our sincere gratitude to N. Fournier and S. Delattre for their invaluable support of this research. This work would not have been possible without their insightful ideas and patient guidance.

\bibliographystyle{amsplain}  
\bibliography{CLT}      

\end{document}